\documentclass[bj,preprint]{imsart}

\RequirePackage[OT1]{fontenc}
\usepackage[bbgreekl]{mathbbol}
\RequirePackage{amsthm,amsmath, amssymb}
\RequirePackage[]{natbib}
\RequirePackage[colorlinks,citecolor=blue,urlcolor=blue]{hyperref}
\usepackage{todonotes}
\usepackage{mathrsfs}
\usepackage{bbm}
\usepackage{hyperref}
\usepackage{enumerate}
\usepackage{graphicx}
\usepackage{float}

\startlocaldefs
\numberwithin{equation}{section}
\theoremstyle{plain}
\newtheorem{thm}{Theorem}[section]
\newtheorem{lemma}[thm]{Lemma}
\newtheorem{corollary}[thm]{Corollary}

\newtheorem{proposition}[thm]{Proposition}
\newtheorem*{remark}{Remark}
\endlocaldefs

\newcommand{\ind}{1\hspace{-0,9ex}1}

\DeclareMathOperator{\dd}{d}
\DeclareMathOperator{\Var}{Var}
\DeclareMathOperator{\E}{E}

\DeclareMathOperator{\diag}{diag}

\DeclareMathOperator{\tr}{tr}

\DeclareMathOperator{\rank}{rank}

\def\d{\mathrm{d}}

\def\C{\mathbb{C}}

\def\E{\mathbb{E}}

\def\N{\mathbb{N}}

\def\R{\mathbb{R}}

\def\NN{\mathcal{N}}

\def\N{\mathbb{N}}
\def\P{\mathbb{P}}

\newenvironment{theorem}{\begin{thm}}{\end{thm}}

\allowdisplaybreaks

\begin{document}

\begin{frontmatter}

	\title{Spectral Analysis of High-Dimensional Sample Covariance Matrices with Missing Observations}
	\runtitle{Spectral analysis with missing observations}
\thankstext{T1}{Supported by the DFG Research Unit 1735, RO 3766/3-1 and DE 502/26-2.}	\begin{aug}
		\author{\fnms{Kamil} \snm{Jurczak}\ead[label=e1]{kamil.jurczak@ruhr-uni-bochum.de}}
		\and
		\author{\fnms{Angelika} \snm{Rohde}\ead[label=e2]{angelika.rohde@ruhr-uni-bochum.de}}

		\affiliation{Ruhr-Universit\"at Bochum}

		\address{Fakult\"at f\"ur Mathematik \\
		Ruhr-Universit\"at Bochum \\
		44780 Bochum \\
		Germany \\
		\printead{e1}
		\phantom{E-mail:\ }\printead*{e2}}
	\end{aug}

\begin{abstract}
We study high-dimensional sample covariance matrices based on independent random vectors with missing coordinates. The presence of missing observations is common in modern applications such as climate studies or gene expression micro-arrays.  A weak approximation on the spectral distribution in the "large dimension $d$ and large sample size $n$" asymptotics is derived for possibly different observation probabilities in the coordinates. The spectral distribution turns out to be strongly influenced by the missingness mechanism. In the null case under the missing at random scenario where each component is observed with the same probability $p$, the limiting spectral distribution is a Mar\v{c}enko-Pastur law shifted by $(1-p)/p$ to the left. As $d/n\rightarrow y\in(0,1)$, the almost sure convergence of the extremal eigenvalues to the respective boundary points of the support of the limiting spectral distribution is proved, which are explicitly given in terms of $y$ and $p$. Eventually, the sample covariance matrix is positive definite if $p$ is larger than
$$
1-\left(1-\sqrt{y}\right)^2,
$$ 
whereas this is not true any longer if $p$ is smaller than this quantity. 
\end{abstract}
\end{frontmatter}
\section{Introduction}
In many modern applications 
high-dimensional data suffers from missing observations. As pointed out in \cite{Olga2001}, ``The data from microarray experiments is usually in
the form of large matrices of expression levels of genes
(rows) under different experimental conditions (columns)
and frequently with some values missing. Missing values
occur for diverse reasons, including insufficient resolution,
image corruption, or simply due to dust or scratches on
the slide. Missing data may also occur systematically
as a result of the robotic methods used to create them.'' ``Data available for climate research typically suffer from
uneven sampling due to ... sporadic instrument failure; or other interruptions during the period of interest,'' \cite{Sherwood2001}. Further, missing observations in telescope data may be caused by a cloudy sky, \cite{Nishizawa2013}.\\  
In the statistical literature, high-dimensional low-rank covariance matrix estimation with missing observations has been recently investigated in \cite{Lounici2014}, where sparsity oracle inequalities for a matrix-Lasso estimator are derived. An adaptive test for large covariance matrices with missing observations have been proposed recently in \cite{Butucea2016}. While in view of inference statements asymptotic properties of the eigenvalues and eigenvectors for high-dimensional sample covariance matrices based on complete data are exhaustively investigated in random matrix theory, the statistically equally important case of missing observations has not been studied so far. Concerning spectral based dimension reduction techniques and statistics such as the log-determinant, a profound spectral analysis is inevitable. The aim of this article is to get this development underway.  
We study asymptotic spectral properties of high-dimensional sample covariance matrices with missing observations. 
Let $$Y=(Y_1,...,Y_n)\in\R^{d\times n},~~Y_k=(Y_{1k},...,Y_{dk})^\ast \in\R^d,~~k=1,...,n,$$ be a sample of independent identically distributed (iid) random vectors with covariance matrix 
$$
T = \E\big((Y_1-\E Y_1)\otimes(Y_1-\E Y_1)\big).
$$ 
In examples as described above, we do not observe the whole random vector $Y_k$ but some of its components. This missingness is represented by a random matrix $\varepsilon\in \R^{d\times n}$ with entries
$$
\varepsilon_{ik} = \begin{cases}
1 & \text{if $Y_{ik}$ is observed}\\
0 & \text{if $Y_{ik}$ is missing}.
\end{cases}
$$
Under the assumption that the matrices $Y$ and $\varepsilon$ are independent, the estimator 
$$
\hat{T}_{ij} = \frac{1}{N_{ij}}\sum_{k\in\NN_{ij}}\left(Y_{ik}-\bar{Y}_i\right)\left(Y_{jk}-\bar{Y}_j\right)
$$
is the analogue of the sample covariance and hence the natural estimator for $T_{ij}$, where
\begin{align}\label{eq: set and number}
\NN_{ij}   =      \Big\{ k\in\{1,\dots ,n \} :\ \varepsilon_{ik}\varepsilon_{jk}=1\Big\}, \ \  N_{ij} = 1\vee\# \NN_{ij} 
\end{align}
and
$$
 \bar{Y}_i = \frac{1}{N_{ii}}\sum_{k\in\NN_{ii}}Y_{ik}.
$$
Subsequently, $\hat T=(\hat T_{ij})\in\R^{d\times d}$ is referred to as sample covariance matrix with missing observations. If $\E Y_k=0$ is known in advance one typically uses the estimator
$$\hat\Sigma=\big( \hat\Sigma_{ij} \big)\in\R^{d\times d},\ \ \hat\Sigma_{ij}=\frac{1}{N_{ij}}\sum_{k\in \NN_{ij}}Y_{ik}Y_{jk}.$$
In what follows we write $\hat \Xi$ for $\hat T$ and $\hat\Sigma$ if a statement holds for both estimators.
The distribution of the missingness matrix $\varepsilon$ substantially influences the spectrum of $\hat \Xi$ (see Figure 1). In the high-dimensional scenario, $\hat \Xi$ may be asymptotically indefinite even if the smallest eigenvalue of $T$ stays uniformly bounded away from zero. Heuristically, it is not clear at all how the high dimensionality affects the spectral properties in the situation of missing observations, and whether well-known phenomena occur in a possibly modified way. In this article we investigate asymptotic spectral properties of $\hat \Xi$ under the classical missing (completely) at random (MAR) setting. 
\begin{figure}[h] 
     \includegraphics[width=0.92\textwidth]{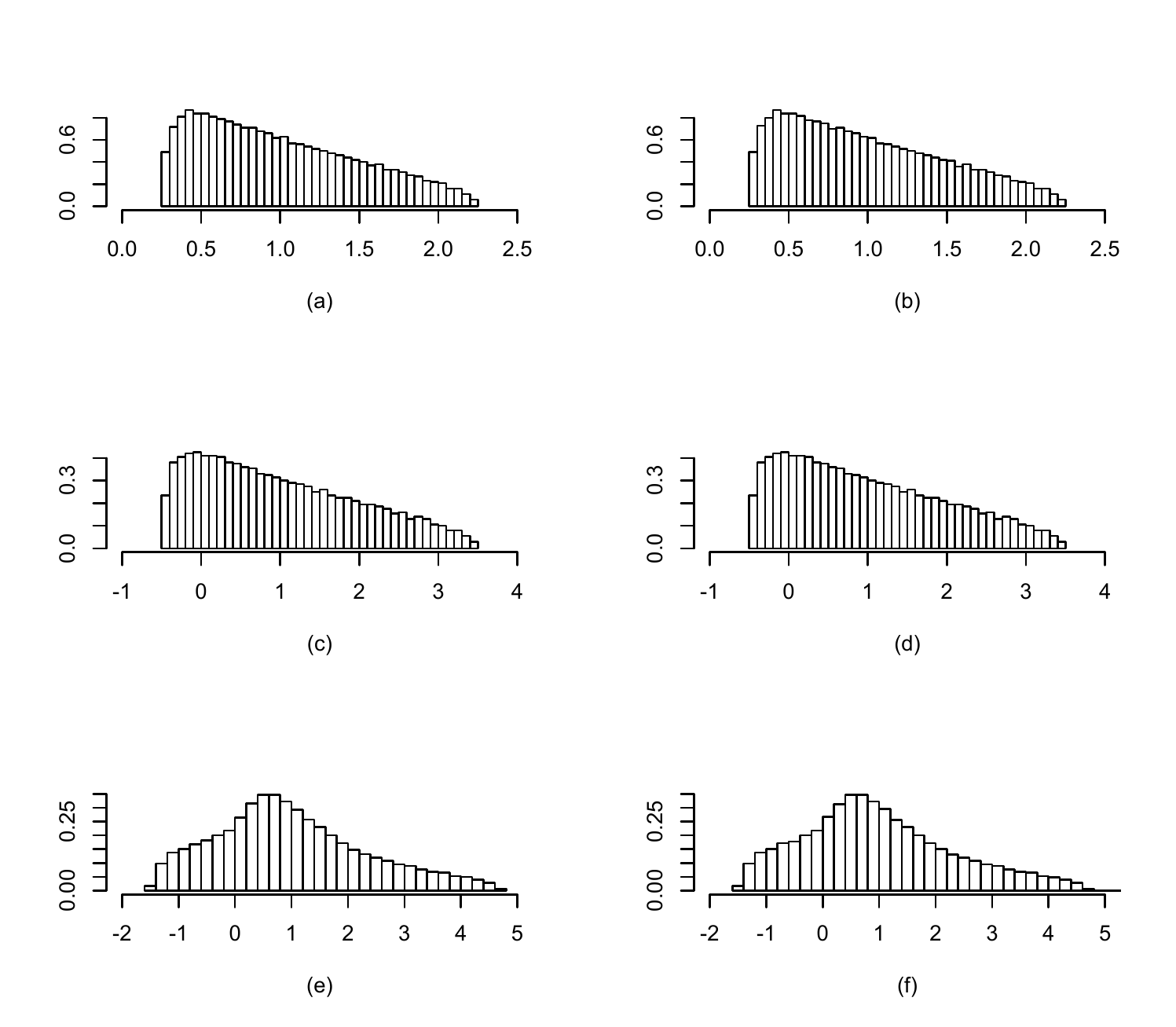}
  \caption{The left column shows histograms of the eigenvalues of the estimator $\hat\Sigma$ and the right column of the estimator $\hat T$ from a centered Gaussian sample. The underlying population covariance matrix in each histogram is the identity. The dimension of the observations in the first row is 2000, the sample size 8000 and all coordinates are observed. In the second row each coordinate is observed with probability 1/2. In the last row the probabilities of observation are changed to 1/4 for the first 1000 coordinates and to 3/4 for the other half of the coordinates.}
  \label{fig:Spektralverteilung2}
  \end{figure}
Here, the variables $\varepsilon_{ik}$, $i=1,...,d$, $k=1,...,n$, are independent random Bernoulli variables with
 $$\P(\varepsilon_{ik}=1)=p_i \ \ \text{ and }\ \ \P( \varepsilon_{ik}=0)=1-p_i,$$ 
and they are jointly independent of $Y_1,\dots,Y_n$. The latter are assumed to be of the form
$$
Y_k\ =\ T^{1/2}X_k+\E Y_k,~~k=1,\dots, n,
$$
where $X_1,\dots,X_n$ are iid centered random vectors with independent coordinates of variance $1$. This representation is common in literature on random matrix theory. Without missing observations, that is, for completely observed random vectors $Y_1,\dots Y_n$, the classical sample covariance matrix is a well-studied object in the large dimension $d$ and large sample size $n$ asymptotics. The first result on its spectral distribution is due to \cite{Marcenko1967}. They  established in particular weak convergence in probability of the empirical spectral distribution for diagonal $T$ under the assumption of finite fourth moment on the entries of $X_1,\dots, X_n$ and some dependency condition reflected in their mixed second and fourth moments. The most general version of this statement has been proved in \cite{Silverstein1995}, where weak convergence (almost surely) is established under the finite second moment assumption for rather general matrices $T$. The almost sure convergence of the largest eigenvalue in the null case $T=I_{d\times d}$ (identity matrix) has been proved in \cite{Bai1988} under the assumption of the existence of the fourth moment, which generalizes a first result in this direction due to \cite{Geman1980}. \cite{Bai1988b} have shown that the existence of the fourth moment is in fact necessary. As concerns the smallest eigenvalue in the null case, the most current theorem on its almost sure convergence has been derived by \cite{Bai1993}. Under quite general regularity conditions on $T$, the convergence of the extremal eigenvalues to the respective boundaries of the support of the limiting spectral distribution follows from \cite{Bai1998}.\\
Our contributions in this article are the following. 
\begin{itemize}
\item[(i)] We establish a weak approximation of the empirical spectral distribution of the sample covariance matrix with missing observations $\hat\Xi$ by a non-random sequence of probability measures expressed in terms of their Stieltjes transforms, which holds true for possibly different observation probabilities in the coordinates. In the null case under the missing at random scenario where each component is observed with the same probability $p$, the limiting spectral distribution is shown to be a Mar\v{c}enko-Pastur law shifted by $(1-p)/p$ to the left.  
\item[(ii)] As $d/n\rightarrow y\in(0,1)$ and under the missing at random scenario where each component is observed with the same probability, we prove almost sure convergence of the extremal eigenvalues of $\hat\Sigma$ to the respective boundary points of the support of the limiting spectral distribution in the null case. A statistically important consequence is the characterization of positive definiteness for the sample covariance matrix with missing observations. 
\end{itemize}
Understanding the empirical spectral distribution of sample covariance matrices with missing observations is of great importance to develop improved estimators for the population covariance matrix and the precision matrix. Such estimators have been already established for completely observed data by \cite{Karoui2008} and \cite{Ledoit2012}  based on non-linear shrinkage of the eigenvalues. However, if some data is missing, the situation is more intricate since the analysis in our article reveals that the limiting behavior of the empirical spectral distribution does not only depend on the eigenvalues of the population covariance matrix but also on its eigenvectors. Nevertheless, we expect that adjusting the diagonal of the sample covariance matrix with missing observations yields a more suitable matrix for spectrum estimation. 

\noindent
Very recently, various authors studied asymptotic spectral properties of sample autocovariance matrices of high-dimensional time series which is another statistically relevant scenario. \cite{Jin2014} derived the limiting spectral distribution of the symmetrized autocovariance matrix  in the iid case. \cite{Liu2015} established a Mar\v{c}enko-Pastur-type law for the empirical spectral distribution in case of general high-dimensional linear time series. They investigated  the moderately high-dimensional case of this problem in \cite{Wang2015b}. \cite{Li2015} developed the limiting singular value distribution of the sample autocovariance matrix by means of the Stieltjes transform for an independent sequence with elements possessing finite fourth moments. \cite{Wang2015c} proved the same result by the method of moments, and additionally the almost sure convergence of the spectral norm. The strong limit of the extreme eigenvalues of symmetrized autocovariance matrices is established in \cite{Wang2015}.\\
The article is organized as follows. First we introduce the essential notation and the model assumptions in the next section. Section \ref{section: results} is devoted to our main results. The proof of Theorem \ref{theorem: main} is quite long and therefore decomposed into Section \ref{section: part 1}, Section \ref{section: part 2} and Appendix \ref{A}. The proof of Theorem \ref{theorem: extreme} is deferred to Section \ref{section: thm 2} and Appendix \ref{B}. Some auxiliary results which are used throughout the proofs are collected in Appendix \ref{section: appendix}.

\section{Notation and preliminaries}
\subsection{Notation}
For any bounded function $f:\R\rightarrow\R$
$$\Arrowvert f \Arrowvert=\sup_{x\in\R}\arrowvert f(x)\arrowvert$$
denotes its supremum norm. If $f$ is Lipschitz in addition then the bounded Lipschitz norm is defined as
$$\Arrowvert f \Arrowvert_{BL}=\max\left(\Arrowvert f \Arrowvert_L,\Arrowvert f \Arrowvert\right),$$
where $\Arrowvert f \Arrowvert_L$ denotes is the best Lipschitz constant of $f$.
We write 
$$
\C^+\ =\ \left\{z\in\C:\, \Im(z)>0\right\}
$$
for the upper complex half plane.
For any Hermitian matrix $A\in\C^{d\times d}$ denote the (normalized) spectral measure by 
$$\mu^A=\frac{1}{d}\sum_{i=1}^d \delta_{\lambda_i(A)},$$
where $\lambda_1(A)\ge...\ge \lambda_d(A)$ are the eigenvalues of $A$ and $\delta_x$ denotes the Dirac measure in $x$. If it is clear that we refer to a matrix $A$, we use the shortened notation $\lambda_1\ge...\ge \lambda_d$. We write $A^\ast$ for the adjoint of $A$. Let us introduce the Schatten norms for matrices 
$$\Arrowvert A \Arrowvert_{S_p}=\left(\sum_{i=1}^d \lambda_i(AA^\ast)^{p/2}\right)^{1/p},\ \ p\ge 1.$$
Furthermore, $\tr(A)$ denotes the trace of $A$ and $\rank(A)$ its rank.  For two matrices $A,B\in \R^{d\times n}$ we write $A\circ B=(A_{ik}B_{ik})_{i,k}$ for the Hadamard product. For any vector $v\in\R^d$, $\diag(v)\in\R^{d\times d}$ is the diagonal matrix with the $i$-th diagonal entry equal to $v_i$. With slight abuse of notation we also write $\diag(A)$ for $\diag(A_{11},\dots,A_{dd})$, $A\in \R^{d\times d}$.
The Stieltjes transform of a measure $\mu$ on the real line is defined by
\begin{align*}
m_\mu(z)=\int_\R \frac{1}{\lambda-z}\d \mu(\lambda),\ \ z\in \C^+.
\end{align*}
On the space of probability measures on $\R$ recall the following distance measures\begin{align*}\text{Kolmogorov metric:}\ \ & 
d_K(\mu,\nu)=\Arrowvert \mu((-\infty,\cdot])-\nu((-\infty,\cdot]) \Arrowvert,\\
\text{Dual bounded Lipschitz metric:}\ \ & d_{BL}(\mu,\nu)=\sup_{\Arrowvert f \Arrowvert_{BL}\le 1}\int_\R f\d(\mu-\nu),\\
\text{L\'evy metric:}\ \ &\\
&\hspace{-4cm}d_L(\mu,\nu)=\inf \Big\{\varepsilon>0 \ \Big\arrowvert \ \mu((-\infty,x-\varepsilon])-\varepsilon\le \nu((-\infty,x])\\
&\hspace{1cm} \le \mu((-\infty,x+\varepsilon])+\varepsilon  \text{ for all }x\in\R\Big\}.
\end{align*}
We will frequently make use of the well-known relation $d_L(\mu,\nu)\le d_K(\mu,\nu)$  for any two probability measures $\mu$ and $\nu$ on the real line, cf. \cite{Petrov1995}, p. 43. For any measures $\mu$ and $\nu$, $\mu\star \nu$ denotes their convolution. As usual, $\Longrightarrow$ stands for weak convergence. The Mar\v{c}enko-Pastur distribution with parameters $y,\sigma^2>0$ is given by
\begin{align}
\mu^{\text{MP}}_{y,\sigma^2}=\left(1-\frac{1}{y}\right)_+\delta_{0}+\frac{1}{2\pi\sigma^2}\frac{\sqrt{(b-x)(x-a)}}{yx}\ind\{a\le x \le b\}\d x\label{eq: mp dichte}
\end{align}
with $a=\sigma^2(1-\sqrt{y})^2$ and $b=\sigma^2(1+\sqrt{y})^2$. Moreover, for $\sigma^2>0$ let
$\mu^{\text{MP}}_{0,\sigma^2}=\delta_{\sigma^2}$.
The notation $\lesssim$ means less or equal up to some positive multiplicative constant which does not depend on the variable parameters in the expression. 

\subsection{Preliminaries}\label{sec: pre}
Let $(X(i,k))_{i,k\in\N}$ be a double array of iid centered random variables with unit variance. The left upper $d\times n$ submatrix is denoted by $X_{d,n}$.
Then the random vectors
$Y_{1,d,n},\dots,Y_{n,d,n}\in\R^d$ are the columns of the matrix
$$Y_{d,n}-\E Y_{d,n}=T_{d,n}^{1/2}X_{d,n}.$$
with 
$$T_{d,n}=\diag(T_{11,d,n},\dots,T_{dd,d,n})\in\R^{d\times d}.$$
This structure on the population covariance matrix is the simplest one which allows to visualize the effects of missing observations on the spectrum of the sample covariance matrix. The non-diagonal case is discussed at the end of Section \ref{section: part 1}. Its treatment requires some technical modification of the arguments presented here but not substantially new ideas and is beyond the scope of the article. $(\varepsilon_{d,n})_{d,n}$ is a triangular array of random matrices $\varepsilon_{d,n}\in\R^{d\times n}$ independent of $(X(i,k))_{i,k\in\N}$, where the entries $\varepsilon_{ik,d,n}$ are independent Bernoulli variables with observation probabilities $$\P(\varepsilon_{ik,d,n}=1)=p_{i,d,n},~~i=1,\dots,d,\ k=1,\dots n.$$
The dependence of the set $\mathcal N_{ij}$ and the number $N_{ij}$ in \eqref{eq: set and number} on the sequence $(\varepsilon_{d,n})$ is indicated by an additional subscript $d,n$.
Throughout this article we impose that the family of spectral measures of the population covariance matrices $(T_{d,n})$ as well as the family of empirical distributions $$\left(\mu^{w_{d,n}}\right)_{d,n}, \ \ \ \text{with}\ \ \mu^{w_{d,n}}=\frac{1}{d}\sum_{i=1}^d\delta_{w_{i,d,n}}\ \ \text{and}\ \ w_{d,n}=\big(p_{1,d,n}^{-1},...,p_{d,d,n}^{-1}\big),$$
are tight. This assumption ensures that there are not too many probabilities of observation $p_{i,d,n}$ in the vector $p_{d,n}$ that are very close to zero, in the sense that for most coordinates $i=1,...,n$ the number of observations remains proportional to $n$, while a few degeneracies may occur. 
Asymptotic statements refer to
\begin{align}\label{eq: asymptotic}
d\to \infty \ \ \text{ while $n=n(d)$ satisfies} \ \ \limsup_{d\to\infty}\,(d/n) <\infty.
\end{align}
The sequence of sample covariance matrices with missing observations is denoted by $$\left(\hat \Xi_{d,n}\right)_{d,n},$$
the corresponding sequence of spectral measures by $(\mu_{d,n})_{d,n}$ and their Stieltjes transforms by $(m_{d,n})_{d,n}.$

\section{Results}\label{section: results}
The main results of the article are the weak approximation of the spectral measure $\mu_{d,n}$ of $\hat\Xi_{d,n}$ by a non-random sequence of probability measures, and, in the null case, the almost sure convergence of the extremal eigenvalues of $\hat \Sigma_{d,n}$. Thereto, define the matrices
\begin{align*}
S_{d,n}&=\diag\left(\frac{1-p_{1,d,n}}{p_{1,d,n}}T_{11,d,n},\dots,\frac{1-p_{d,d,n}}{p_{d,d,n}}T_{dd,d,n}\right) \ \ \\ \text{and} \ \ R_{d,n}&=\diag\left(\frac{1}{p_{1,d,n}}T_{11,d,n},\dots,\frac{1}{p_{d,d,n}}T_{dd,d,n}\right).
\end{align*}
\begin{theorem}\label{theorem: main}
Suppose that the assumptions stated in Subsection \ref{sec: pre} hold, and $$\sup_d \Arrowvert R_{d,n} \Arrowvert_{S_\infty}<\infty.
$$ 
Then for any $z\in \C^+$, $$\arrowvert m_{d,n}(z)-m_{d,n}^\circ(z)\arrowvert \rightarrow 0 \ \ \text{a.s.},$$ where $m_{d,n}^{\circ}(z)$ satisfies 
\begin{align}
m_{d,n}^{\circ}(z)\ =\ \frac{1}{d}\tr\left\{\left(  \frac{1}{1+\frac{d}{n}e_{d,n}^{\circ}(z)}R_{d,n}-S_{d,n} - zI_{d\times d}\right)^{-1}\right\}\label{eq: def m-circ}
\end{align}
and $e_{d,n}^{\circ}$ is the (unique) solution of the fixed point equation
\begin{align*}
e_{d,n}^{\circ}(z)\ &=\ \frac{1}{d}\tr\left\{R_{d,n}\left(\frac{1}{1+\frac{d}{n}e_{d,n}^{\circ}(z)}R_{d,n}- S_{d,n}-zI_{d\times d}\right)^{-1}\right\}.
\end{align*}
Moreover, $m_{d,n}^\circ$ is the Stieltjes transform of a probability measure $\mu_{d,n}^\circ$ on the real line and
\begin{align*}
\mu_{d,n}^\circ-\mu_{d,n}\Longrightarrow 0 \ \ \text{a.s.}
\end{align*}
\end{theorem}
\begin{remark}
Note that the theorem covers in particular the case $d/n\rightarrow 0$. It follows from the proof that 
$$\arrowvert e_{d,n}^\circ(z)\arrowvert\le \frac{\| R_{d,n} \|_{S_\infty}}{\Im(z)}, \ \ z\in \C^+.$$
Due to $R_{d,n}-S_{d,n}=T_{d,n}$, this implies that the Stieltjes transforms $m_{d,n}^\circ$ approach those of the spectral measures of $T_{d,n}$ as $d/n\to 0$. That is, an effect caused by missing observations appears asymptotically only in the high-dimensional scenario $\liminf_d d/n>0$. 
\end{remark}
\noindent The equation \eqref{eq: def m-circ} characterizes uniquely the approximating spectral measure via its Stieltjes transform. Without missing observation, i.e. $p_{i,d,n}=1$, the solution of \eqref{eq: def m-circ} coincides with the solution to the Mar\v{c}enko-Pastur equation
\begin{align*}
m_{d,n}^\circ(z)=\frac{1}{d}\sum_{i=1}^d\frac{1}{T_{ii,d,n}\left(1-\frac{d}{n}-\frac{d}{n}z\cdot m_{d,n}^\circ(z)\right)-z}.
\end{align*}
The difference in the representation results from the fact that the spectra of $$T_{d,n}^{1/2}X_{d,n}^{~}X_{d,n}^\ast T_{d,n}^{1/2}\ \ \text{ and }\ \ X_{d,n}^\ast T_{d,n}^{~}X_{d,n}^{~} $$ are identical up to $\vert d-n\vert$ zero eigenvalues, which is used in the classical analysis. Except for special cases, this simplification is not possible in the missing at random scenario. \\
It is well-known that the Stieltjes transform of the Mar\v{c}enko-Pastur law  with parameters $\left(y,\sigma^2/p_0\right)$ is the unique solution to
$$
s(z)=\left(\frac{\sigma^2}{p_0}\cdot\frac{1}{1+\frac{\sigma^2}{p_0}ys(z)}-z\right)^{-1}
$$
from $\C^+\to \C^+$. In the special case $T_{d,n}=\sigma^2I_{d\times d}$ and $p_{d,n}=(p_0,\dots, p_0)\in(0,1)^d$, we have
$$
m_{d,n}^{\circ}\left(z-\sigma^2\frac{1-p_0}{p_0}\right)=\left(\frac{\sigma^2}{p_0}\frac{1}{1+\frac{d}{n}\frac{\sigma^2}{p_0}m_{d,n}^{\circ}\left(z-\sigma^2\frac{1-p_0}{p_0}\right)}-z\right)^{-1}.
$$
Hence, $\mu_{d,n}^\circ$ is the Mar\v{c}enko-Pastur law $\mu^{MP}_{\frac{d}{n},\frac{\sigma^2}{p_0}}$ shifted by $\sigma^2\frac{1-p_0}{p_0}$ to the left.
\begin{corollary}\label{cor: 1}
Grant the conditions of Theorem \ref{theorem: main}. If $p_{i,d,n}=p_0>0$ for $i=1,\dots,d$ and $d,n\in\N$ and $T_{d,n}=\sigma^2I_{d\times d}$, $\sigma^2>0$, we obtain
$$\mu_{d,n}\Longrightarrow   \mu^{\text{MP}}_{y,\frac{\sigma^2}{p_0}}\star \delta_{-\frac{1-p_0}{p_0}\sigma^2}\ \ \ a.s.$$
as $d\rightarrow\infty$ and $d/n\rightarrow y>0$. Eventually, as $ y<1$,
\begin{align*}
\limsup_d \lambda_{\min}\left(\hat  \Xi_{d,n}\right)<0 \ \ \text{a.s.~ ~ if }\ \ p_0<1-(1-\sqrt{y})^2.
\end{align*}

\end{corollary}
\noindent In other words, under the missing at random scenario where each component is observed with the same probability $p_0$, the limiting spectral distribution is a Mar\v{c}enko-Pastur law shifted by $\sigma^2(1-p_0)/p_0$ to the left. Eventually, the sample covariance matrix is not positive definite if $p_0$ is smaller than
$$
1-\left(1-\sqrt{y}\right)^2.
$$ 
For the estimator $\hat\Sigma_{d,n}$ we even determine the almost sure limit of the extremal eigenvalues. 

\begin{theorem}\label{theorem: extreme}
Grant the conditions of Corollary \ref{cor: 1} let additionally $\E X_{11}^4<\infty$ and $\varepsilon_{d,n}\in\R^{d\times n}$ be the upper left corner of a double array $(\varepsilon(i,k))_{i,k\in\N}$ of iid Bernoulli variables with parameter $p_0$.  Assume that $\E Y_{d,n}=0$.Then, if $0<y<1$,
\begin{align*}
\lim_{d\to\infty} \lambda_{\min}\left(\hat  \Sigma_{d,n}\right)&= \frac{\sigma^2}{p_0}\left(1-\sqrt{y}\right)^2-\frac{1-p_0}{p_0}\sigma^2 \ \ \text{a.s.},\ \ \ and\\
\lim_{d\to\infty} \lambda_{\max}\left(\hat  \Sigma_{d,n}\right)&= \frac{\sigma^2}{p_0}\left(1+\sqrt{y}\right)^2-\frac{1-p_0}{p_0}\sigma^2 \ \ \text{a.s.}
\end{align*}
\end{theorem}
\noindent The limit of the smallest eigenvalue is always smaller than in the completely observed case $p_0=1$, whereas the largest eigenvalue is always larger.
In the limiting case $y\to 0$ both expressions on the right-hand side reduce to $\sigma^2$ as in the completely observed classical case, independently of $p_0$.\\
As in Theorem 1 of \cite{Bai1993} the existence of the fourth moment is necessary for the above Theorem to hold. The proof of the necessity is a straightforward adaption of the arguments in \cite{Bai1988}.\\
The characterization of positive definiteness in the null case under the missing at random scenario is an immediate corollary of Theorem \ref{theorem: extreme}. 
\begin{corollary}
Under the condition of Theorem \ref{theorem: extreme},
\begin{align*}
\lim_{d\rightarrow\infty} \lambda_{\min}\left(\hat  \Sigma_{d,n}\right)<0 \ \ &\text{a.s.~ ~ if }\ \ p_0<1-(1-\sqrt{y})^2,\ \ \ \text{and}\\
\lim_{d\rightarrow\infty} \lambda_{\min}\left(\hat  \Sigma_{d,n}\right)>0 \ \ &\text{a.s.~ ~ if }\ \ p_0>1-(1-\sqrt{y})^2.
\end{align*}

\end{corollary}

\section{Proof of Theorem \ref{theorem: main}, Part I}\label{section: part 1}
$$ 
\text{\bf Reduction to the form $\boldsymbol{\frac{1}{n}R_{d,n}^{1/2}Z_{d,n}^{~}Z_{d,n}^\ast R_{d,n}^{1/2}-S_{d,n}^{~}}$}
$$
With the notation $$\bar T_{d,n}=\frac{1}{n}R_{d,n}^{1/2}Z_{d,n}^{~}Z_{d,n}^\ast R_{d,n}^{1/2}-S_{d,n}^{~}$$
and 
$$Z_{d,n}\in\R^{d\times n},\ \ Z_{ik,d,n}=\frac{X_{ik,d,n}\varepsilon_{ik,d,n}}{p_{i,d,n}^{1/2}}, \ \ i=1,\dots,d, \ \ k=1,\dots,n,$$ let
$\bar\mu_{d,n}$ be the spectral measure of $\bar T_{d,n}$. 
The aim of this section is to show that the spectral distributions $\mu_{d,n}$ of $\hat \Xi_{d,n}$ may be approximated by $\bar \mu_{d,n}$. 
\begin{proposition}\label{proposition: prop}
Grant the conditions of Subsection \ref{sec: pre}. Then
$$
d_L\left(\bar \mu_{d,n},\mu_{d,n}\right)\longrightarrow 0\ \  \text{a.s.}
$$
\end{proposition}
\begin{remark}
Corollary \ref{cor: 1} can be equally deduced from Proposition \ref{proposition: prop}. Since in that case $S_{d,n}$ is a multiple of identity, the eigenvalues satisfy
$$\lambda_i\left(\bar T_{d,n}\right)=\lambda_i\left(\frac{1}{n}R_{d,n}^{1/2}Z_{d,n}^{~}Z_{d,n}^\ast R_{d,n}^{1/2}\right)-\frac{1-p_0}{p_0}\sigma^2,\ \ i=1,\dots,d.$$
For the matrix
$$\frac{1}{n}R_{d,n}^{1/2}Z_{d,n}^{~}Z_{d,n}^\ast R_{d,n}^{1/2}$$
it is well-known (see e.g. \cite{Silverstein1995}) that the spectral distribution converges weakly  to $\mu^{MP}_{y,\sigma^2/p_0}$ almost surely as $ d/n \rightarrow y>0$.
\end{remark}
\noindent The proof of Proposition \ref{proposition: prop} is postponed to Appendix \ref{A}. At this place we give a sketch of the proof. Subsequently we restrict our attention to the estimator $\hat T_{d,n}$. The proof for $\hat \Sigma_{d,n}$ is just a simplified version.\\
The proof of Proposition \ref{proposition: prop} is subdivided into eight steps. In each step 
$\hat T_{d,n}$ is modified in a way which does not affect its spectral distribution asymptotically. In order to simplify the notation each modification of $\hat T_{d,n}$ from one step will be again denoted by $\hat T_{d,n}$ in the next step. Within the proof denote
$$\hat W_{d,n}\in\R^{d\times d},\ \hat W_{ij,d,n}=\frac{n}{N_{ij,d,n}},$$
$$W_{d,n}\in\R^{d\times d}, \ W_{ij,d,n}=\frac{n}{\E \#\NN_{ij,d,n}}.$$
Before we start with the description of the proof we rearrange the entries $\hat T_{ij,d,n}$ as follows
\begin{align*}
&\frac{1}{N_{ij,d,n}}\sum_{k\in\NN_{ij,d,n}}\left(Y_{ik,d,n}-\bar{Y}_{i,d,n}\right)\left(Y_{jk,d,n}-\bar{Y}_{j,d,n}\right)\\
&=\frac{1}{N_{ij,d,n}}\sum_{k\in\NN_{ij,d,n}}\Big(\left(Y_{ik,d,n}-\E Y_{ik,d,n}\right)-\left(\bar{Y}_{i,d,n}-\E Y_{ik,d,n}\right)\Big)\\
&\hspace{3cm}\times\Big(\left(Y_{jk,d,n}-\E Y_{jk,d,n}\right)-\left(\bar{Y}_{j,d,n}-\E Y_{jk,d,n}\right)\Big)\\
&=\frac{1}{N_{ij,d,n}}\sum_{k\in\NN_{ij,d,n}}\left[(Y_{ik,d,n}-\E Y_{ik,d,n})-\frac{1}{N_{ii,d,n}}\sum_{l\in\NN_{ii,d,n}}(Y_{il,d,n}-\E Y_{il,d,n})\right]\\
&\hspace{3cm}\times\left[(Y_{jk,d,n}-\E Y_{jk,d,n})-\frac{1}{N_{jj,d,n}}\sum_{l\in\NN_{jj,d,n}}(Y_{jl,d,n}-\E Y_{jl,d,n})\right].
\end{align*}
Therefore, we may assume without loss of generality $Y_{d,n}$ to be centered. Rewrite $\hat T_{d,n}$ in the following way
\begin{align*}
\hat T_{d,n}&=\frac{1}{n}\hat W_{d,n}\circ \Big((Y_{d,n}\circ\varepsilon_{d,n})(Y_{d,n}\circ\varepsilon_{d,n})^\ast\Big)-\frac{1}{n}\hat W_{d,n}\circ\left((\hat M_{d,n}\circ\varepsilon_{d,n})(Y_{d,n}\circ\varepsilon_{d,n})^\ast\right)\\
&\hspace{1.5cm}-\frac{1}{n}\hat W_{d,n}\circ \left((Y_{d,n}\circ\varepsilon_{d,n})(\hat M_{d,n}\circ\varepsilon_{d,n})^\ast\right)\\
&\hspace{1.5cm}+\frac{1}{n}\hat W_{d,n}\circ \left((\hat M_{d,n}\circ\varepsilon_{d,n})(\hat M_{d,n}\circ\varepsilon_{d,n})^\ast\right),
\end{align*}
where  
\begin{align}\label{eq: M}
\hat M_{d,n}=(\underbrace{\hat m_{d,n},...,\hat m_{d,n}}_{\text{n times}})\in\R^{d\times n} \ \ \ \text{with} \ \ \  \hat m_{i,d,n}=\frac{1}{N_{ii,d,n}}\sum_{k\in \NN_{ii,d,n}}Y_{ik,d,n}.\end{align}
Let us briefly describe the separate steps of the proof. The first three steps use the inequality
$$ d_K(\mu^A,\mu^B)\le \frac{1}{d}\rank(A-B)$$
for Hermitian matrices $A,B\in\R^{d\times d}$ in order to regularize certain rows of $\varepsilon_{d,n}$ for which the probability of observation $p_{i,d,n}$ is smaller than some given value $p_0>0$, to get rid of the additive term 
$$\frac{1}{n}\hat W_{d,n}\circ \left((\hat M_{d,n}\circ\varepsilon_{d,n})(\hat M_{d,n}\circ\varepsilon_{d,n})^\ast\right),$$
and to truncate the diagonal entries of $T_{d,n}$. Thereafter we want to make use of the inequality
\begin{equation}\label{equation: spur}
d_L^3\left(\mu^A,\mu^B\right)\le \frac{1}{d}\tr\big((A-B)(A-B)^\ast\big),
\end{equation}
where, in our case, $A$ and $B$ are two $d\times d$ random Hermitian matrices. In order to deduce almost sure convergence to $0$ of the right-hand side by means of the Borel-Cantelli lemma, truncation of the random variables $X_{ik,d,n}$ is necessary to guarantee the existence of higher order moments of the empirical spectral distribution of $\hat T_{d,n}$. This is realized in Step IV. In Step V the matrix $\hat W_{d,n}$ is replaced by its deterministic counterpart $W_{d,n}$ the evaluation of which is based on a sophisticated combinatorial analysis of moments. In Step VI a combination of both inequalities displayed above is applied. More precisely, an entry $Y_{ik,d,n}$ is preserved depending on whether its absolute row sum $\sum_{l}\left\arrowvert Y_{il,d,n}\right\arrowvert$ exceeds a certain value or not. The number of removed rows is asymptotically negligible while the remaining matrix is suitable for an application of \eqref{equation: spur}. Hereby, the matrices 
$$-\frac{1}{n}W_{d,n}\circ\left((\hat M_{d,n}\circ\varepsilon_{d,n})(Y_{d,n}\circ\varepsilon_{d,n})^\ast\right)\ \ \text{and}\ \ -\frac{1}{n}W_{d,n}\circ \left((Y_{d,n}\circ\varepsilon_{d,n})(\hat M_{d,n}\circ\varepsilon_{d,n})^\ast\right)$$
are removed from $\hat T_{d,n}$.
The form $$W_{d,n}=w_{d,n}w_{d,n}^\ast+\diag\left(W_{d,n}-w_{d,n}w_{d,n}^\ast\right)$$  is the motivation for replacing 
$$ \frac{1}{n} \diag\left(W_{d,n}-w_{d,n}w_{d,n}^\ast\right)\circ \left((Y_{d,n}\circ\varepsilon_{d,n})(Y_{d,n}\circ\varepsilon_{d,n})^\ast\right) $$
by its expectation in Step VII. Reverting finally the truncation Steps II, III, IV yields the claim.\\
\noindent In the next section $\hat \Xi_{d,n}$ denotes the matrix
$$\frac{1}{n}(w_{d,n}w_{d,n}^\ast)\circ \left((Y_{d,n}\circ\varepsilon_{d,n})(Y_{d,n}\circ\varepsilon_{d,n})^\ast\right)-S_{d,n}=\frac{1}{n}R_{d,n}^{1/2}Z_{d,n}^{~}Z_{d,n}^{\ast~} R_{d,n}^{1/2}-S_{d,n}$$
which is obtained in step VIII. Correspondingly, we write $\mu_{d,n}$ and $m_{d,n}$ for its spectral measure and the Stieltjes transform.
\begin{remark}
In the case of non-diagonal $T_{d,n}$ we  cannot reduce the sample covariance matrix with missing observations to the form
$$\frac{1}{n}R_{d,n}^{1/2}Z_{d,n}^{~}Z_{d,n}^{\ast~} R_{d,n}^{1/2}-S_{d,n}$$
but instead have to analyze the spectrum of
\begin{align*}
\frac{1}{n}(w_{d,n}w_{d,n}^\ast)\circ \Big((Y_{d,n}\circ \varepsilon_{d,n})(Y_{d,n}\circ\varepsilon_{d,n})^\ast\Big)-S_{d,n}=\frac{1}{n}\left(\tilde Y_{d,n}\circ\varepsilon\right)\left(\tilde Y_{d,n}\circ\varepsilon\right)^\ast-S_{d,n}
\end{align*}
with 
$$\tilde Y_{d,n}=\diag(w)Y_{d,n}.$$
Nevertheless, the arguments of Section \ref{section: part 2} can be modified at the cost of additional technical expenditure. We find that the ideas of the proof are much clearer for the diagonal special case and therefore omitted this extension due to length of the paper.
\end{remark}

\section{Proof of Theorem \ref{theorem: main}, Part II}\label{section: part 2} Note that, in general, the spectral analysis and limiting behavior of $\hat \Xi_{d,n}$ significantly differ from those of the matrix analyzed in \cite{Bai1995}.
By Proposition \ref{proposition: prop} as well as Lemma \ref{7.7} and Lemma \ref{7.8}, we continue to show that 
$$\left\arrowvert m_{d,n}(z)-m_{d,n}^\circ(z) \right\arrowvert\longrightarrow 0 \ \ \text{a.s.}$$
for all $z\in \C^+$. 
Such type of convergence has been established in \cite{couillet2011} for
\begin{align*}
B_{d,n}^{1/2}X_{d,n}^{~}X_{d,n}^\ast B_{d,n}^{1/2}+A_{d,n}
\end{align*}
for positive semidefinite Hermitian matrices $A_{d,n},B_{d,n}\in \C^{d\times d}$. For the proof of Theorem \ref{theorem: main} we establish the weak approximation in case of the negative semidefinite matrix $A_{d,n}=-S_{d,n}$. This requires several changes in the arguments of \cite{couillet2011} due to the fact that the function 
$$z\mapsto -\frac{1}{z(1+m(z))}$$
is a Stieltjes transform if $m$ is a Stieltjes transform of a finite measure on $[0,\infty)$ but in general, this is not true any longer if $m$
is just a Stieltjes transform of a finite measure on $\R$. Moreover, our proof includes also the case $d/n\rightarrow 0$.\\
The proof is structured as follows. In the first step we truncate the entries of $X_{d,n}$ at the threshold level $K>0$ which goes to infinity at the very end. Afterwards we start to analyze the Stieltjes transform of the empirical spectral distribution of $\hat\Xi_{d,n}$. With the resolvent
$$\hat G_{d,n}(z)=\left(\Xi_{d,n}-zI_{d\times d}\right)^{-1}$$
we prove that 
$$e_{d,n}(z) = \frac{1}{d}\tr\left\{R_{d,n}\hat G_{d,n}(z)\right\}$$
is an approximate solution to the fixed point equation in Theorem \ref{theorem: main} in Step II. Correspondingly, the Stieltjes transform $m_{d,n}$ is shown to be approximated by the expression \eqref{eq: def m-circ} with $e_{d,n}$ in place of $e_{d,n}^\circ$. In the third step existence and uniqueness of a solution to the system of equations for $m_{d,n}^\circ$ is established. The solution $m_{d,n}^\circ$ is identified as a Stieltjes transform in Step IV. In Step V and VI, pointwise almost sure convergence of $e_{d,n}-e_{d,n}^\circ$ and $m_{d,n}-m_{d,n}^\circ$ to zero is derived. Finally, we deduce the weak convergence $\mu_{d,n}-\mu_{d,n}^\circ\Longrightarrow 0$ almost surely in Step VII.

\subsection{Step I: Second truncation of $X_{d,n}$} 
For arbitrary $K>0$ define matrices $\tilde X_{d,n}$, $\tilde Z_{d,n}$ and $\tilde \Xi_{d,n}=n^{-1}R_{d,n}^{1/2}\tilde Z_{d,n} \tilde Z_{d,n}^\ast R_{d,n}^{1/2}-S_{d,n}$, where
$$\tilde X_{ik}=X_{ik}\ind\{ |X_{ik}|\le K\} \ \ \text{and} \ \ \tilde Z_{ik,d,n}=\frac{\tilde X_{ik,d,n}\varepsilon_{ik,d,n}}{p_{i,d,n}^{1/2}}.$$
Moreover, define for arbitrary $\delta>0$ the event
$$\Delta_{i,d,n}=\left\{\frac{1}{n}\left|\sum_{l=1}^nX_{il}^2-\E X_{il}^2\right|\vee\frac{1}{n}\left|\sum_{l=1}^nX_{il}^2\ind\{|X_{il}|>K\}-\E X_{il}^2\ind \{|X_{il}|>K\}\right|< \delta\right\}.$$
With this notation, let $$\hat\Xi_{d,n}'=\frac{1}{n}R_{d,n}^{1/2} Z_{d,n}' (Z_{d,n}')^\ast R_{d,n}^{1/2}-S_{d,n}\text{ and} \ \ \tilde \Xi_{d,n}'=\frac{1}{n}R_{d,n}^{1/2}\tilde Z_{d,n}' (\tilde Z_{d,n}')^\ast R_{d,n}^{1/2}-S_{d,n},$$ where
$$X'_{ik,d,n}=X_{ik}\ind_{\Delta_{i,d,n}}, \ \ \tilde X'_{ik,d,n}=\tilde X_{ik}\ind_{\Delta_{i,d,n}}$$
and
$$Z_{ik,d,n}'=\frac{ X_{ik,d,n}'\varepsilon_{ik,d,n}}{p_{i,d,n}^{1/2}}, \ \ \tilde Z_{ik,d,n}'=\frac{\tilde X_{ik,d,n}'\varepsilon_{ik,d,n}}{p_{i,d,n}^{1/2}}.$$
Then,
\begin{align}
&d_L\left(\mu^{\hat\Xi_{d,n}},\mu^{\tilde \Xi_{d,n}}\right)\notag\\ 
&\hspace{1cm}\le d_L\left(\mu^{\hat\Xi_{d,n}},\mu^{\hat \Xi_{d,n}'}\right)+d_L\left(\mu^{\hat\Xi_{d,n}'},\mu^{\tilde \Xi_{d,n}'}\right)+d_L\left(\mu^{\tilde\Xi_{d,n}'},\mu^{\tilde \Xi_{d,n}}\right).\label{equation: second truncation levy}
\end{align}
First, we evaluate the second term $d_L(\mu^{\hat\Xi_{d,n}'},\mu^{\tilde \Xi_{d,n}'})$ in \eqref{equation: second truncation levy}.
By Theorem \ref{theorem: A38} for $\alpha=1$, the Lidskii-Wielandt perturbation bound (1.2) in \cite{Li1999}, and H\"older's inequality for Schatten norms,
\begin{align*}
\ \ &d_L^2\left(\mu^{\hat \Xi_{d,n}'},\mu^{\tilde \Xi_{d,n}'}\right)\\
&\le \frac{1}{d}\sum_{i=1}^d\left|\lambda_i(\hat \Xi_{d,n}')-\lambda_i(\tilde \Xi_{d,n}')\right|\\
&\le \frac{1}{dn}\left\Arrowvert R_{d,n}^{1/2}Z_{d,n}'(Z_{d,n}')^\ast R_{d,n}^{1/2}-R_{d,n}^{1/2}\tilde Z_{d,n}' (\tilde Z_{d,n}')^\ast R_{d,n}^{1/2}\right\Arrowvert_{S_1}\\
&\le  \frac{1}{dn}\left\Arrowvert Z_{d,n}'(Z_{d,n}')^\ast-\tilde Z_{d,n}' (\tilde Z_{d,n}')^\ast\right\Arrowvert_{S_1}\left\Arrowvert R_{d,n}\right\Arrowvert_{S_\infty}\\
&=\frac{1}{dn}\left\Arrowvert(Z_{d,n}'-\tilde Z_{d,n}')(Z_{d,n}'-\tilde Z_{d,n}')^\ast+(Z_{d,n}'-\tilde Z_{d,n}')(\tilde Z_{d,n}')^\ast+\tilde Z_{d,n}'(Z_{d,n}'-\tilde Z_{d,n}')^\ast\right\Arrowvert_{S_1}\\
&\hspace{5cm}\times\left\Arrowvert R_{d,n}\right\Arrowvert_{S_\infty}\\
&\le \frac{1}{dn}\bigg(\left\Arrowvert (Z_{d,n}'-\tilde Z_{d,n}')(Z_{d,n}'-\tilde Z_{d,n}')^\ast \right\Arrowvert_{S_1}+2\left\Arrowvert (Z_{d,n}'-\tilde Z_{d,n}')(\tilde Z_{d,n}')^\ast\right\Arrowvert_{S_1}\bigg)\left\Arrowvert R_{d,n}\right\Arrowvert_{S_\infty}\\
&\le\frac{1}{dn}\bigg(\left\Arrowvert (Z_{d,n}'-\tilde Z_{d,n}')(Z_{d,n}'-\tilde Z_{d,n}')^\ast \right\Arrowvert_{S_1}+2\left\Arrowvert Z_{d,n}'-\tilde Z_{d,n}'\right\Arrowvert_{S_2}\left\Arrowvert\tilde Z_{d,n}'\right\Arrowvert_{S_2}\bigg)\\
&\hspace{5cm}\times\left\Arrowvert R_{d,n}\right\Arrowvert_{S_\infty}\\
&\le\frac{1}{dn}\bigg\{\tr\big((Z_{d,n}'-\tilde Z_{d,n}')(Z_{d,n}'-\tilde Z_{d,n}')^\ast\big)\\
&\hspace{1.2cm}+2\Big(\tr\big((Z_{d,n}'-\tilde Z_{d,n}')(Z_{d,n}'-\tilde Z_{d,n}')^\ast\big)\Big)^{1/2}\big(\tr (\tilde Z_{d,n}'(\tilde Z_{d,n}')^\ast)\big)^{1/2}\bigg\} \left\Arrowvert R_{d,n}\right\Arrowvert_{S_\infty}.
\end{align*}
As in  Subsection \ref{subsec: modify} let $p_0>0$ be the lower bound on $p_{i,d,n},~i=1,\dots,d$ and $d\in\N$. With this notation, we show that
$$\sup_{d}\frac{1}{dn}\tr\Big((Z_{d,n}'-\tilde Z_{d,n}')(Z_{d,n}'-\tilde Z_{d,n}')^\ast\Big)\le \frac{\E X_{11}^2\ind\{|X_{11}|>K\}+\delta}{p_0},$$
while $$\sup_{d}\frac{1}{dn}\tr\big(\tilde Z_{d,n}' (\tilde Z_{d,n}')^\ast\big) \le \frac{\delta+\E X_{11}^2}{p_0}.$$ 
We have
\begin{align*}
&\frac{1}{dn}\tr\Big((Z_{d,n}'-\tilde Z_{d,n}')(Z_{d,n}'-\tilde Z_{d,n}')^\ast\Big)\\
&\hspace{2cm}=\frac{1}{dn}\sum_{i=1}^d\sum_{k=1}^n(Z_{ik,d,n}'-\tilde Z_{ik,d,n}')^2\\
&\hspace{2cm}\le \frac{1}{dn}\max_{i=1,\dots,d} \left(\frac{1}{p_{i,d,n}}\right)\sum_{i=1}^d\ind_{\Delta_{i,d,n}}\sum_{k=1}^nX_{ik}^2\ind\{|X_{ik}|>K\}\\
&\hspace{2cm}\le \frac{\E X_{11}^2\ind\{|X_{11}|>K\}+\delta}{p_0}.
\end{align*}
Moreover,
\begin{align*}
\frac{1}{dn}\tr\left(\tilde Z_{d,n}' (\tilde Z_{d,n}')^\ast\right)&=\frac{1}{dn}\sum_{i=1}^d\sum_{k=1}^n (\tilde Z_{ik,d,n}')^2\\
&\le\frac{1}{dn}\max_{i=1,\dots,d} \left(\frac{1}{p_{i,d,n}}\right)\sum_{i=1}^d\ind_{\Delta_{i,d,n}}\sum_{k=1}^n X_{ik,d,n}^2\ind\{ |X_{ik,d,n}|\le K\} \\
&\le \frac{\E X_{11}^2+\delta}{p_0}.
\end{align*}
As concerns the first summand in \eqref{equation: second truncation levy}, it holds $\P(\Delta_{i,d,n})\rightarrow 1$ as $d\to\infty$ by weak law of large numbers. Note that $\P(\Delta_{1,d,n})=\P(\Delta_{2,d,n})=\dots=\P(\Delta_{d,d,n})$. Then by Hoeffding's inequality for sufficiently large $d$,
\begin{align*}
\P\left(\sum_{i=1}^d\ind_{\Delta_{i,d,n}^c}\ge \delta d\right)&\le \P\left(\sum_{i=1}^d\left(\ind_{\Delta_{i,d,n}^c}-\P(\Delta_{i,d,n}^c)\right)\ge \frac{1}{2}\delta d\right)\\
&\le \exp\left(-\frac{\delta^2d}{2}\right).
\end{align*}
Hence, by the Borel-Cantelli lemma
$$
\limsup_{d\to\infty}\frac{1}{d}\sum_{i=1}^d\ind_{\Delta_{i,d,n}^c}<\delta
$$
almost surely. As in inequality \eqref{inequality: null vector} of Subsection \ref{subsec: trunc X 1} we deduce
\begin{align*}\limsup_{d\to \infty} d_L\left(\mu^{\hat\Xi_{d,n}},\mu^{\hat \Xi_{d,n}'}\right)&\le \limsup_{d\to \infty} d_K\left(\mu^{\hat\Xi_{d,n}},\mu^{\hat \Xi_{d,n}'}\right)\\
&\le \limsup_{d\to \infty}\frac{1}{d}\rank\left(\hat\Xi_{d,n}-\hat \Xi_{d,n}'\right)\\
&\le 2\delta\end{align*}
almost surely. The third summand in \eqref{equation: second truncation levy} is bounded in the same way. Putting things together in right hand side of \eqref{equation: second truncation levy},
\begin{align*}
&\limsup_{d\to\infty} d_L\left(\mu^{\hat\Xi_{d,n}},\mu^{\tilde \Xi_{d,n}}\right)\\
&\hspace{1.5cm}\le 4\delta+\sup_d\left\Arrowvert R_{d,n}\right\Arrowvert_{S_\infty}^{1/2}\Bigg[\frac{\E X_{11}^2\ind\{|X_{11}|>K\}+\delta}{p_0}\\
&\hspace{4.5cm}+2\frac{\sqrt{\E X_{11}^2\ind\{|X_{11}|>K\}+\delta} \sqrt{\delta+\E X_{11}^2}}{p_0}\Bigg]^{1/2}
\end{align*}
almost surely. Since $\delta$ may be chosen arbitrarily small, we conclude
\begin{align*}
&\limsup_{d\to\infty} d_L\left(\mu^{\hat\Xi_{d,n}},\mu^{\tilde \Xi_{d,n}}\right)\\ 
&\hspace{0.8cm}\le \sup_d\left\Arrowvert R_{d,n}\right\Arrowvert_{S_\infty}^{1/2}\Bigg[\frac{\E X_{11}^2\ind\{|X_{11}|>K\}}{p_0}+2\frac{\sqrt{\E X_{11}^2\ind\{|X_{11}|>K\}} \sqrt{\E X_{11}^2}}{p_0}\Bigg]^{1/2}.
\end{align*}
In turn, the last expression can be made arbitrary small for $K$ sufficiently large.
Since the centralization of the truncated random variables $\tilde X_{ik}$ leads to a finite rank perturbation of $\tilde\Xi_{d,n}$ (uniformly in $d$), we may assume the entries of $\tilde X_{ik}$ to be centered. In the following denote the centered truncated random matrix again by $X_{d,n}$. Then, analogously to the truncation step by replacing $\ind\{ |X_{ik}|\le K\}$ with $(\E X_{11}^2)^{-1/2}$ in the definition of $\tilde X$ we may assume the entries to be standardized since the variance of the truncated variables converges to one as the truncation level tends to infinity.
Therefore, in the rest of the proof we analyze the matrix
$$\hat \Xi_{d,n}=\frac{1}{n}R^{1/2}_{d,n}Z_{d,n}^{~}Z_{d,n}^\ast R^{1/2}_{d,n}-S_{d,n},$$
where the entries of the matrix $Z_{d,n}$ are centered, standardized and bounded.


\subsection{Step II: Approximate solution to the fixed point equation \eqref{eq: def m-circ}} Subsequently, we assume that
\begin{equation}\label{dn}
\liminf_{d\to\infty}\frac{d^{\frac{3}{2}}}{n}>0.
\end{equation}
The general case is treated in Step VI.
Recall that $\mu_{d,n}$ denotes the (normalized) spectral measure of $\hat \Xi_{d,n}$, and denote its Stieltjes transform by
\begin{align}
m_{d,n}(z)\ =\ \int\frac{1}{\lambda-z}\dd \mu_{d,n}(\lambda), \ \ z\in\C^+\label{eq: def m}.
\end{align}

\medskip
\noindent
We use subsequently the following abbreviations for the resolvents
$$
\hat G_{d,n}(z)=\left(\hat \Xi_{d,n}-zI_{d\times d}\right)^{-1} \ \ \text{and} \ \ \hat G_{d,n}^{(k)}(z)=\left(\hat \Xi_{d,n}^{(k)}-zI_{d\times d}\right)^{-1}, \ k=1,\dots,n.
$$
For $z\in\C^+$, define
$$
e_{d,n}(z)\ =\ \frac{1}{d}\tr\left\{R_{d,n}\hat G_{d,n}(z)\right\}.
$$
Our goal in this step is to show that
\begin{align}
\frac{1}{d}\tr\left(D_{d,n}^{-1}(z)\right)-m_{d,n}(z)\ &\rightarrow\ 0\ \ \ a.s., \ \ \ \text{and}\label{eq: intermediate1}\\
\frac{1}{d}\tr\left(R_{d,n}^{~}D_{d,n}^{-1}(z)\right)-e_{d,n}(z)\ &\rightarrow\ 0\ \ \ a.s.\label{eq: intermediate2}
\end{align}
with 
\begin{equation}
D_{d,n}(z)=\frac{1}{1+\frac{d}{n}e_{d,n}(z)}R_{d,n}-S_{d,n}-zI_{d\times d}.\label{eq: definition D}
\end{equation}
Let
$
\hat \Xi_{d,n}=O_{d,n}\Lambda_{d,n} O_{d,n}^\ast 
$ 
denote a spectral decomposition, where $$\Lambda_{d,n}=\diag(\lambda_{1,d,n},\dots,\lambda_{d,d,n}),$$ and define $
\underline{R}_{d,n}=O_{d,n}^\ast R_{d,n}O_{d,n}$.
With this notation,
\begin{align}\notag
e_{d,n}(z) &= \frac{1}{d}\tr\left\{R_{d,n}\hat G_{d,n}(z)\right\}\\
&= \frac{1}{d}\tr \left\{R_{d,n}\left(O_{d,n}\Lambda_{d,n} O_{d,n}^\ast-zI_{d\times d}\right)^{-1}\right\}\nonumber\\
&=\frac{1}{d}\tr \left\{R_{d,n}\left(O_{d,n}\left[\Lambda_{d,n} -zI_{d\times d}\right]O_{d,n}^\ast\right)^{-1}\right\}\nonumber\\
&= \frac{1}{d}\tr \left\{R_{d,n}\left(O_{d,n}\left[\Lambda_{d,n} -zI_{d\times d}\right]^{-1}O_{d,n}^\ast\right)\right\}\nonumber\\
&= \frac{1}{d}\tr \left\{O_{d,n}^\ast R_{d,n}O_{d,n}\left(\Lambda_{d,n} -zI_{d\times d}\right)^{-1}\right\}\nonumber\\
&= \frac{1}{d}\tr \left\{\underline{R}_{d,n}\left(\Lambda_{d,n} -zI_{d\times d}\right)^{-1}\right\}\nonumber\\
&= \frac{1}{d}\sum_{i=1}^d\frac{\underline{R}_{ii,d,n}}{\lambda_{i,d,n}-z}.\label{eq: Stieltjes 1}\end{align}
Since $R_{d,n}$ and therefore $\underline{R}_{d,n}$ are positive semidefinite, the diagonal entries $ \underline{R}_{ii,d,n}$, $i=1,...,d$, are non-negative. Hence, $e_{d,n}$ is the Stieltjes transform of a measure on $\R$ with at most $d$ support points and total mass 
$$
\frac{1}{d}\tr R_{d,n}.
$$ 
Note that $\hat \Xi_{d,n}$ is not necessarily positive semidefinite, hence the support points are not restricted to $[0,\infty)$. As a Stieltjes transform,
\begin{equation}
e_{d,n}: \C^+\rightarrow\C^+.\label{eq: e Stieltjes}
\end{equation}
This implies in particular that $D_{d,n}(z)$ as defined in \eqref{eq: definition D} is in fact invertible by means of Lemma \ref{lemma8}. 
Moreover, since $\Arrowvert R_{d,n}\Arrowvert_{S_{\infty}}\leq \kappa$ for some constant $\kappa>0$,
it follows by H\"older's inequality and the positive definiteness of $R_{d,n}$,
\begin{align}
\vert e_{d,n}(z) \vert\   & \leq\ \frac{1}{d}\Arrowvert\underline{R}_{d,n}\Arrowvert_{S_1}\left\Arrowvert \left(\Lambda_{d,n}-zI_{d\times d}\right)^{-1}\right\Arrowvert_{S_{\infty}}\nonumber\\
&=\ \left(\frac{1}{d}\tr R_{d,n}\right)\max_{1\leq i\leq d}\frac{1}{\arrowvert\lambda_{i,d,n}-z\arrowvert}\notag\\
&\leq\ \frac{ \kappa}{\Im z}\label{eq: stieltjes3}.
\end{align}
Let $Z_{k,d,n}$ be the $k$-th column of the matrix $Z_{d,n}$, and define
\begin{equation*}
Y_{k,d,n}= \frac{1}{\sqrt{n}}R_{d,n}^{1/2}Z_{k,d,n}\ \ \text{and} \ \ \hat \Xi_{d,n}^{(k)}\ =\ \hat \Xi_{d,n}-Y_{k,d,n}Y_{k,d,n}^\ast, \ k=1,\dots,n,
\end{equation*}
which arises from $\hat \Xi_{d,n}$ by taking away the $k$-th sample vector, and recall \eqref{eq: definition D}.
Then,
$$
\hat \Xi_{d,n}-zI_{d\times d}-D_{d,n}(z) = \sum_{k=1}^nY_{k,d,n}Y_{k,d,n}^\ast-\frac{1}{1+\frac{d}{n}e_{d,n}(z)}R_{d,n},
$$
whence
\begin{align*}
D_{d,n}(z)&\left\{\hat G_{d,n}(z)-D_{d,n}^{-1}(z)\right\}\left(\hat \Xi_{d,n}-zI_{d\times d}\right)\\
&=\ D_{d,n}(z)-\left(\hat \Xi_{d,n}-zI_{d\times d}\right)\\
&=\ \frac{1}{1+\frac{d}{n}e_{d,n}(z)}R_{d,n}-\sum_{k=1}^nY_{k,d,n}Y_{k,d,n}^\ast.
\end{align*}
Therefore,
\begin{align}
\hat G_{d,n}(z)-D_{d,n}^{-1}(z) &= -\sum_{k=1}^nD_{d,n}^{-1}(z)Y_{k,d,n}Y_{k,d,n}^\ast \hat G_{d,n}(z)\nonumber\\
&\quad\quad\quad +\frac{1}{1+\frac{d}{n}e_{d,n}(z)}D_{d,n}^{-1}(z)R_{d,n}\hat G_{d,n}(z)\nonumber\\
&=\ -\sum_{k=1}^n\frac{D_{d,n}^{-1}(z)Y_{k,d,n}Y_{k,d,n}^\ast\hat G_{d,n}^{(k)}(z)}{1+Y_{k,d,n}^\ast\hat G_{d,n}^{(k)}(z)Y_{k,d,n}}\label{eq: lemma4}\\
&\quad\quad\quad +\frac{1}{1+\frac{d}{n}e_{d,n}(z)}D_{d,n}^{-1}(z)R_{d,n}\hat G_{d,n}(z),\nonumber\end{align}
where \eqref{eq: lemma4} follows from Lemma \ref{lemma4}. Altogether,
\begin{equation}
\frac{1}{d}\tr\left(D_{d,n}^{-1}(z)\right)-m_{d,n}(z)\ =\ \frac{1}{n}\sum_{k=1}^nf_{k,m}\label{eq: m}
\end{equation}
with 
\begin{align*}
f_{k,m}\ =\ \frac{1}{d}\frac{Z_{k,d,n}^\ast R_{d,n}^{1/2}\hat G_{d,n}^{(k)}(z)D_{d,n}^{-1}(z)R_{d,n}^{1/2}Z_{k,d,n}}{1+Y_{k,d,n}^\ast\hat G_{d,n}^{(k)}(z)Y_{k,d,n}}\ -\ \frac{1}{d}\frac{\tr\left(R_{d,n}\hat G_{d,n}(z)D_{d,n}^{-1}(z)\right)}{1+\frac{d}{n}e_{d,n}(z)}.
\end{align*}
Multiplication of the matrix equality \eqref{eq: lemma4} with $R_{d,n}$ from the right, we deduce
\begin{equation}
\frac{1}{d}\tr\left(R_{d,n}D_{d,n}^{-1}(z)\right)-e_{d,n}(z) =\frac{1}{n}\sum_{k=1}^nf_{k,e}\label{eq: e}
\end{equation}
with
\begin{align*}
f_{k,e} &= \frac{1}{d}\frac{Z_{k,d,n}^\ast R_{d,n}^{1/2}\hat G_{d,n}^{(k)}(z)R_{d,n}D_{d,n}^{-1}(z)R_{d,n}^{1/2}Z_{k,d,n}}{1+Y_{k,d,n}^\ast\hat G_{d,n}^{(k)}(z)Y_{k,d,n}}\\ 
&\ \ \ \ \ \ \ \ \ \ \ \ \  -\ \frac{1}{d}\frac{\tr\left(R_{d,n}\hat G_{d,n}(z)R_{d,n}D_{d,n}^{-1}(z)\right)}{1+\frac{d}{n}e_{d,n}(z)}.
\end{align*}
Subsequently, we show that 
\begin{align}
\lim_{d\rightarrow\infty}\frac{1}{n}\sum_{k=1}^nf_{k,x}\ =\ 0\ \ a.s.,\ \ \ x=e,m.\label{eq: Konvergenz}
\end{align}
First observe that
\begin{align*}
Y_{k,d,n}^\ast\hat G_{d,n}^{(k)}(z)Y_{k,d,n} = \tr\left( Y_{k,d,n}Y_{k,d,n}^\ast\hat G_{d,n}^{(k)}(z)\right)
\end{align*}
is the Stieltjes transform of a measue on $\R$ with total mass $\Arrowvert Y_{k,d,n}\Arrowvert_2^2$, following the arguments in \eqref{eq: Stieltjes 1}. Next, with $\lambda_{1,d,n}^{(k)},...,\lambda_{d,d,n}^{(k)}$ denoting the eigenvalues of $\hat \Xi_{d,n}^{(k)}$,
\begin{align}
\left\Arrowvert\hat G_{d,n}^{(k)}(z)\right\Arrowvert_{S_{\infty}} 
&= \max_{i=1,\dots,d} \frac{1}{\sqrt{\left(\lambda_{i,d,n}^{(k)}-\Re z\right)^2+\Im(z)^2}}\nonumber\\
&\leq \frac{1}{\Im z}\label{eq: 22}. 
\end{align}
The same holds true for $\hat G_{d,n}(z)$ in place of $\hat G_{d,n}^{(k)}(z)$. Therefore, 
$$
\left\arrowvert Y_{k,d,n}^\ast\hat G_{d,n}^{(k)}(z)Y_{k,d,n}\right\arrowvert \leq \frac{\Arrowvert Y_k\Arrowvert_2^2}{\Im z},
$$
which gives
\begin{equation}\label{eq: first bound}
\left\arrowvert\frac{1}{1+ Y_{k,d,n}^\ast\hat G_{d,n}^{(k)}(z)Y_{k,d,n}}\right\arrowvert \leq  \frac{1}{1-\frac{\Arrowvert Y_{k,d,n}\Arrowvert_2^2}{\Im z}}\ \ \ \text{if}\ \frac{\Arrowvert Y_{k,d,n}\Arrowvert_2^2}{\Im z}<1.
\end{equation}
Denoting with   $O\Lambda O^\ast$ the spectral decomposition of $\hat \Xi_{d,n}^{(k)}$ and  $V_{ii}^{(k)}=(O^\ast Y_{k,d,n}Y_{k,d,n}^\ast O)_{ii}$ for the moment, we obtain for $\Arrowvert Y_{k,d,n}\Arrowvert_2>0$ the bound
\begin{align}
\left\arrowvert\frac{1}{1+ Y_{k,d,n}^\ast \hat G_{d,n}^{(k)}(z)Y_{k,d,n}}\right\arrowvert &\leq 
\frac{1}{\Im\left(Y_{k,d,n}^\ast \hat G_{d,n}^{(k)}(z)Y_{k,d,n}\right)}\nonumber\\
&= \frac{1}{\Im(z)\sum_{i=1}^d\frac{V_{ii}^{(k)}}{(\lambda_{i,d,n}^{(k)}-\Re(z))^2+\Im(z)^2}}\nonumber\\
&\leq  \frac{1}{\Im(z)\sum_{i=1}^d\frac{V_{ii}^{(k)}}{2\max_i\arrowvert \lambda_{i,d,n}^{(k)}\arrowvert^2 +2\arrowvert z\arrowvert^2}}\nonumber\\
&\leq \frac{2\max_i\left\arrowvert\lambda_{i,d,n}^{(k)}\right\arrowvert^2+2\arrowvert z\arrowvert^2}{\Im(z)\Arrowvert Y_k\Arrowvert^2_2}.\label{eq: second bound}
\end{align}
Combining the first bound \eqref{eq: first bound} in case $\Arrowvert Y_{k,d,n}\Arrowvert_2^2/\Im z\leq 1/2$ with the second bound \eqref{eq: second bound} if  $\Arrowvert Y_{k,d,n}\Arrowvert_2^2/\Im z> 1/2$ yields 
\begin{align}
\left\arrowvert\frac{1}{1+ Y_{k,d,n}^\ast \hat G_{d,n}^{(k)}(z)Y_{k,d,n}}\right\arrowvert &\leq 2\left\{\frac{\max_i\left\arrowvert\lambda_{i,d,n}^{(k)}\right\arrowvert^2+\arrowvert z\arrowvert^2}{\Im(z)^2}+ 1\right\}\nonumber\\
&\leq \frac{2\max_i\left\arrowvert\lambda_{i,d,n}^{(k)}\right\arrowvert^2+4\arrowvert z\arrowvert^2}{\Im(z)^2}.\label{eq: Stieltjes5}
\end{align}
Finally, due to 
$$\big\Arrowvert\hat \Xi_{d,n}\big\Arrowvert_{S_\infty}\le \big\Arrowvert S_{d,n}\big\Arrowvert_{S_\infty}+\Bigg\Arrowvert\sum_{\substack{l=1\\ l\neq k}}^nY_{l,d,n}Y_{l,d,n}^\ast \Bigg\Arrowvert_{S_\infty}$$ 
and Lemma \ref{lemma: log d},
\begin{align}
\underset{d\rightarrow\infty}{\lim\sup}\left\{\left( \frac{2c+4\arrowvert z\arrowvert^2}{\Im(z)^2}\right)^{-1}  \left\arrowvert\frac{1}{1+ Y_{k,d,n}^\ast \hat G_{d,n}^{(k)}(z)Y_{k,d,n}}\right\arrowvert\right\} &\leq C<\infty\label{eq: Stieltjes15}
\end{align}
almost surely for some constants $C,c>0$. Define
$$
e^{(k)}_{d,n} = \frac{1}{d}\tr\left(R_{d,n}\hat G_{d,n}^{(k)}(z)\right),\ \ k\in\{1,...,n\}.
$$
Note that analogously to \eqref{eq: Stieltjes 1}, $e_{d,n}^{(k)}$ is a Stieltjes transform. Using \eqref{eq: 22} and the arguments of \eqref{eq: first bound} for the case $n^{-1}\tr(R_{d,n})/\Im(z)\leq 1/2$ as well as  \eqref{eq: second bound} for  $n^{-1}\tr(R_{d,n})/\Im(z)> 1/2$
we obtain analogously
\begin{equation}\label{eq: bound3}
\left\arrowvert\frac{1}{1+\frac{d}{n}e_{d,n}^{(k)}(z)}\right\arrowvert \leq \frac{2\max_i\left\arrowvert\lambda_{i,d,n}^{(k)}\right\arrowvert^2+4\arrowvert z\arrowvert^2}{\Im(z)^2}
\end{equation}
and for some constants $C,c>0$
\begin{equation}
\underset{d\rightarrow\infty}{\lim\sup}\left\{\left( \frac{2c+4\arrowvert z\arrowvert^2}{\Im(z)^2}\right)^{-1} \left\arrowvert\frac{1}{1+\frac{d}{n}e_{d,n}^{(k)}(z)}\right\arrowvert\ \right\}\leq  C<\infty.\label{eq: Stieltjes16}
\end{equation}
The same bound holds true for $e_{d,n}$ instead of $e_{d,n}^{(k)}$, in which case $\lambda_{i,d,n}^{(k)}$ are to be replaced by the eigenvalues $\lambda_{i,d,n}$ of $\hat \Xi_{d,n}$. Therefore, with 
$$
\psi_{d,n}^{(k)} = \max_{i=1,\dots,d}\left\{\left(\lambda_{i,d,n}^{(k)}\right)^2,\lambda_{i,d,n}^2\right\},
$$
\begin{align}
\left\arrowvert \frac{1}{1+\frac{d}{n}e_{d,n}(z)}-\frac{1}{1+\frac{d}{n}e_{d,n}^{(k)}(z)}\right\arrowvert &= \frac{d}{n}\cdot\frac{\left\arrowvert e_{d,n}^{(k)}(z)-e_{d,n}(z)\right\arrowvert}{\left\arrowvert \left(1+\frac{d}{n}e_{d,n}(z)\right) \left(1+\frac{d}{n}e_{d,n}^{(k)}(z)\right)\right\arrowvert}\nonumber\\
&\leq \frac{d}{n}\frac{\Arrowvert R_{d,n}\Arrowvert_{S_{\infty}}}{\Im z}\frac{1}{d}\frac{1}{\left\arrowvert \left(1+\frac{d}{n}e_{d,n}(z)\right) \left(1+\frac{d}{n}e_{d,n}^{(k)}(z)\right)\right\arrowvert}\label{eq: stieltjes1}\\
&\leq \frac{1}{n}\frac{\Arrowvert R_{d,n}\Arrowvert_{S_{\infty}}}{\Im z}\left(\frac{2\psi_{d,n}^{(k)}+4\arrowvert z\arrowvert^2}{\Im(z)^2}\right)^2,
\label{eq: stieltjes2}
\end{align}
where inequality \eqref{eq: stieltjes1} follows from Lemma \ref{lemma2.6} and \eqref{eq: stieltjes2} results from \eqref{eq: bound3}.
Furthermore, with 
\begin{equation}
D_{d,n}^{(k)}(z)= \frac{1}{1+\frac{d}{n}e_{d,n}^{(k)}(z)}R_{d,n}-S_{d,n}-zI_{d\times d},
\end{equation}
it follows from Lemma \ref{lemma8} that
\begin{equation}
\left\Arrowvert D_{d,n}^{-1}(z)\right\Arrowvert_{S_{\infty}} \leq \frac{1}{\Im z}\ \  \ \text{as well as}\ \ \ \left\Arrowvert \left(D_{d,n}^{(k)}(z)\right)^{-1}\right\Arrowvert_{S_{\infty}} \leq \frac{1}{\Im z}.\label{eq: spectralnormD}
\end{equation}
We begin with establishing \eqref{eq: Konvergenz}. To this aim, let 
\begin{align*}
E_{x,d,n}=\begin{cases}I_{d\times d} &\text{for }x=m,\\ R_{d,n} &\text{for }x=e.\end{cases}
\end{align*}
 We  decompose 
$$
f_{k,x}=f_{k,x}^{[1]}+f_{k,x}^{[2]}+f_{k,x}^{[3]}+f_{k,x}^{[4]},
$$ 
where
\begin{align*}
f_{k,x}^{[1]}&= \frac{1}{d}\frac{Z_{k,d,n}^\ast R_{d,n}^{1/2}\hat G_{d,n}^{(k)}(z)E_{x,d,n}D_{d,n}^{-1}(z)R_{d,n}^{1/2}Z_{k,d,n}}{1+Y_{k,d,n}^\ast \hat G_{d,n}^{(k)}(z)Y_{k,d,n}}\\ 
&\quad\quad\quad\quad -\ \frac{1}{d}\frac{Z_{k,d,n}^\ast R_{d,n}^{1/2}\hat G_{d,n}^{(k)}(z)E_{x,d,n}\left(D_{d,n}^{(k)}(z)\right)^{-1}R_{d,n}^{1/2}Z_{k,d,n}}{1+Y_{k,d,n}^\ast \hat G_{d,n}^{(k)}(z)Y_{k,d,n}},\\
f_{k,x}^{[2]} &= \frac{1}{d}\frac{Z_{k,d,n}^\ast R_{d,n}^{1/2}\hat G_{d,n}^{(k)}(z)E_{x,d,n}\left(D_{d,n}^{(k)}(z)\right)^{-1}R_{d,n}^{1/2}Z_{k,d,n}}{1+Y_{k,d,n}^\ast \hat G_{d,n}^{(k)}(z)Y_{k,d,n}}\\
&\quad\quad\quad\quad -\ \frac{1}{d}\frac{\tr\left(R_{d,n}\hat G_{d,n}^{(k)}(z)E_{x,d,n}\left(D_{d,n}^{(k)}(z)\right)^{-1}\right)}{1+Y_{k,d,n}^\ast \hat G_{d,n}^{(k)}(z)Y_{k,d,n}},\\
f_{k,x}^{[3]} &= \frac{1}{d}\frac{\tr\left(R_{d,n}\hat G_{d,n}^{(k)}(z)E_{x,d,n}\left(D_{d,n}^{(k)}(z)\right)^{-1}\right)}{1+Y_{k,d,n}^\ast \hat G_{d,n}^{(k)}(z)Y_{k,d,n}}\\
&\quad\quad\quad\quad -\  \frac{1}{d}\frac{\tr\left(R_{d,n}\hat G_{d,n}(z)E_{x,d,n}D_{d,n}^{-1}(z)\right)}{1+Y_{k,d,n}^\ast \hat G_{d,n}^{(k)}(z)Y_{k,d,n}},\\
f_{k,x}^{[4]} &= \frac{1}{d}\frac{\tr\left(R_{d,n}\hat G_{d,n}(z)E_{x,d,n}D_{d,n}^{-1}(z)\right)}{1+Y_{k,d,n}^\ast \hat G_{d,n}^{(k)}(z)Y_{k,d,n}}\\
&\quad\quad\quad\quad  -\  \frac{1}{d}\frac{\tr\left(R_{d,n}\hat G_{d,n}(z)E_{x,d,n}D_{d,n}^{-1}(z)\right)}{1+\frac{d}{n}e_{d,n}(z)}.
\end{align*}
Using Lemma \ref{lemma4} in \eqref{eq: lemma4b} as well as the spectral norm bounds \eqref{eq: stieltjes2}, \eqref{eq: 22} and \eqref{eq: spectralnormD} in \eqref{eq: 23}, we obtain
\begin{align}
\left\arrowvert f_{k,x}^{[1]}\right\arrowvert &= \left\arrowvert\frac{1}{d}\frac{Z_{k,d,n}^\ast R_{d,n}^{1/2}\hat G_{d,n}^{(k)}(z)E_{x,d,n}\left[D_{d,n}^{-1}(z)-\left(D_{d,n}^{(k)}(z)\right)^{-1}\right]R_{d,n}^{1/2}Z_{k,d,n}}{1+Y_{k,d,n}^\ast \hat G_{d,n}^{(k)}(z)Y_{k,d,n}}\right\arrowvert\nonumber\\
&= \left\arrowvert \frac{n}{d}Y_{k,d,n}^\ast \hat G_{d,n}(z)E_{x,d,n}\left[D_{d,n}^{-1}(z)-\left(D_{d,n}^{(k)}(z) \right)^{-1}\right]Y_{k,d,n}\right\arrowvert\label{eq: lemma4b}\\
&= \bigg\arrowvert \frac{n}{d}Y_{k,d,n}^\ast \hat G_{d,n}(z)E_{x,d,n}\left(D_{d,n}^{(k)}(z) \right)^{-1}\notag\\
&\quad\quad\quad\times\left[D_{d,n}^{(k)}(z) -D_{d,n}(z)\right]D_{d,n}^{-1}(z)Y_{k,d,n}\bigg\arrowvert\nonumber\\
&\leq \frac{n}{d}\Arrowvert Y_{k,d,n}\Arrowvert_2^2\left\Arrowvert \hat G_{d,n}(z)\right\Arrowvert_{S_{\infty}}\left\Arrowvert E_{x,d,n}\right\Arrowvert_{S_{\infty}}\label{eq: lemma5}\\
&\quad\quad\quad\times \left\Arrowvert D_{d,n}^{-1}(z)\right\Arrowvert_{S_{\infty}}\left\Arrowvert D_{d,n}^{(k)}(z) -D_{d,n}\right\Arrowvert_{S_{\infty}}\left\Arrowvert \left(D_{d,n}^{(k)}(z)\right)^{-1} \right\Arrowvert_{S_{\infty}}\nonumber\\
&\leq \frac{1}{d} \Arrowvert Y_{k,d,n}\Arrowvert_2^2 \frac{\left(2\psi_{d,n}^{(k)}+4|z|^2\right)^2\Arrowvert R_{d,n}\Arrowvert_{S_{\infty}}^2\left\Arrowvert E_{x,d,n}\right\Arrowvert_{S_{\infty}}}{(\Im z)^8}\label{eq: 23}\\
&\le  \frac{1}{dn}  \Arrowvert Z_{k,d,n}\Arrowvert_2^2\frac{\left(2\psi_{d,n}^{(k)}+4|z|^2\right)^2\Arrowvert R_{d,n}\Arrowvert_{S_{\infty}}^3\left\Arrowvert E_{x,d,n}\right\Arrowvert_{S_{\infty}}}{(\Im z)^8}.\nonumber
\end{align}
By \eqref{eq: Stieltjes5},
\begin{align*}
\left\arrowvert f_{k,x}^{[2]}\right\arrowvert &= \left\arrowvert\frac{1}{d}\frac{\tr \left[ \left(R_{d,n}^{1/2}Z_{k,d,n}Z_{k,d,n}^\ast R_{d,n}^{1/2}-R_{d,n}\right)\hat G_{d,n}^{(k)}(z)E_{x,d,n}\left(D_{d,n}^{(k)}(z)\right)^{-1}\right]}{1+Y_{k,d,n}^\ast \hat G_{d,n}^{(k)}(z)Y_{k,d,n}}\right\arrowvert\\
&\leq \frac{2\max_i\left\arrowvert\lambda_{i,d,n}^{(k)}\right\arrowvert^2+4\arrowvert z\arrowvert^2}{\Im(z)^2}\\
&\ \ \ \ \ \times\frac{1}{d}\left\arrowvert \tr \left[ \left(R_{d,n}^{1/2}Z_{k,d,n}Z_{k,d,n}^\ast R_{d,n}^{1/2}-R_{d,n}\right)\hat G_{d,n}^{(k)}(z)E_{x,d,n}\left(D_{d,n}^{(k)}(z)\right)^{-1}\right]\right\arrowvert.
\end{align*}
Furthermore, using \eqref{eq: Stieltjes5} in \eqref{eq: in1}, the invariance of the trace under cyclic permutation and Lemma \ref{lemma2.6} in \eqref{eq: 24} for the first term in the curly brackets and the spectral norm bounds \eqref{eq: stieltjes2}, \eqref{eq: 22} and \eqref{eq: spectralnormD} in \eqref{eq: in2} yields the bound
\begin{align}
\left\arrowvert f_{k,x}^{[3]}\right\arrowvert &= \left\arrowvert\frac{1}{d}\frac{\tr\left[R_{d,n}\left(\hat G_{d,n}^{(k)}(z)E_{x,d,n}\left(D_{d,n}^{(k)}(z)\right)^{-1}-\hat G_{d,n}(z)E_{x,d,n}D_{d,n}^{-1}(z)\right)\right]}{1+Y_{k,d,n}^\ast \hat G_{d,n}^{(k)}(z)Y_{k,d,n}}\right\arrowvert\nonumber\\
&\leq \frac{2\max_i\left\arrowvert\lambda_{i,d,n}^{(k)}\right\arrowvert^2+4\arrowvert z\arrowvert^2}{\Im(z)^2}\label{eq: in1}\\
&\ \ \ \ \times\left\{\frac{1}{d}\left\arrowvert \tr\left[R_{d,n}\left(\hat G_{d,n}^{(k)}(z)-\hat G_{d,n}(z)\right)E_{x,d,n}\left(D_{d,n}^{(k)}(z)\right)^{-1}\right]\right\arrowvert\right.\notag\\
&\hspace{1.7cm}+ \left.\frac{1}{d}\left\arrowvert \tr\left[R_{d,n}\hat G_{d,n}(z)E_{x,d,n}\left(\left(D_{d,n}^{(k)}(z)\right)^{-1}-D_{d,n}^{-1}(z)\right)\right]\right\arrowvert\right\}\nonumber\\
&\leq  \frac{2\max_i\left\arrowvert\lambda_{i,d,n}^{(k)}\right\arrowvert^2+4\arrowvert z\arrowvert^2}{\Im(z)^2}\left\{\frac{1}{d}\frac{\left\Arrowvert E_{x,d,n}\left(D_{d,n}^{(k)}(z)\right)^{-1}R_{d,n}\right\Arrowvert_{S_{\infty}}}{\Im z}\right.\label{eq: 24}\\
&\hspace{-25mm}\left.\phantom{\frac{\left\Arrowvert \left(D_d^{(k)}\right)^{-1}\right\Arrowvert_{S_1}}{\Im z}}+ \frac{1}{d}\left\arrowvert \tr\left[R_{d,n}\hat G_{d,n}(z)E_{x,d,n}D_{d,n}^{-1}(z)\left[D_{d,n}(z)-D_{d,n}^{(k)}(z)\right]\left(D_{d,n}^{(k)}(z)\right)^{-1}\right]\right\arrowvert\right\}
\nonumber\\
&\leq  \frac{1}{d} \frac{2\max_i\left\arrowvert\lambda_{i,d,n}^{(k)}\right\arrowvert^2+4\arrowvert z\arrowvert^2}{\Im(z)^2}\frac{1}{(\Im z)^2}\left\Arrowvert R_{d,n}\right\Arrowvert_{S_{\infty}}\left\Arrowvert E_{x,d,n}\right\Arrowvert_{S_{\infty}}\label{eq: in2}\\
&\hspace{1.7cm} +\frac{1}{n} \left(\frac{2\psi_{d,n}^{(k)}+4\arrowvert z\arrowvert^2}{\Im(z)^2}\right)^3\frac{\Arrowvert R_{d,n}\Arrowvert_{S_{\infty}}^{3}}{(\Im z)^4}\left\Arrowvert E_{x,d,n}\right\Arrowvert_{S_{\infty}}.\nonumber
\end{align}
Finally, using \eqref{eq: 22} and \eqref{eq: spectralnormD} in \eqref{eq: 50},   \eqref{eq: Stieltjes5} and \eqref{eq: bound3} in \eqref{eq: 51} and Lemma \ref{lemma2.6} in \eqref{eq: nuller},
\begin{align}
\left\arrowvert f_{k,x}^{[4]}\right\arrowvert &= \frac{1}{d}\left\arrowvert \tr\left(R_{d,n}\hat G_{d,n}(z)E_{x,d,n}D_{d,n}^{-1}(z)\right)\right\arrowvert\nonumber\\
&\ \ \ \ \ \times \left\arrowvert \frac{n^{-1}\tr\left[R_{d,n}^{1/2}\hat G_{d,n}(z)R_{d,n}^{1/2}\right]-n^{-1}Z_{k,d,n}^\ast R_{d,n}^{1/2}\hat G_{d,n}^{(k)}(z)R_{d,n}^{1/2}Z_{k,d,n}}{\left(1+Y_{k,d,n}^\ast \hat G_{d,n}^{(k)}(z)Y_{k,d,n}\right)\left(1+\frac{d}{n}e_{d,n}(z)\right)}\right\arrowvert\nonumber\\
&\leq \frac{1}{d}\frac{\left\Arrowvert R_{d,n}\right\Arrowvert_{S_1}\left\Arrowvert E_{x,d,n}\right\Arrowvert_{S_{\infty}}}{(\Im z)^2}\label{eq: 50}\\
&\ \ \ \ \ \times \left\arrowvert \frac{n^{-1}\tr\left[R_{d,n}^{1/2}\hat G_{d,n}(z)R_{d,n}^{1/2}\right]-n^{-1}Z_{k,d,n}^\ast R_{d,n}^{1/2}\hat G_{d,n}^{(k)}(z)R_{d,n}^{1/2}Z_{k,d,n}}{\left(1+Y_{k,d,n}^\ast \hat G_{d,n}^{(k)}(z)Y_{k,d,n}\right)\left(1+\frac{d}{n}e_{d,n}(z)\right)}\right\arrowvert\nonumber\\
&\leq\ \frac{1}{d}\frac{\left\Arrowvert R_{d,n}\right\Arrowvert_{S_1}\left\Arrowvert E_{x,d,n}\right\Arrowvert_{S_{\infty}}}{(\Im z)^2}\left(\frac{2\left(\psi_{d,n}^{(k)}\right)^2+4\arrowvert z\arrowvert^2}{\Im(z)^2}\right)^2\label{eq: 51}\\
&\ \ \ \ \ \times \left\arrowvert \frac{1}{n}\tr\left[R_{d,n}^{1/2}\hat G_{d,n}(z)R_{d,n}^{1/2}\right]-\frac{1}{n}Z_{k,d,n}^\ast R_{d,n}^{1/2}\hat G_{d,n}^{(k)}(z)R_{d,n}^{1/2}Z_{k,d,n}\right\arrowvert\nonumber\\
&\leq\ \frac{1}{d}\frac{\left\Arrowvert R_{d,n}\right\Arrowvert_{S_1}\left\Arrowvert E_{x,d,n}\right\Arrowvert_{S_{\infty}}}{(\Im z)^2}\left(\frac{2\left(\psi_{d,n}^{(k)}\right)^2+4\arrowvert z\arrowvert^2}{\Im(z)^2}\right)^2\label{eq: nuller}\\
&\ \ \ \ \ \times \Bigg\{ \left\arrowvert \frac{1}{n}\tr\left[R_{d,n}^{1/2}\hat G_{d,n}^{(k)}(z)R_{d,n}^{1/2}\right]-\frac{1}{n}Z_{k,d,n}^\ast R_{d,n}^{1/2}\hat G_{d,n}^{(k)}(z)R_{d,n}^{1/2}Z_{k,d,n}\right\arrowvert \nonumber\\
&\hspace{2cm}+\frac{1}{n}\frac{\Arrowvert R_{d,n}\Arrowvert_{S_\infty}}{\Im z}\Bigg\}\notag
\end{align}
Based on these estimates on $f_{k,x}^{[l]}$, $l=1,2,3,4$, we are ready to prove \eqref{eq: intermediate1} and \eqref{eq: intermediate2}. In the next display, $c>0$ denotes a constant depending only on the support of $Z_{11}$, and may change from line to line. By means of Lemma \ref{lemma: moments}, Lemma \ref{lemma: eigenvalue}, Lemma \ref{lemma7} and the spectral norm bounds \eqref{eq: 22} and \eqref{eq: spectralnormD},
\begin{align*}
\E \left\arrowvert f_{k,x}^{[1]}\right\arrowvert^6 &\leq \frac{\Arrowvert R_{d,n}\Arrowvert_{S_{\infty}}^{18}\left\Arrowvert E_{x,d,n}\right\Arrowvert_{S_{\infty}}^{6}}{{n^6d^6(\Im z)^{48}}}\E\left\{\Arrowvert Z_{k,d,n}\Arrowvert_2^{12}
\left(2\psi_{d,n}^{(k)}+4|z|^2\right)^6 \right\}\\
& \le  2^{17}\frac{\Arrowvert R_{d,n}\Arrowvert_{S_{\infty}}^{18}\left\Arrowvert E_{x,d,n}\right\Arrowvert_{S_{\infty}}^{6}}{{d^6n^6(\Im z)^{48}}}\Bigg[\Big(\E\Arrowvert Z_{k,d,n}\Arrowvert_2^{24}\Big)^{1/2}\\
&\hspace{0.5cm}\times\Bigg(\E\Bigg(\Arrowvert S_{d,n}\Arrowvert_{S_\infty}^2+\max \Bigg\{\Bigg\Arrowvert \sum_{l=1}^nY_lY_l^\ast \Bigg\Arrowvert_{S_\infty}^2,\Bigg\Arrowvert \sum_{\substack{l=1 \\ l\neq k}}^nY_lY_l^\ast \Bigg\Arrowvert_{S_\infty}^2\Bigg\}\Bigg)^{12}\Bigg)^{1/2}\\
&\hspace{5cm}+|z|^{12}\Bigg) \Bigg]\\
& \le c \frac{\Arrowvert R_{d,n}\Arrowvert_{S_{\infty}}^{18}\left\Arrowvert E_{x,d,n}\right\Arrowvert_{S_{\infty}}^{6}}{{n^6(\Im z)^{48}}}\Big(\Arrowvert S_{d,n} \Arrowvert_{S_\infty}^{12}+\Arrowvert R_{d,n} \Arrowvert_{S_\infty}^{12} +|z|^{12} \Big)\\
\E \left\arrowvert f_{k,x}^{[2]}\right\arrowvert^6 &\le  \frac{c}{d^3}\E\Bigg\{\Bigg(\frac{2\max_i\left\arrowvert\lambda_{i,d,n}^{(k)}\right\arrowvert^2+4\arrowvert z\arrowvert^2}{\Im(z)^2}\Bigg)^6\\
&\hspace{3cm}\times\left\Arrowvert R_{d,n}^{1/2}\hat G_{d,n}^{(k)}(z)E_{x,d,n}\left(D_{d,n}^{(k)}(z)\right)^{-1}R_{d,n}^{1/2}\right\Arrowvert_{S_{\infty}}^6\Bigg\}\\
&\leq  \frac{c}{d^3(\Im z)^{24}}\left\Arrowvert R_{d,n}\right\Arrowvert_{S_{\infty}}^6\left\Arrowvert E_{x,d,n}\right\Arrowvert_{S_{\infty}}^6\Big(\Arrowvert S_{d,n} \Arrowvert_{S_\infty}^{12}+\Arrowvert R_{d,n} \Arrowvert_{S_\infty}^{12} +|z|^{12} \Big),\\
\E \left\arrowvert f_{k,x}^{[3]}\right\arrowvert^6 &\leq  \frac{c}{(\Im z)^{24}d^6}\left\Arrowvert R_{d,n}\right\Arrowvert_{S_{\infty}}^{6}\left\Arrowvert E_{x,d,n}\right\Arrowvert_{S_{\infty}}^6\Big(\Arrowvert S_{d,n} \Arrowvert_{S_\infty}^{12}+\Arrowvert R_{d,n} \Arrowvert_{S_\infty}^{12} +|z|^{12} \Big)\\
&\hspace{0.5cm}+ \frac{c}{(\Im z)^{60}n^6}\left\Arrowvert R_{d,n}\right\Arrowvert_{S_{\infty}}^{18}\left\Arrowvert E_{x,d,n}\right\Arrowvert_{S_{\infty}}^6\Big(\Arrowvert S_{d,n} \Arrowvert_{S_\infty}^{36}+\Arrowvert R_{d,n} \Arrowvert_{S_\infty}^{36} +|z|^{36} \Big),\\
\E \left\arrowvert f_{k,x}^{[4]}\right\arrowvert^6 &\leq \frac{cd^3\left\Arrowvert R_{d,n}\right\Arrowvert_{S_{\infty}}^{12}\left\Arrowvert E_{x,d,n}\right\Arrowvert_{S_{\infty}}^6}{n^6(\Im z)^{42}}\Big(\Arrowvert S_{d,n} \Arrowvert_{S_\infty}^{24}+\Arrowvert R_{d,n} \Arrowvert_{S_\infty}^{24} +|z|^{24}\Big).
\end{align*}
In order to show finally \eqref{eq: Konvergenz}, it remains to note that for any $\varepsilon>0$,
\begin{align*}
\sum_{d=1}^{\infty}\P\left(\left\arrowvert \frac{1}{n}\sum_{k=1}^nf_{k,x}\right\arrowvert >\varepsilon\right)&\leq \sum_{d=1}^{\infty}\sum_{k=1}^n \sum_{l=1}^4 \P\left(\left\arrowvert f_{k,x}^{[l]}\right\arrowvert >\varepsilon/4\right)\\
&\leq  \sum_{d=1}^{\infty}\sum_{k=1}^n \sum_{l=1}^4\left(\frac{\varepsilon}{4}\right)^{-6}\E\left\arrowvert f_{k,x}^{[l]}\right\arrowvert^6< \infty
\end{align*}
by an application of the union bound, Markov's inequality, and \eqref{dn}. \eqref{eq: Konvergenz} is then a consequence of the Borel-Cantelli lemma.

\subsection{Step III: Existence and uniqueness of $e_{d,n}^{\circ}$}
We show that for any $d,n$ and $R_{d,n}$, there exists a unique $e(z)\in\C^+$ which solves the fixed point equation
\begin{equation}\label{eq: fixed point}
e^\circ_{d,n}(z) = \frac{1}{d}\tr\left\{R_{d,n}\left(\frac{1}{1+\frac{d}{n}e^\circ_{d,n}(z)}R_{d,n}-S_{d,n}-z I_{d\times d}\right)^{-1}\right\}, \ \ z\in\C^+.
\end{equation}
To this end, define for any fixed $d,n$ the subsequences $(d_l)_{l\in\N}$ and $(n_l)_{l\in\N}$, where $d_l=ld$ and $n_l=ln$, $l\in\N$, and correspondingly the $l$-block diagonal matrices
$$
R_{(d,n)_l} = \mathrm{diag}\left(R_{d,n},...., R_{d,n}\right)\ \ \ \text{and}\ \ \ S_{(d,n)_l} = \mathrm{diag}\left(S_{d,n},...., S_{d,n}\right)
$$
of size $dl\times dl$. Note that the right-hand side of \eqref{eq: fixed point} remains unchanged when
replacing $d,n,R_{d,n}$ and $I_{d\times d}$ by $d_l,n_l,R_{(d,n)_l}$ and $I_{d_l\times d_l}$. By \eqref{eq: intermediate2} of the previous section,
\begin{align*}
e_{(d,n)_l}(z)-\frac{1}{d_l}\tr\left\{R_{(d,n)_l}\left(\frac{1}{1+\frac{d_l}{n_l}e_{(d,n)_l}(z)}R_{(d,n)_l}-S_{(d,n)_l}-z I_{d_l\times d_l}\right)^{-1}\right\} \rightarrow 0
\end{align*}
a.s. as $l\to\infty$ with
$$
e_{(d,n)_l}(z) = \frac{1}{d_l}\tr\left\{R_{(d,n)_l}\left(\hat \Xi_{(d,n)_l}-z I_{d_l\times d_l}\right)^{-1}\right\},
$$
where
$$
\hat \Xi_{(d,n)_l}=\frac{1}{nl}\sum_{k=1}^{nl}R_{(d,n)_l}^{1/2}\check Z_{k,d_l,n_l}\check Z_{k,d_l,n_l}^\ast R_{(d,n)_l}^{1/2}-S_{(d,n)_l},
$$
$\check Z=(\check Z_{ik})_{i,k\in\N}$ is a double array of iid Rademacher variables, and $\check Z_{k,d,n}$ is the $k$-th column of the submatrix $\check Z_{d,n}=(\check Z_{ik,d,n})_{i\le d,~k\le n}.$
Consider a realization of these random variables where this convergence occurs. First note by \eqref{eq: stieltjes3},
$$
\left\arrowvert e_{(d,n)_l}(z)\right\arrowvert \ \leq\ \frac{\kappa}{\Im z}\ \ \ \forall\ l\in\N.
$$
By Bolzano-Weierstra{\ss}, there exists a convergent subsequence of $(e_{(d,n)_l})$ with limit $e(z)$, say, such that in particular
\begin{equation}\label{eq: ek-e}
\frac{1}{1+\frac{d_l}{n_l}e_{(d,n)_l}(z)}\ \ \rightarrow\ \ \frac{1}{1+\frac{d}{n}e(z)}
\end{equation}
along this subsequence due to \eqref{eq: Stieltjes16} for $e_{(d,n)_l}(z)$. By \eqref{eq: intermediate2}, $e(z)$ solves the fixed point equation \eqref{eq: fixed point}. As $\Im\left(e_{(d,n)_l}(z)\right)>0$ for any $l\in \N$ and $z\in\C^+$, it follows that its limit satisfies $\Im\left(e(z)\right)\geq 0$ and therefore $\Im\left(e(z)\right)>0$, because $\Im (e(z))=0$ contradicts with $e(z)$ being a solution of the fixed point equation. Consequently, any such solution $e$ of  \eqref{eq: fixed point} enjoys the following two properties:
\begin{equation}
e: \C^+\rightarrow\C^+
\end{equation}
and
\begin{equation}
\arrowvert e(z)\arrowvert \leq \frac{\kappa}{\Im z}\ \ \ \forall\ z\in\C^+.\label{eq: e circ im bound}
\end{equation}
It remains to show uniqueness. Denoting
\begin{equation}\label{eq: D}
\tilde{D}_{d,n}(z) = \tilde{D}_{d,n}(z,e(z))  = \frac{1}{1+\frac{d}{n}e(z)}R_{d,n}-S_{d,n}-zI_{d\times d},
\end{equation}
we obtain the representation
\begin{align*}
e(z) &=\frac{1}{d}\tr\left(\tilde{D}_{d,n}^{-1}(z)R_{d,n}\right)\\
&= \frac{1}{d} \tr\left(\tilde{D}_{d,n}^{-1}(z)R_{d,n}\left(\tilde{D}_{d,n}^{\ast}(z)\right)^{-1} \left[\frac{1}{1+\frac{d}{n}e(z)^*}R_{d,n}-S_{d,n}-z^*I_{d\times d}\right]\right).
\end{align*}
Note that $
(A^\ast)^{-1} = (A^{-1})^\ast.
$
Now, the expression
\begin{align}\label{eq: >0}
\tr\left(\tilde{D}_{d,n}^{-1}(z)R_{d,n}\left(\tilde{D}_{d,n}^\ast(z)\right)^{-1}S_{d,n}\right) 
\geq 0
\end{align}
is in particular real because the trace of the product of two positive semidefinite Hermitian matrices is non-negative. Hence
\begin{align*}
\Im (e(z)) &= \frac{1}{d}\Im\left(\tr\left\{\tilde{D}_{d,n}^{-1}(z)R_{d,n}\left(\tilde{D}_{d,n}^\ast(z)\right)^{-1}\left(\frac{1}{1+\frac{d}{n}e(z)^*}R_{d,n}-z^*I_{d\times d}\right)\right\}\right)\\
&= \frac{1}{d}\tr\left\{\tilde{D}_{d,n}^{-1}(z)R_{d,n}\left(\tilde{D}_{d,n}^\ast(z)\right)^{-1}\left(\Im\left(\frac{1}{1+\frac{d}{n}e(z)^*}\right)R_{d,n}-\Im\left(z^*\right)I_{d\times d}\right)\right\}\\
&= \frac{1}{d}\tr\left\{\tilde{D}_{d,n}^{-1}(z)R_{d,n}\left(\tilde{D}_{d,n}^\ast(z)\right)^{-1}\left(\frac{(d/n)\Im(e(z))}{\left\arrowvert1+\frac{d}{n}e(z)\right\arrowvert^2}R_{d,n}+\Im\left(z\right)I_{d\times d}\right)\right\}\\
&= \alpha(e(z)) \Im(e(z)) +\beta(e(z)) \Im(z)
\end{align*}
with 
\begin{align*}
\alpha(e(z)) &=  \frac{1}{n}\left\arrowvert1+\frac{d}{n}e(z)\right\arrowvert^{-2}\tr\left\{\tilde{D}_{d,n}^{-1}(z)R_{d,n}\left(\tilde{D}_{d,n}^\ast(z)\right)^{-1}R_{d,n}\right\}\\
\beta(e(z)) &= \frac{1}{d}\tr\left\{\tilde{D}_{d,n}^{-1}(z)R_{d,n}\left(\tilde{D}_{d,n}^\ast(z)\right)^{-1}\right\}
\end{align*}
Note that both, $\alpha$ and $\beta$, are non-negative, and $\alpha(e(z))>0$ implies $\beta(e(z))>0$ since the trace of a positive semidefinite Hermitian matrix equals zero only for the null matrix.
If $\bar{e}(z)$ is another solution of \eqref{eq: fixed point}, we obtain the analogous identity
$$
\Im\left(\bar{e}(z)\right) = \alpha\left(\bar{e}(z)\right)\Im\left(\bar{e}(z)\right)\ +\ \beta\left(\bar{e}(z)\right)\Im\left(z\right).
$$
We denote by $\bar{D}_{d,n}(z)$ the matrix $\tilde{D}_{d,n}(z)$ as defined in \eqref{eq: D} with $\bar{e}(z)$ in place of $e(z)$, and define $\alpha(\bar{e}(z))$ and $\beta(\bar{e}(z))$ correspondingly.  Then
\begin{align}
e(z)-\bar{e}(z) &=\ \frac{1}{d}\tr\left\{\left(\tilde{D}_{d,n}^{-1}(z)-\bar{D}_{d,n}^{-1}(z)\right)R_{d,n}\right\}\notag\\
&= \frac{1}{d}\tr\left\{\tilde{D}_{d,n}^{-1}(z)\left(\bar{D}_{d,n}(z)-\tilde{D}_{d,n}(z)\right)\bar{D}_{d,n}^{-1}(z)R_{d,n}\right\}\notag\\
&= \frac{1}{d}\tr\left\{\tilde{D}_{d,n}^{-1}(z)\frac{\left(1+\frac{d}{n}{e(z)}\right)-\left(1+\frac{d}{n}\bar{e}(z)\right)}{\left(1+\frac{d}{n}\bar{e}(z)\right)\left(1+\frac{d}{n}{e(z)}\right)}R_{d,n}\bar{D}_{d,n}^{-1}(z)R_{d,n}\right\}\notag\\
&=  \left(e(z)-\bar{e}(z)\right)\frac{d/n}{\left(1+\frac{d}{n}\bar{e}(z)\right)\left(1+\frac{d}{n}e(z)\right)}\frac{1}{d}\tr\left\{\tilde{D}_{d,n}^{-1}(z)R_{d,n}\bar{D}_{d,n}^{-1}(z)R_{d,n}\right\}\notag\\
&=:  \left(e(z)-\bar{e}(z)\right)\gamma. \label{eq: gamma}
\end{align}
If $\gamma=0$, uniqueness of $e(z)$ follows immediately. In case $\gamma\not=0$, we deduce the inequality
\begin{align}
\arrowvert\gamma\arrowvert &\leq \left[\frac{d/n}{\left\arrowvert 1+\frac{d}{n}e(z)\right\arrowvert^2}\frac{1}{d}\tr\left\{\tilde{D}_{d,n}^{-1}(z)R_{d,n}\left(\tilde{D}^\ast_{d,n}(z)\right)^{-1}R_{d,n}\right\}\right]^{1/2}\notag \\
&\hspace{2cm} \times \left[\frac{d/n}{\left\arrowvert 1+\frac{d}{n}\bar{e}(z)\right\arrowvert^2}\frac{1}{d}\tr\left\{\bar{D}_{d,n}^{-1}(z)R_{d,n}\left(\bar{D}^\ast_{d,n}(z)\right)^{-1}R_{d,n}\right\}\right]^{1/2}\notag \\
&= \sqrt{\alpha(e(z))}\cdot\sqrt{\alpha(\bar{e}(z))}\label{eq: alpha}\\
&=  \left(\frac{\Im(e(z))\alpha(e(z))}{\Im(e(z))\alpha(e(z))+\Im(z)\beta(e(z))}\right)^{1/2}\notag\\
&\hspace{2cm}\times \left(\frac{\Im(\bar{e}(z))\alpha(\bar{e}(z))}{\Im(\bar{e}(z))\alpha(\bar{e}(z))+\Im(z)\beta(\bar{e}(z))}\right)^{1/2}\notag.
\end{align}
But $\beta(e(z)),\beta(\bar{e}(z))>0$ for $\alpha(e(z)),\alpha(\bar e(z))>0$ which implies $\arrowvert \gamma\arrowvert<1$ and therefore, $e=\bar{e}$.

\subsection{Step IV: Identification of $e_{d,n}^\circ$ and $m_{d,n}^\circ$ as Stieltjes transforms}
As concerns $e_{d,n}^{\circ}$, we know already that $e_{d,n}^{\circ}:\C^+\rightarrow \C^+$. Its analyticity follows by the analyticity of the pointwise approximating sequence $e_{(d,n)_l}$ and the local boundedness of $(e_{(d,n)_l})$ on $\C^+$. Note that the pointwise convergence occurs simultaneously on a countable set with a accumulation point in $\C^+$ with probability 1. Using on the right hand side of \eqref{eq: fixed point} the fact that $e_{d,n}^{\circ}(z)\rightarrow 0$ as $\Im(z)\rightarrow\infty$ which follows from \eqref{eq: e circ im bound}, we also have
\begin{align*}
z\cdot e_{d,n}^{\circ}(z)\ \rightarrow \ -\frac{1}{d}\tr\left(R_{d,n}\right)\ \ \ \text{as $\Im(z),\Re(z)\rightarrow\infty$}.
\end{align*}
Hence, Lemma 2.2 in \cite{Shohat1943} implies that $e_{d,n}^{\circ}$ is the Stieltjes transform of a measure on the real line with total mass $d^{-1}\tr(R_{d,n})$.\\
Define
\begin{equation}
D_{d,n}^{\circ}(z) = \frac{1}{1+\frac{d}{n}e_{d,n}^{\circ}(z)}R_{d,n}-S_{d,n}-zI_{d\times d}.\label{eq: def D-circ}
\end{equation}
 Finally, observe that for any $z\in\C^+$,
\begin{align}
\Im\left(m_{d,n}^{\circ}(z)\right) &= \frac{1}{d}\Im\tr\left\{D_{d,n}^\circ(z)^{-1}\left(\left(D_{d,n}^\circ(z)\right)^\ast\right)^{-1}\left(D_{d,n}^\circ(z)\right)^\ast\right\} \notag \\
&= \frac{1}{d}\Im\tr\Bigg\{D_{d,n}^\circ(z)^{-1}\left(\left(D_{d,n}^\circ(z)\right)^\ast\right)^{-1}\notag\\
&\hspace{2cm}\times\Bigg(\frac{1}{1+\frac{d}{n}(e_{d,n}^{\circ}(z))^\ast}R_{d,n}-z^\ast I_{d\times d}\Bigg)\Bigg\} \notag \\
&=\frac{1}{n}\frac{\Im\left(e_{d,n}^\circ(z)\right)}{\left|1+\frac{d}{n}e_{d,n}^\circ(z)\right|^2}\tr\left(D_{d,n}^\circ(z)^{-1}\left(\left(D_{d,n}^\circ(z)\right)^\ast\right)^{-1}R_{d,n}  \right)\notag\\
&\hspace{1cm}+\frac{1}{d}\Im(z)\tr\left(D_{d,n}^\circ(z)^{-1}\left(\left(D_{d,n}^\circ(z)\right)^\ast\right)^{-1}\right)\notag\\
&>0 \label{eq: im m}
\end{align}
since both $\Im(z)$ and $\Im\left(e_{d,n}^{\circ}(z)\right)$ are strictly positive. Furthermore, since $e_{d,n}^{\circ}(z)\rightarrow 0$ as $\Im(z)\rightarrow\infty$ by \eqref{eq: e circ im bound}, we conclude 
\begin{align*}
z\cdot m_{d,n}^{\circ}(z) \rightarrow -1\ \ \ \text{as $\Im(z),\Re(z)\rightarrow\infty$}.
\end{align*}
As above, $m_{d,n}^{\circ}$ is the Stieltjes transform of a measure on the real line with total mass $1$. 

\subsection{Step V: Approximation of $e_{d,n}$ by $e_{d,n}^{\circ}$}
Let $e_{d,n}^{\circ}$ denote the solution of \eqref{eq: fixed point}. We will show that for any $z\in\C^+$,
\begin{equation}\label{eq: konvergenz 3}
e_{d,n}(z)-e_{d,n}^{\circ}(z) \rightarrow 0\ \ a.s.\ \ \text{as}\ \ d\rightarrow\infty.
\end{equation}
Define
$$
\alpha^{\circ}(z) = \alpha\left(e_{d,n}^{\circ}(z)\right)\ \ \ \text{and}\ \ \ \beta^{\circ}(z) = \beta\left(e_{d,n}^{\circ}(z)\right)
$$
such that
\begin{equation}\label{eq: circs}
\Im\left(e_{d,n}^{\circ}(z)\right) = \alpha^{\circ}(z)\Im\left(e_{d,n}^{\circ}\right)+\beta^{\circ}(z)\Im(z).
\end{equation}
Noting that
\begin{align*}
\frac{\alpha^{\circ}(z)}{\beta^{\circ}(z)} \leq \Arrowvert R_{d,n}\Arrowvert_{S_{\infty}}\frac{d}{n}\left\arrowvert1+\frac{d}{n}e_{d,n}^{\circ}(z)\right\arrowvert^{-2},
\end{align*}
we deduce
\begin{align}
\Im\left(e_{d,n}^{\circ}(z)\right)\frac{\alpha^{\circ}(z)}{\beta^{\circ}(z)}& \leq \Im\left(e_{d,n}^{\circ}(z)\right)\Arrowvert R_{d,n}\Arrowvert_{S_\infty}\frac{d}{n}\left\arrowvert1+\frac{d}{n}e_{d,n}^{\circ}(z)\right\arrowvert^{-2}\label{eq: doch number}\\
&= - \Arrowvert R_{d,n}\Arrowvert_{S_\infty}\Im\left(\frac{1}{1+\frac{d}{n}e_{d,n}^{\circ}(z)}\right)\nonumber\\
&\le  \Arrowvert R_{d,n}\Arrowvert_{S_\infty} \left\arrowvert \frac{1}{1+\frac{d}{n}e_{d,n}^{\circ}(z)}\right\arrowvert \nonumber\\
&\leq \Arrowvert R_{d,n}\Arrowvert_{S_\infty}\limsup_{l\to\infty}\frac{2\max_i\left\arrowvert \lambda_{i}\left(\hat\Xi_{(d,n)_l}\right)\right\arrowvert^2+4|z|^2}{\Im(z)^2},\label{eq: haesslich}
\end{align}
where the last inequality follows by convergence  \eqref{eq: ek-e} and bound \eqref{eq: bound3} (in the latter the eigenvalues corresponding to $\hat \Xi_{(d,n)_l}$ have to be inserted). As a consequence,
\begin{align}
\alpha^{\circ}(z) &= \left(\frac{\Im\left(e_{d,n}^{\circ}(z)\right)\alpha^{\circ}(z)}{(\Im z)\beta^{\circ}(z) +\Im\left(e_{d,n}^{\circ}(z)\right)\alpha^{\circ}(z)}\right)\label{eq: first id}\\
&\leq \frac{2\Arrowvert R_{d,n}\Arrowvert_{S_\infty}\Arrowvert\hat \Xi_{d,n}\Arrowvert_{S_\infty}^2+4|z|^2}{(\Im z)^3+2\Arrowvert R_{d,n}\Arrowvert_{S_\infty}\Arrowvert\hat \Xi_{d,n}\Arrowvert_{S_\infty}^2+4|z|^2},\label{eq: last in}
\end{align}
where the first identity \eqref{eq: first id} follows by rearrangement of \eqref{eq: circs}, and after expanding the fraction by $(\beta^{\circ}(z))^{-1}$ we used the elementary inequality
$$
\frac{x}{y+x}\ \leq\ \frac{z}{y+z}\ \ \ \text{for $x,y,z>0$ and $x\leq z$}
$$
and \eqref{eq: haesslich} in \eqref{eq: last in}. 
By \eqref{eq: e},
\begin{equation*}
e_{d,n}(z) = \frac{1}{d}\tr\left(R_{d,n}D_{d,n}^{-1}(z)\right)- \frac{1}{n}\sum_{k=1}^nf_{k,e}.
\end{equation*}
Then as previously in \eqref{eq: >0} and the subsequent display, we obtain the representation
\begin{align}
\Im\left(e_{d,n}(z)\right)\ &=\ \frac{1}{d}\Im\left(\frac{1}{1+\frac{d}{n}e_{d,n}^\ast(z)}\right)\tr\left\{{D}_{d,n}^{-1}(z)R_{d,n}\left({D}_{d,n}^\ast(z)\right)^{-1}R_{d,n}\right\} \notag\\
&\hspace{1cm}-\frac{1}{d}\Im\left(z^*\right)\tr\left\{{D}_{d,n}^{-1}(z)R_{d,n}\left({D}_{d,n}^\ast(z)\right)^{-1}\right\}-\frac{1}{n}\sum_{k=1}^n\Im\left(f_{k,e}\right) \notag\\
&=\ \Im\left(e_{d,n}(z)\right)\alpha\left(e_{d,n}(z)\right)+\Im(z)\beta\left(e_{d,n}(z)\right) -\frac{1}{n}\sum_{k=1}^n\Im\left(f_{k,e}\right), \label{eq: alpha fehler}
\end{align}
and as in \eqref{eq: gamma},
\begin{equation}\label{eq: e identity}
e_{d,n}(z)-e_{d,n}^{\circ}(z)=\ \gamma\left(e_{d,n}(z)-e_{d,n}^{\circ}(z)\right)-\frac{1}{n}\sum_{k=1}^nf_{k,e}
\end{equation}
with 
\begin{equation}
\arrowvert\gamma\arrowvert\leq\sqrt{\alpha^{\circ}(z)\alpha(e_{d,n}(z))}.\label{eq: alpha circ alpha}
\end{equation} 
Consider a realization for which the convergence
$$
\frac{1}{n}\sum_{k=1}^nf_{k,e}\ \rightarrow\ 0
$$
occurs. Then in particular,
\begin{equation}\label{eq: beta2}
\left\arrowvert \frac{1}{n}\sum_{k=1}^nf_{k,e}\right\arrowvert  \leq \Im(z)\frac{n}{4d\left(\Arrowvert R_{d,n} \Arrowvert_{S_\infty}\vee 1\right)}  \left(\frac{2\Arrowvert \hat \Xi_{d,n}\Arrowvert_{S_\infty}^2+4\arrowvert z\arrowvert^2}{\Im(z)^2}\right)^{-2}
\end{equation}
for sufficiently large $d$. Recall that by definition of $\alpha(e_{d,n}(z))$ and $\beta(e_{d,n}(z))$,
\begin{align}\label{eq: mal wieder}
\frac{\alpha(e_{d,n}(z))}{\beta(e_{d,n}(z))} \leq \Arrowvert R_{d,n}\Arrowvert_{S_{\infty}}\frac{d}{n}\left\arrowvert1+\frac{d}{n}e_{d,n}(z)\right\arrowvert^{-2}.
\end{align}
Hence, if 
$$
\beta(e_{d,n}(z)) \leq \frac{n}{4d\left(\Arrowvert R_{d,n} \Arrowvert_{S_\infty}\vee 1\right)}  \left(\frac{2\Arrowvert \hat \Xi_{d,n}\Arrowvert_{S_\infty}^2+4\arrowvert z\arrowvert^2}{\Im(z)^2}\right)^{-2}, 
$$
then inserting \eqref{eq: bound3} into
\eqref{eq: mal wieder} yields 
\begin{align*}
\alpha(e_{d,n}(z)) &\leq \Arrowvert R_{d,n} \Arrowvert_{S_\infty}\frac{d}{n}\left(\frac{2\Arrowvert \hat \Xi_{d,n}\Arrowvert_{S_\infty}^2+4\arrowvert z\arrowvert^2}{\Im(z)^2}\right)^{2} \beta(e_{d,n}(z)) \leq \frac{1}{4},
\end{align*}
in which case \eqref{eq: alpha circ alpha} implies $\arrowvert\gamma\arrowvert\leq 1/2$ since $\alpha^{\circ}(z)\leq 1$ by \eqref{eq: first id} and the non-negativity of $\alpha^{\circ}(z),\beta^{\circ}(z)$ and $\Im(e_{d,n}^{\circ}(z))$. Otherwise, if 
$$
\beta(e_{d,n}^\circ(z)) > \frac{n}{4d\left(\Arrowvert R_{d,n} \Arrowvert_{S_\infty}\vee 1\right)}  \left(\frac{2\Arrowvert \hat \Xi_{d,n}\Arrowvert_{S_\infty}^2+4\arrowvert z\arrowvert^2}{\Im(z)^2}\right)^{-2}, 
$$
\eqref{eq: alpha circ alpha}, \eqref{eq: alpha fehler}, \eqref{eq: beta2}, and \eqref{eq: last in} imply
\begin{align}
\arrowvert\gamma\arrowvert &\leq \sqrt{\alpha^{\circ}(z)}\left(\frac{\Im\left(e_{d,n}(z)\right)\alpha(e_{d,n}(z))}{\Im\left(e_{d,n}(z)\right)\alpha(e_{d,n}(z))+\Im(z)\beta(e_{d,n}(z))-\frac{1}{n}\sum_{k=1}^n\Im\left(f_{k,e}\right)}\right)^{1/2}\nonumber\\
&\leq \left(\frac{2\Arrowvert R_{d,n}\Arrowvert_{S_\infty}\Arrowvert\hat \Xi_{d,n}\Arrowvert_{S_\infty}^2+4|z|^2}{(\Im z)^3+2\Arrowvert R_{d,n}\Arrowvert_{S_\infty}\Arrowvert\hat \Xi_{d,n}\Arrowvert_{S_\infty}^2+4|z|^2}\right)^{1/2}.\nonumber
\end{align}
As $d\to\infty$ the limes superior of the last expression is bounded by some positive constant $\tilde\gamma(z)<1$ almost surely. Finally, solving the equation \eqref{eq: e identity} for $e_{d,n}-e_{d,n}^{\circ}$ and using the upper bounds on $\arrowvert\gamma\arrowvert$, we obtain 
\begin{align}
\left\arrowvert e_{d,n}(z)-e_{d,n}^{\circ}(z)\right\arrowvert &\leq \frac{\left\arrowvert \frac{1}{n}\sum_{k=1}^nf_{k,e}\right\arrowvert }{1- \left(\frac{1}{4}\vee\tilde\gamma(z)\right)}\notag \\ 
&\rightarrow\ 0\ \ \text{a.s.} \label{eq: conv e}
\end{align}
as $d\rightarrow\infty$, by \eqref{eq: Konvergenz}. This proves \eqref{eq: konvergenz 3}.

\subsection{Step VI: Approximation of $m_{d,n}$ by $m_{d,n}^{\circ}$}
Without loss of generality we may assume that either
$$\frac{d^{\frac{3}{2}}}{n}>1 \ \ \text{or} \ \ \frac{d^{\frac{3}{2}}}{n}\le 1$$
holds on the whole sequence.
We start with the first case. Recall the definition \eqref{eq: def m-circ} of $m_{d,n}^{\circ}$ and \eqref{eq: def D-circ} of $D_{d,n}^{\circ}(z)$, and note that
$$
m_{d,n}^{\circ}(z)=\frac{1}{d}\tr\left( \left(D_{d,n}^{\circ}(z)\right)^{-1}\right),
$$
while by \eqref{eq: m},
$$
m_{d,n}=\frac{1}{d}\tr\left(\left(D_{d,n}(z)\right)^{-1}\right)-\frac{1}{n}\sum_{k=1}^nf_{k,m}
$$
with
$$
\frac{1}{n}\sum_{k=1}^nf_{k,m}\rightarrow 0 \ \ \text{a.s. as }d\to\infty.
$$
Then,
\begin{align*}
m_{d,n}(z)-m_{d,n}^{\circ}(z) &= \frac{1}{d}\tr \left\{D_{d,n}^{-1}(z)-\left(D_{d,n}^{\circ}(z)\right)^{-1}\right\}-\frac{1}{n}\sum_{k=1}^nf_{k,m}\\
&= \frac{1}{d}\tr \left\{D_{d,n}^{-1}(z)\left(D_{d,n}^{\circ}(z)-D_{d,n}(z)\right)\left(D_{d,n}^{\circ}(z)\right)^{-1}\right\}-\frac{1}{n}\sum_{k=1}^nf_{k,m}\\
&=  \frac{1}{n}\frac{e_{d,n}(z)-e_{d,n}^{\circ}(z)}{(1+\frac{d}{n}e_{d,n}(z))(1+\frac{d}{n}e_{d,n}^\circ(z))}\tr\left\{D_{d,n}^{-1}(z)R_{d,n}\left(D_{d,n}^{\circ}(z)\right)^{-1}\right\}\\
&\hspace{6cm}-\frac{1}{n}\sum_{k=1}^nf_{k,m}.
\end{align*}
So, almost surely by \eqref{eq: spectralnormD}, \eqref{eq: bound3} and \eqref{eq: conv e},
\begin{align*}
&\limsup_{d\to\infty}\arrowvert m_{d,n}-m_{d,n}^\circ\arrowvert\\
&\hspace{1cm}\le \limsup_{d\to\infty} \arrowvert e_{d,n}-e_{d,n}^\circ \arrowvert \limsup_{d\to\infty}\frac{d}{n}\Arrowvert R_{d,n}\Arrowvert_{S_\infty}\frac{\big(2\Arrowvert\hat \Xi_{d,n}\Arrowvert_{S_\infty}+4\arrowvert z\arrowvert^4\big)^2}{\Im(z)^6}\\
&\hspace{1cm}=0.
\end{align*}
Now, consider the case
$$\frac{d^{\frac{3}{2}}}{n}\le 1.$$
Due to
$$\frac{d}{n}|e_{d,n}^\circ|\le \frac{d}{n}\frac{\sup_d\Arrowvert R_{d,n}\Arrowvert_{S_\infty}}{\Im z}\longrightarrow 0$$
for any $z\in C^+$ and by reasons of continuity, we conclude
$$\left|m_{d,n}^\circ(z)-m_{\mu^{T_{d,n}}}(z)\right|\rightarrow 0$$
for any $z\in \C^+$, where $\mu^{T_{d,n}}$ is the spectral measure of the matrix $T_{d,n}$. Therefore, it remains to show that
$$\left|m_{d,n}(z)-m_{\mu^{T_{d,n}}}(z)\right|\rightarrow 0.$$
By Lemma \ref{7.7} and Lemma \ref{7.8}  this convergence holds true if $d_L(\mu_{d,n},\mu^{T_{d,n}})\rightarrow 0$.
Theorem \ref{theorem: A38} for $\alpha=1$ and inequality (1.2) of \cite{Li1999} yield
$$d_L^2\left(\mu_{d,n},\mu^{T_{d,n}}\right)\le \frac{1}{d}\sum_{i=1}^d\left|\lambda_{i}(\Xi_{d,n})-\lambda_{i}(T_{d,n})\right|\le \left\Arrowvert \frac{1}{n}R_{d,n}^{1/2}X_{d,n}X_{d,n}^\ast R_{d,n}^{1/2}-R_{d,n} \right\Arrowvert_{S_\infty}.$$
Finally, for arbitrary $\varepsilon>0$ and $d$ sufficiently large we apply Corollary 5.50 of \cite{Vershynin2011} with $t=1$ so that
$$
\left\Arrowvert \frac{1}{n}R_{d,n}^{1/2}X_{d,n}X_{d,n}^\ast R_{d,n}^{1/2}-R_{d,n} \right\Arrowvert_{S_\infty}\le \varepsilon
$$
with probability at least $1-2\exp(-d)$. Again, by the Borel-Cantelli lemma,
$$
d_L\left(\mu_{d,n},\mu^{T_{d,n}}\right)\le \left\Arrowvert \frac{1}{n}R_{d,n}^{1/2}X_{d,n}X_{d,n}^\ast R_{d,n}^{1/2}-R_{d,n} \right\Arrowvert_{S_\infty}^{1/2}\rightarrow 0
$$
almost surely as $d\to\infty$.
\subsection{Step VII: Weak approximation of the spectral measures} First we show that the measure $\mu_{d,n}^\circ$ has compact support. Thereto, define similarly to the definition of $e_{(d,n)_l},~l\in\N$ in Step III,
$$m_{(d,n)_l}(z)=\frac{1}{d_l}\tr\left\{\left(\hat\Xi_{(d,n)_l}-zI_{d_l\times d_l}\right)^{-1}\right\}.$$
By \eqref{eq: intermediate1},
$$m_{(d,n)_l}-\frac{1}{d_l}\tr\left\{\left(\frac{1}{1+\frac{d}{n}e_{(d,n)_l}(z)}R_{(d,n)_l}-S_{(d,n)_l}-zI_{d_l\times d_l}\right)^{-1}\right\}\rightarrow 0 \ \ \text{as }l\to\infty$$
almost surely. Note that
\begin{align*}
&\frac{1}{d_l}\tr\left\{\left(\frac{1}{1+\frac{d}{n}e_{(d,n)_l}(z)}R_{(d,n)_l}-S_{(d,n)_l}-zI_{d_l\times d_l}\right)^{-1}\right\}\\
&\hspace{4cm}=\frac{1}{d}\tr\left\{\left(\frac{1}{1+\frac{d}{n}e_{(d,n)_l}(z)}R_{d,n}-S_{d,n}-zI_{d\times d}\right)^{-1}\right\},
\end{align*}
and therefore by reasons of continuity
$$m_{(d,n)_l}-\frac{1}{d}\tr\left\{\left(\frac{1}{1+\frac{d}{n}e_{d,n}^\circ(z)}R_{d,n}-S_{d,n}-zI_{d\times d}\right)^{-1}\right\} \rightarrow 0 \ \ \text{as }l\to\infty$$
almost surely because of \eqref{eq: konvergenz 3}. This implies that $\mu_{d,n}^\circ$ is the weak limit of $\mu_{(d,n)_l}$, and in particular the support of $\mu_{d,n}^\circ$ is bounded since
\begin{align*}
\left\arrowvert\inf\left\{x:\mu_{d,n}^\circ((-\infty,x])>0\right\}\right\arrowvert\ge \liminf_{l\to\infty} \lambda_{dl}\left(\hat\Xi_{{(d,n)}_l}\right)\ge -\Arrowvert S_{d,n} \Arrowvert_{S_\infty}
\end{align*}
and
\begin{align*}
\left\arrowvert\sup\left\{x:\mu_{d,n}^\circ((-\infty,x])<1\right\}\right\arrowvert\le \limsup_{l\to\infty} \big\Arrowvert \hat\Xi_{(d,n)_l}\big\Arrowvert_{S_\infty}\le \Arrowvert S_{d,n} \Arrowvert_{S_\infty}+c',
\end{align*}
where $c'>0$ is a constant satisfying inequality \eqref{c prime} of Lemma \ref{lemma: log d} applied to 
$$\frac{1}{nl}\sum_{k=1}^{nl}R_{(d,n)_l}^{1/2}\check Z_{k,d_l,n_l}\check Z_{k,d_l,n_l}^\ast R_{(d,n)_l}^{1/2},$$ 
and is chosen uniformly over $d\in\N$. Subsequently, we assume that $d$ (in dependence on the specific realization) is sufficiently large such that 
 $$\left\arrowvert\inf\left\{x:\mu_{d,n}((-\infty,x])>0\right\}\right\arrowvert\ge -\Arrowvert S_{d,n} \Arrowvert_{S_\infty}-c''$$
 and
 $$
 \left\arrowvert\sup\left\{x:\mu_{d,n}((-\infty,x])<1\right\}\right\arrowvert\le \Arrowvert S_{d,n} \Arrowvert_{S_\infty}+c''
 $$
 with an appropriate contant $c''>0$ from \eqref{c prime}. Now, define $c=c'\vee c''$.
For fixed $0<v< 1$, define the closed interval $K=\big[u_0,u_{\lfloor v^{-3} \rfloor+1}\big]$ 
with 
$$u_l=-v^{-1/4}\left(\left\|S_{d,n}\right\|_{S_{\infty}}+c\right)+\frac{2l}{\lfloor v^{-3} \rfloor+1}v^{-1/4}\left(\left\|S_{d,n}\right\|_{S_{\infty}}+c\right)$$ 
for $l=1,\dots,\lfloor v^{-3} \rfloor+1.$ By Step VI, we have $$|m_{d,n}(u_l+iv)-m_{d,n}^\circ(u_l+iv)|<v$$   simultaneously at all points $u_l,~l=0,\dots,\lfloor v^{-3} \rfloor+1,$ almost surely for all $d$ sufficiently large. Furthermore, for any inner point $u$ of $K$, pick $l$ such that $u\in [u_{l},u_{l+1})$. Then,
\begin{align*}
&\left\arrowvert m_{d,n}(u+iv)-m_{d,n}^\circ(u+iv) \right\arrowvert\\
&\hspace{0.5cm}\le \left\arrowvert m_{d,n}(u+iv)-m_{d,n}(u_l+iv) \right\arrowvert+\left\arrowvert m_{d,n}^\circ(u+iv)-m_{d,n}^\circ(u_l+iv) \right\arrowvert\\
&\hspace{1cm}+\left\arrowvert m_{d,n}(u_l+iv)-m_{d,n}^\circ(u_l+iv) \right\arrowvert \\
&\hspace{0.5cm}\le\int\left| \frac{1}{x-u-iv}-\frac{1}{x-u_l-iv}\right|\d\mu_{d,n}(x)\\
&\hspace{1cm}+\int\left| \frac{1}{x-u-iv}-\frac{1}{x-u_l-iv}\right|\d\mu_{d,n}^\circ(x)+v\\
&\hspace{0.5cm}\le \int\frac{u-u_l}{v^2}\d(\mu_{d,n}+\mu_{d,n}^\circ)(x)+v\\
&\hspace{0.5cm}\le v(4|u_0|+1).
\end{align*}
Next, we derive an upper bound on the integral
$$\int_{K^c}\left\arrowvert m_{d,n}(u+iv)-m_{d,n}^\circ(u+iv) \right\arrowvert\d u$$
which tends to zero for $v\to 0$. For this aim, we decompose the integral into
\begin{align*}
&\int_{K^c}\left\arrowvert m_{d,n}(u+iv)-m_{d,n}^\circ(u+iv) \right\arrowvert\d u\\ 
&\hspace{0.5cm}= \int_{(-\infty,u_0)}\left\arrowvert m_{d,n}(u+iv)-m_{d,n}^\circ(u+iv) \right\arrowvert\d u\\
&\hspace{3cm}+ \int_{(u_{\lfloor v^{-3} \rfloor+1},\infty)}\left\arrowvert m_{d,n}(u+iv)-m_{d,n}^\circ(u+iv) \right\arrowvert\d u
\end{align*}
We can use the same arguments for both integrals and therefore only consider the first one. By Fubini's theorem and the bounds on the support of $\mu_{d,n}$ and $\mu_{d,n}^\circ$,
\begin{align*}
&\int_{(-\infty,u_0)}\left\arrowvert m_{d,n}(u+iv)-m_{d,n}^\circ(u+iv) \right\arrowvert\d u\\
&\hspace{0.5cm}\le
\int\int\int_{(-\infty,u_0)}\left|\frac{1}{x-u-iv}-\frac{1}{y-u-iv}\right| \d u~\d \mu_{d,n}(x)~\d\mu_{d,n}^\circ(y) \\
&\hspace{0.5cm}\le
\int\int\int_{(-\infty,u_0)}\frac{|x-y|}{(u-v^{1/4}u_0)^2} \d u~\d \mu_{d,n}(x)~\d\mu_{d,n}^\circ(y) \\
&\hspace{0.5cm}\le
\frac{1}{(1-v^{1/4})|u_0|}\int\int|x-y| ~\d\mu_{d,n}(x)~\d\mu_{d,n}^\circ(y) \\
&\hspace{0.5cm}\le \frac{1}{(1-v^{1/4})|u_0|}\left(\int |x|  \d\mu_{d,n}(x)+\int |y|  \d\mu_{d,n}^\circ(y)\right)\\
&\hspace{0.5cm}\le 2\frac{v^{1/4}}{1-v^{1/4}}.
\end{align*}  
Now, by Lemma \ref{lemma: Levy distance} we conclude
\begin{align*}
d_L(\mu_{d,n},\mu^\circ_{d,n})&\le 2\sqrt{\frac{v}{\pi}}+ \frac{1}{2\pi}\int\left\arrowvert m_{d,n}(u+iv)-m_{d,n}^\circ(u+iv) \right\arrowvert\d u\\
&\le 2\sqrt{\frac{v}{\pi}}+\frac{1}{\pi}|u_0|(4|u_0|+1)v+\frac{2}{\pi}\frac{v^{1/4}}{1-v^{1/4}},
\end{align*}
where the inequalities hold almost surely for all $d$ sufficiently large. Hence, 
$$d_L(\mu_{d,n},\mu^\circ_{d,n})\longrightarrow 0$$ 
almost surely as $d\to \infty$. Lemma  \ref{7.8} yields finally
$\mu_{d,n}-\mu_{d,n}^{\circ}\Longrightarrow 0$ a.s.\hfill$\square$

\subsection{Proof of Corollary \ref{cor: 1}}
As afore-mentioned to the corollary, 
$$\mu_{d,n}^\circ=\mu^{MP}_{\frac{d}{n},\frac{p_0}{\sigma^2}}\star \delta_{-\sigma^2\frac{1-p_0}{p_0}}.$$
Therefore, by the representation \eqref{eq: mp dichte} of the Mar\v{c}enko-Pastur distribution we deduce $$\mu_{d,n}^\circ\Longrightarrow \mu^{\text{MP}}_{y,\frac{\sigma^2}{p_0}}\star \delta_{-\frac{1-p_0}{p_0}\sigma^2},$$
such that
$$ \mu_{d,n}\Longrightarrow \mu^{\text{MP}}_{y,\frac{\sigma^2}{p_0}}\star \delta_{-\frac{1-p_0}{p_0}\sigma^2}. $$
Furthermore, if the left edge of the limiting distribution  $$\mu^{\text{MP}}_{y,\frac{\sigma^2}{p_0}}\star \delta_{-\frac{1-p_0}{p_0}\sigma^2}$$ is smaller than zero, then almost surely
$$\limsup_{d\to\infty}\lambda_{\min}(\hat \Xi_{d,n})<0.$$
For $y<1$ the left edge of the limiting distribution is smaller than zero if and only if $p_0<1-(1-\sqrt{y})^2$. \hfill$\square$

\section{Proof of Theorem \ref{theorem: extreme}}\label{section: thm 2}
We will show Theorem \ref{theorem: extreme} by means of the next proposition. The proof of the proposition is postponed to Appendix \ref{B}.
\begin{proposition}\label{proposition: sigma}
Let $(X(i,k))_{i,k\in\N}$ be a double array of iid centered random variables with unit variance and finite fourth moment, and denote by $X_{d,n}\in \R^{d\times n}$ its $d\times n$ submatrix in the upper left corner. Moreover, let $(A_{d,n})_{d,n}$, $A_{d,n}\in \R^{d\times d}$, be a sequence of symmetric random matrices and  $(B_{d,n})_{d,n}$, $B_{d,n}\in \R^{d\times n}$ be another sequence of random matrices such that $(A_{d,n},B_{d,n})$ and $X_{d,n}$ are independent. Let $d,n\to \infty$ and $d/n\to y>0$. If 
\begin{align}\label{eq: bound A B}
\limsup_{d\to\infty}\max_{i,j} |A_{ij,d,n}|\max_{i,k}B_{ik,d,n}^2< \alpha \ \ \text{a.s.}
\end{align}
for some absolute constant $\alpha>0$, then
\begin{align}
\limsup_{d\to \infty}\left\Arrowvert \frac{1}{n}A_{d,n}\circ\left(\left(X_{d,n}\circ B_{d,n}\right)\left(X_{d,n}\circ B_{d,n}\right)^\ast\right)\right\Arrowvert_{S_\infty}\le \alpha\left(1+\sqrt{y}\right)^2   \ \ \text{a.s.}
\end{align}
\end{proposition}

\begin{proof}[Proof of Theorem \ref{theorem: extreme}]
By Weyl's inequality, we obtain
\begin{align*}
\lambda_{\max}&\left(\frac{1}{n}W_{d,n}\circ\left(\left(X_{d,n}\circ\varepsilon_{d,n}\right)\left(X_{d,n}\circ\varepsilon_{d,n}\right)^\ast\right)\right)\\
&\hspace{2cm}+\lambda_{\min}\left(\frac{1}{n}\left(\hat W_{d,n}-W_{d,n}\right)\circ\left(\left(X_{d,n}\circ\varepsilon_{d,n}\right)\left(X_{d,n}\circ\varepsilon_{d,n}\right)^\ast\right)\right)\\
&\le \lambda_{\max}\left(\frac{1}{n}\hat W_{d,n}\circ\left(\left(X_{d,n}\circ\varepsilon_{d,n}\right)\left(X_{d,n}\circ\varepsilon_{d,n}\right)^\ast\right)\right)\\
&\le \lambda_{\max}\left(\frac{1}{n}W_{d,n}\circ\left(\left(X_{d,n}\circ\varepsilon_{d,n}\right)\left(X_{d,n}\circ\varepsilon_{d,n}\right)^\ast\right)\right)\\
&\hspace{2cm}+\lambda_{\max}\left(\frac{1}{n}\left(\hat W_{d,n}-W_{d,n}\right)\circ\left(\left(X_{d,n}\circ\varepsilon_{d,n}\right)\left(X_{d,n}\circ\varepsilon_{d,n}\right)^\ast\right)\right).
\end{align*}
and
\begin{align*}
\lambda_{\min}&\left(\frac{1}{n}W_{d,n}\circ\left(\left(X_{d,n}\circ\varepsilon_{d,n}\right)\left(X_{d,n}\circ\varepsilon_{d,n}\right)^\ast\right)\right)\\
&\hspace{2cm}+\lambda_{\min}\left(\frac{1}{n}\left(\hat W_{d,n}-W_{d,n}\right)\circ\left(\left(X_{d,n}\circ\varepsilon_{d,n}\right)\left(X_{d,n}\circ\varepsilon_{d,n}\right)^\ast\right)\right)\\
&\le \lambda_{\min}\left(\frac{1}{n}\hat W_{d,n}\circ\left(\left(X_{d,n}\circ\varepsilon_{d,n}\right)\left(X_{d,n}\circ\varepsilon_{d,n}\right)^\ast\right)\right)\\
&\le \lambda_{\min}\left(\frac{1}{n}W_{d,n}\circ\left(\left(X_{d,n}\circ\varepsilon_{d,n}\right)\left(X_{d,n}\circ\varepsilon_{d,n}\right)^\ast\right)\right)\\
&\hspace{2cm}+\lambda_{\max}\left(\frac{1}{n}\left(\hat W_{d,n}-W_{d,n}\right)\circ\left(\left(X_{d,n}\circ\varepsilon_{d,n}\right)\left(X_{d,n}\circ\varepsilon_{d,n}\right)^\ast\right)\right).
\end{align*}
Because of
\begin{align*}
&\lambda_{\max}\left(\frac{1}{n}W_{d,n}\circ\left(\left(X_{d,n}\circ\varepsilon_{d,n}\right)\left(X_{d,n}\circ\varepsilon_{d,n}\right)^\ast\right)\right)\\
&\hspace{0.25cm}=\lambda_{\max}\bigg(\frac{1}{n}\left(w_{d,n}w_{d,n}^\ast\right)\circ\left(\left(X_{d,n}\circ\varepsilon_{d,n}\right)\left(X_{d,n}\circ\varepsilon_{d,n}\right)\right)-\frac{1-p_0}{p_0}\sigma^2I_{d\times d}\\
&\hspace{0.5cm}+\diag\left[\frac{1}{n}\left(W_{d,n}-w_{d,n}w_{d,n}^\ast\right)\circ\left(\left(X_{d,n}\circ\varepsilon_{d,n}\right)\left(X_{d,n}\circ\varepsilon_{d,n}\right)^\ast\right)\right]+\frac{1-p_0}{p_0}\sigma^2I_{d\times d}\bigg),
\end{align*}
and 
\begin{align*}
&\left\Arrowvert \diag\left[\frac{1}{n}\left(W_{d,n}-w_{d,n}w_{d,n}^\ast\right)\circ\left(\left(X_{d,n}\circ\varepsilon_{d,n}\right)\left(X_{d,n}\circ\varepsilon_{d,n}\right)^\ast\right)\right]+\frac{1-p_0}{p_0}\sigma^2I_{d\times d} \right\Arrowvert_{S_\infty}\\
&\hspace{2cm}\longrightarrow 0 \ \ \text{a.s. as $d\to\infty$}
\end{align*}
by the Marcinkiewicz-Zygmund strong law of large numbers (cf. Lemma B.25 in \cite{Bai2010}),
we obtain again by Weyl's inequality and Theorem 1 of \cite{Bai1993}
$$
\lambda_{\max}\left(\frac{1}{n}W_{d,n}\circ\left(\left(X_{d,n}\circ\varepsilon_{d,n}\right)\left(X_{d,n}\circ\varepsilon_{d,n}\right)^\ast\right)\right)\overset{\text{a.s.}}{\longrightarrow} \frac{\sigma^2}{p_0}\left(1+\sqrt{y}\right)^2-\frac{1-p_0}{p_0}\sigma^2.
$$
With same argument,
\begin{align*}
\lambda_{\min}\left(\frac{1}{n}W_{d,n}\circ\left(\left(X_{d,n}\circ\varepsilon_{d,n}\right)\left(X_{d,n}\circ\varepsilon_{d,n}\right)^\ast\right)\right)\overset{\text{a.s.}}{\longrightarrow}  \frac{\sigma^2}{p_0}\left(1-\sqrt{y}\right)^2-\frac{1-p_0}{p_0}\sigma^2.
\end{align*}
In order to finish the proof, it suffices to show that
\begin{align*}
\left\Arrowvert\frac{1}{n}\left(\hat W_{d,n}-W_{d,n}\right)\circ\left(\left(X_{d,n}\circ\varepsilon_{d,n}\right)\left(X_{d,n}\circ\varepsilon_{d,n}\right)^\ast\right)\right\Arrowvert_{S_{\infty}}\overset{\text{a.s.}}{\longrightarrow} 0.
\end{align*}
But this is an easy consequence of Proposition \ref{proposition: sigma} since by \eqref{eq: p},
\begin{align*}
\limsup_{n\to\infty} \max_{i,j} \left|\hat W_{ij,d,n}-W_{ij,d,n}\right| \overset{\text{a.s.}}{\longrightarrow} 0.
\end{align*}
\end{proof}

\appendix
\smallskip
\section{Proof of Proposition 4.1}\label{A}
\smallskip
\subsection{Step I: Modifying $\varepsilon_{d,n}$}\label{subsec: modify}
By tightness of $(\mu^{w_{d,n}})$ we have for any $\delta>0$ a constant $p_0>0$ such that for sufficiently large $d\in\N$ 
\[
\#\{p_{i,d,n}<p_0\}\le d\delta.
\]
We replace the matrix $\varepsilon_{d,n}$ by $\tilde \varepsilon_{d,n}$, where $\tilde\varepsilon_{ik,d,n}=\varepsilon_{ik,d,n}$ if $p_i\ge p_0$ and otherwise $\tilde\varepsilon_{ik,d,n}$ is a Bernoulli random variable with $\P(\tilde\varepsilon_{ik}=1)=p_0$ such that the entries of $\tilde\varepsilon_{d,n}$ are independent and jointly independent of $Y_{1,d,n},...,Y_{n,d,n}$. $\tilde T_{d,n}$ be the matrix as $\hat T_{d,n}$ but relying on the missingness matrix $\tilde \varepsilon_{d,n}$ in place of $\varepsilon_{d,n}$. Since by Theorem \ref{theorem: A43}
\[
d_K\left(\mu^{\tilde T_{d,n}},\mu^{\hat T_{d,n}}\right)\le \frac{1}{d}\rank(\tilde T_{d,n}-\hat T_{d,n})\le \delta,
\]
we may assume subsequently $p_{i,d,n}\ge p_0$.
\subsection{Step II: Removing $\frac{1}{n}\hat W_{d,n}\circ \left((\hat M_{d,n}\circ\varepsilon_{d,n})(\hat M_{d,n}\circ\varepsilon_{d,n})^\ast\right)$}\label{subsec: m} Let 
$$\tilde T_{d,n}=\hat T_{d,n}-\frac{1}{n}\hat W_{d,n}\circ \left((\hat M_{d,n}\circ\varepsilon_{d,n})(\hat M_{d,n}\circ\varepsilon_{d,n})^\ast\right).$$ 
First note that
$$\P\left(\min_{i,j}\#\NN_{ij}=0\right)\le d^2 \max_{i,j} P\left(\#\NN_{ij}=0\right)\le d^2 (1-p_0^2)^n.$$
Hence, by the Borel-Cantelli lemma we have almost surely for all but finitely many indices $d$
$$\frac{1}{n}\hat W_{d,n}\circ \left((\hat M_{d,n}\circ\varepsilon_{d,n})(\hat M_{d,n}\circ\varepsilon_{d,n})^\ast\right)=\hat m_{d,n}\hat m_{d,n}^\ast.$$
Now, by Theorem \ref{theorem: A43} we have
\[
\limsup_{d\to\infty}d_K\left( \mu^{\hat T_{d,n}},\mu^{\tilde T_{d,n}}\right)=0 \ \ \text{a.s.}
\]
Therefore it is sufficient to prove $d_L(\mu^{\tilde T_{d,n}},\mu^{\bar T_{d,n}})\to 0$. In the next subsection, we refer to $\tilde T_{d,n}$ as $\hat T_{d,n}$.

\subsection{Step III: Truncation of $T_{d,n}$}\label{subsec: trunc T} By the tightness of the sequence $(\mu^{T_{d,n}})$ we have for any $\delta>0$ a constant $\tau_0>0$ such that for sufficiently large $d\in\N$
\[
\#\{T_{kk,d,n}>\tau_0\}\le d\delta. 
\]
Therefore, let $\check T_{d,n}=\diag(\ind\{T_{11,d,n}\le \tau_0\}T_{11,d,n},...,\ind\{T_{dd,d,n}\le \tau_0\}T_{dd,d,n})$ and $\tilde T_{d,n}$ be the sample covariance matrix with missing observations built from the random variables 
$$\tilde Y_{i,d,n}=\check T_{d,n}X_{i,d,n}, \ \ i=1,\dots,n,$$ while $\varepsilon_{d,n}$ remains the same. Since again by Theorem \ref{theorem: A43}  $$d_K\left( \mu^{\check T_{d,n}},\mu^{\hat T_{d,n}}\right)\le \frac{1}{d}\rank\left(\check T_{d,n}-\hat T_{d,n}\right)\le \delta,$$
it is sufficient to assume subsequently that the spectral measures of the sequence $(T_{d,n})$ have uniformly bounded support.

\subsection{Step IV: Truncation of $X_{d,n}$} \label{subsec: trunc X 1} For $0<\delta< \frac{1}{2}$ we truncate the variables $X_{ik,d,n}$ at the threshold level $n^{1/2}d^{\alpha-1/2}$, $\alpha>\frac{1+\delta}{2}$. Hence, let 
\begin{equation}
\tilde X_{ik,d,n}=X_{ik,d,n}\ind(|X_{ik,d,n}|\le n^{1/2}d^{\alpha-1/2})\label{eq: tilde1}
\end{equation}
 and  $\tilde T_{d,n},~\tilde Y_{d,n},\text{ and }\tilde M_{d,n}$ be the matrices constructed by replacing $X_{d,n}$ with $\tilde X_{d,n}=(\tilde X_{ik,d,n})$ in $\hat T_{d,n},~Y_{d,n},\text{ and }\hat M_{d,n}$. We have
\begin{align}
&\hspace{-1cm}d_K\left( \mu^{\tilde T_{d,n}},\mu^{\hat T_{d,n}}\right) \notag\\
&\le \frac{1}{d}\rank(\tilde T_{d,n}-\hat T_{d,n}) \notag \\
&=\frac{1}{d}\rank\bigg[\frac{1}{n}\hat W_{d,n}\circ\Big((Y_{d,n}\circ \varepsilon_{d,n})(Y_{d,n}\circ \varepsilon_{d,n})^\ast-(\tilde Y_{d,n}\circ \varepsilon_{d,n})(\tilde Y_{d,n}\circ \varepsilon_{d,n})^\ast\notag \\
&\hspace{2cm}-(\hat M_{d,n}\circ \varepsilon_{d,n})( Y_{d,n}\circ \varepsilon_{d,n})^\ast+(\tilde M_{d,n}\circ \varepsilon_{d,n})(\tilde Y_{d,n}\circ \varepsilon_{d,n})^\ast\notag \\
&\hspace{2cm}-(Y_{d,n}\circ \varepsilon_{d,n})(\hat M_{d,n}\circ \varepsilon_{d,n})^\ast+(\tilde Y_{d,n}\circ \varepsilon_{d,n})(\tilde M_{d,n}\circ \varepsilon_{d,n})^\ast\Big)\bigg] \notag\\
&=\frac{1}{d}\rank\bigg[\frac{1}{n}\hat W_{d,n}\circ\Big(((Y_{d,n}-\tilde Y_{d,n})\circ \varepsilon_{d,n})(Y_{d,n}\circ \varepsilon_{d,n})^\ast\notag \\
&\hspace{2cm}+(\tilde Y_{d,n}\circ \varepsilon_{d,n})((Y_{d,n}-\tilde Y_{d,n})\circ \varepsilon_{d,n})^\ast\notag\\
&\hspace{2cm}-(\hat M_{d,n}\circ \varepsilon_{d,n})( (Y_{d,n}-\tilde Y_{d,n})\circ \varepsilon_{d,n})^\ast \notag\\
&\hspace{2cm}-((\hat M_{d,n}-\tilde M_{d,n})\circ \varepsilon_{d,n})(\tilde Y_{d,n}\circ \varepsilon_{d,n})^\ast\notag\\
&\hspace{2cm}-( (Y_{d,n}-\tilde Y_{d,n})\circ \varepsilon_{d,n})(\hat M_{d,n}\circ \varepsilon_{d,n})^\ast\notag\\
&\hspace{2cm}-(\tilde Y_{d,n}\circ \varepsilon_{d,n})((\hat M_{d,n}-\tilde M_{d,n})\circ \varepsilon_{d,n})^\ast\Big)\bigg] \notag\\
&\le \frac{1}{d}\rank\bigg[\frac{1}{n}\hat W_{d,n}\circ\big(((Y_{d,n}-\tilde Y_{d,n})\circ \varepsilon_{d,n})((Y_{d,n}-\hat M_{d,n})\circ \varepsilon_{d,n})^\ast\notag\\
&\hspace{2cm}-((\hat M_{d,n}-\tilde M_{d,n})\circ \varepsilon_{d,n})(\tilde Y_{d,n}\circ \varepsilon_{d,n})^\ast\big)\bigg]\notag\\
&\hspace{0.5cm}+\frac{1}{d}\rank\bigg[\frac{1}{n}\hat W_{d,n}\circ\big(((\tilde Y_{d,n}-\hat M_{d,n})\circ \varepsilon_{d,n})((Y_{d,n}-\tilde Y_{d,n})\circ \varepsilon_{d,n})^\ast\notag\\
&\hspace{2.5cm}-(\tilde Y_{d,n}\circ \varepsilon_{d,n})((\hat M_{d,n}-\tilde M_{d,n})\circ \varepsilon_{d,n})^\ast\big)\bigg]\notag\\
&\le \frac{2}{d} \#\left\{i\in\{1,\dots,d\}:\sum_{k=1}^n\ind(|X_{ik,d,n}|>n^{1/2}d^{\alpha-1/2})>0\right\} \label{inequality: null vector}\\
&\le \frac{2}{d}\sum_{i,k}\ind(|X_{ik,d,n}|>n^{1/2}d^{\alpha-1/2}) \label{inequality: bound number null vectors},
\end{align}
where inequality (\ref{inequality: null vector}) follows by the simple observation that the $i$-th row respectively the $i$-th column of the matrices 
$$\left((Y_{d,n}-\tilde Y_{d,n})\circ \varepsilon_{d,n}\right)\left((Y_{d,n}-\hat M_{d,n})\circ \varepsilon_{d,n}\right)^\ast$$ 
and
$$\left((\hat M_{d,n}-\tilde M_{d,n})\circ \varepsilon_{d,n}\right)(\tilde Y_{d,n}\circ \varepsilon_{d,n})^\ast$$
respectively
$$\left((\tilde Y_{d,n}-\hat M_{d,n})\circ \varepsilon_{d,n}\right)\left((Y_{d,n}-\tilde Y_{d,n})\circ \varepsilon_{d,n}\right)^\ast$$
and
$$(\tilde Y_{d,n}\circ \varepsilon_{d,n})\left((\hat M_{d,n}-\tilde M_{d,n})\circ \varepsilon_{d,n}\right)^\ast$$ 
is the null vector if 
$$\sum_{k=1}^n\ind(|X_{ik,d,n}|>n^{1/2}d^{\alpha-1/2})=0.$$
Next we prove that 
$$\frac{2}{d}\sum_{i,k}\ind(|X_{ik,d,n}|>n^{1/2}d^{\alpha-1/2})\overset{\text{a.s.}}{\longrightarrow}0$$
as $d\rightarrow \infty$. Note first that by Markov's inequality
\begin{align}
\Var\left(\ind\{|X_{11,d,n}|>n^{1/2}d^{\alpha-1/2}\}\right)&\le \E\ind(|X_{ik,d,n}|>n^{1/2}d^{\alpha-1/2}\}\notag\\
&\le n^{-1}d^{1-2\alpha}.\label{eq: markov}
\end{align}
 Using \eqref{eq: markov} in \eqref{eq: 3.4}, and  \eqref{eq: markov} in Bernstein's inequality in \eqref{eq: 3.7}, we conclude for sufficiently large $d$ and some constant $\beta>0$
\begin{align}\notag
\P&\left(\sum_{k,i}\ind\{|X_{ik,d,n}|>n^{1/2}d^{\alpha-1/2}\}\ge d^{1-\delta}\right)\notag \\
&=\P\Bigg(\sum_{k,i}\Big(\ind\{|X_{ik,d,n}|>n^{1/2}d^{\alpha-1/2}\}-\E\ind(|X_{ik,d,n}|>n^{1/2}d^{\alpha-1/2}\}\Big)\notag\\
&\hspace{4cm}\ge d^{1-\delta}-nd\E \ind\{|X_{11,d,n}|>n^{1/2}d^{\alpha-1/2}\}\Bigg)\notag\\
&\le\P\Bigg(\sum_{k,i}(\ind\{|X_{ik,d,n}|>n^{1/2}d^{\alpha-1/2}\}-\E\ind\{|X_{ik,d,n}|>n^{1/2}d^{\alpha-1/2}\})\notag\\
&\hspace{6cm}\ge d^{1-\delta}-d^{2(1-\alpha)}\Bigg)\label{eq: 3.4}\\
&\le\P\Bigg(\sum_{k,i}(\ind\{|X_{ik,d,n}|>n^{1/2}d^{\alpha-1/2}\}-\E\ind\{|X_{ik,d,n}|>n^{1/2}d^{\alpha-1/2}\})\notag\\
&\hspace{6cm}\ge \frac{1}{2} d^{1-\delta}\Bigg)\label{ine: fir}\\
&\le\exp\left(-\beta d^{1-\delta}\right) \label{eq: 3.7},
\end{align}
where inequality \eqref{ine: fir} holds since $\alpha>(1+\delta)/2$. So, by inequality (\ref{inequality: bound number null vectors}) follows
$$d_K\left(\mu^{\tilde T_{d,n}},\mu^{\hat T_{d,n}}\right)\overset{\text{a.s.}}{\longrightarrow} 0 \ \ \text{ for } \ \ d\to\infty.$$ Note that $\tilde X_{d,n}$ is not centered and standardized, but by Cauchy-Schwarz inequality and Markov inequality,
\begin{align}
|\E\tilde X_{ik,d,n}|&=|\E X_{ik,d,n}-\E \tilde X_{ik,d,n}|\notag\\
&=|\E X_{ik,d,n}\ind(|X_{ik,d,n}|> n^{1/2}d^{\alpha-1/2})|\notag\\
&\le \sqrt{\P(|X_{ik,d,n}|>n^{1/2}d^{\alpha-1/2})}\notag\\
&\le n^{-1/2}d^{1/2-\alpha}\label{eq: tilde}
\end{align}
and moreover, $\Var(\tilde X_{ik,d,n})\uparrow 1$ as $d\to \infty$. In the subsequent section we redefine the matrix $X_{d,n}$ by $\tilde X_{d,n}$ and keep the initial notations.
\subsection{Step V: Replacing the normalizing matrix $n^{-1}\hat W_{d,n}$} Let 
\begin{align*}
\tilde T_{d,n}&=\frac{1}{n}W_{d,n}\circ \Big((Y_{d,n}\circ\varepsilon_{d,n})(Y_{d,n}\circ\varepsilon_{d,n})^\ast\Big)\\
&\hspace{2cm}-\frac{1}{n}W_{d,n}\circ\left((\hat M_{d,n}\circ\varepsilon_{d,n})(Y_{d,n}\circ\varepsilon_{d,n})^\ast\right)\\
&\hspace{2cm}-\frac{1}{n}W_{d,n}\circ \left((Y_{d,n}\circ\varepsilon_{d,n})(\hat M_{d,n}\circ\varepsilon_{d,n})^\ast\right).
\end{align*}
By Theorem \ref{theorem: inequality Levy distance trace}, the elementary inequality $$\tr\left((C+D)^2\right)\le 2\tr(C^2+D^2)$$
 for symmetric $d\times d$ matrices $C$ and $D$, applied to
 \begin{align*}
 C&=\frac{1}{n}(\hat W_{d,n}-W_{d,n})\circ \Big(\big((Y_{d,n}\circ\varepsilon_{d,n})(Y_{d,n}\circ\varepsilon_{d,n})^\ast \big)\Big),\\
D&=-\frac{1}{n}(\hat W_{d,n}-W_{d,n})\circ \Big(\big((\hat M_{d,n}\circ\varepsilon_{d,n})(Y_{d,n}\circ\varepsilon_{d,n})^\ast\big)+\big((Y_{d,n}\circ\varepsilon_{d,n})(\hat M_{d,n}\circ\varepsilon_{d,n})^\ast\big)\Big),
 \end{align*}
as well as the inequality
$$\tr[\left(A+A^\ast\right)^2]\le 4\tr(AA^\ast)$$
 for any matrix $A$ with real entries, we deduce 
\begin{align}
\notag d_L^3\Big(&\mu^{\tilde T_{d,n}},\mu^{\hat T_{d,n}}\Big)\\
&\le \frac{1}{d}\tr \Bigg[\Bigg(\frac{1}{n}(\hat W_{d,n}-W_{d,n})\circ \Big(\big((Y_{d,n}\circ\varepsilon_{d,n})(Y_{d,n}\circ\varepsilon_{d,n})^\ast \big)\notag\\
&\hspace{1.1cm}-\big((\hat M_{d,n}\circ\varepsilon_{d,n})(Y_{d,n}\circ\varepsilon_{d,n})^\ast\big)-\big((Y_{d,n}\circ\varepsilon_{d,n})(\hat M_{d,n}\circ\varepsilon_{d,n})^\ast\big)\Big)\Bigg)^2\Bigg]\notag\\
&\le \frac{2}{d}\tr\Bigg[\left(\frac{1}{n}(\hat W_{d,n}-W_{d,n})\circ \Big(\big((Y_{d,n}\circ\varepsilon_{d,n})(Y_{d,n}\circ\varepsilon_{d,n})^\ast \big)\Big)\right)^2\Bigg]\notag\\
&\hspace{1.1cm}+\frac{8}{d}\tr\Bigg[\frac{1}{n^2}(\hat W_{d,n}-W_{d,n})^2\circ\Big(\big((\hat M_{d,n}\circ\varepsilon_{d,n})(Y_{d,n}\circ\varepsilon_{d,n})^\ast\big)\notag\\&\hspace{6cm}\times\big((Y_{d,n}\circ\varepsilon_{d,n})(\hat M_{d,n}\circ\varepsilon_{d,n})^\ast\big)\Big)\Bigg]\notag\\
\label{ine:detwei}&=:h_{d,n}.
\end{align}
We prove that $h_{d,n}\to 0$ a.s. as $d\to\infty$. 
Thereto, define for an arbitrary constant 
\begin{equation}
\gamma>\sqrt{4\alpha+7} \label{eq: gamma1}
\end{equation}
the event 
\begin{equation}
A_{d,n}=\left\{\forall  \ 1\le i,j\le d: \left|(\hat W_{ij,d,n})^{-1}-(W_{ij,d,n})^{-1}\right|\le \gamma\sqrt{\frac{\log n}{n}}\right\}. \label{eq: def A}
\end{equation}
Then, for sufficiently large $d$ the union bound and Hoeffding's inequality yield
\begin{align*}
\P (A_{d,n})&=1-\P(A_{d,n}^c)\\ 
&\ge 1- d^2 \max_{i, j}\P\left(\left|(\hat W_{ij,d,n})^{-1}-(W_{ij,d,n})^{-1}\right|>\gamma\sqrt{\frac{\log n}{n}}\right)\\
&\ge 1-2d^2\exp\left(-\frac{\gamma^2\log n}{2}\right)\\
&= 1-2d^{2}n^{-2\gamma^2}.
\end{align*}
By the Borel-Cantelli lemma all but finitely many events $A_{d,n}$ almost surely occur. Hence, if $\ind_{A_{d,n}}h_{d,n}\to 0$ a.s. for $d\to \infty$ then $h_{d,n}\to 0$ a.s.
Note furthermore that on the event $A_{d,n}$,
\begin{align}
\left|\hat W_{ij,d,n}-W_{ij,d,n}\right|&=\left|\frac{1}{(\hat W_{ij,d,n})^{-1}}-\frac{1}{(W_{ij,d,n})^{-1}}\right| \notag \\
&=\frac{\left|(\hat W_{ij,d,n})^{-1}-W_{ij,d,n}^{-1}\right|}{\left|(\hat W_{ij,d,n})^{-1}(W_{ij,d,n})^{-1}\right|}\notag \\
&\leq \frac{\gamma \sqrt{(\log n)/n}}{\min_i p_{i,d,n}^2\left((W_{ij,d,n})^{-1}-\left|(\hat W_{ij,d,n})^{-1}-(W_{ij,d,n})^{-1}\right|\right)}\notag \\
&\le \frac{\gamma\sqrt{(\log n)/n}}{\min_i p_{i,d,n}^2\left(\min_ip_{i,d,n}^2-\gamma \sqrt{(\log n)/n}\right)} \notag \\
&\le \frac{2\gamma}{\min_i p_{i,d,n}^4}\sqrt{\frac{\log n}{n}}\label{eq: p}
\end{align}
for $d$ sufficiently large. 
Now we prove that
$\E \ind_{A_{d,n}} h_{d,n}\to 0$. In order to save space the explicit dependence on  $d$ and $n$ is suppressed in the displays until the end of the section.
 By inequality \eqref{eq: p}, we have
\begin{align*}
\E h\ind_A&\le \frac{8\gamma^2\log n}{\min_ip_i^8dn^3}\sum_{i,j=1}^d\E\left(\left(\sum_{k=1}^n Y_{ik}Y_{jk}\varepsilon_{ik}\varepsilon_{jk}\right)^2+4\left(\sum_{k=1}^n\hat M_{ik}Y_{jk}\varepsilon_{ik}\varepsilon_{jk}\right)^2\right)\ind_A\\
&\le \frac{8\gamma^2\log n}{\min_ip_i^8dn^3}\Bigg(\sum_{i,j=1}^d\sum_{k,l=1}^n|\E Y_{ik}Y_{jk}Y_{il}Y_{jl}|\\
&\hspace{3cm}+4\sum_{k,l=1}^n\E\hat M_{ik}Y_{jk}\hat M_{il}Y_{jl}\varepsilon_{ik}\varepsilon_{jk}\varepsilon_{il}\varepsilon_{jl}\ind_A\Bigg)\\
&=I_1+I_2,
\end{align*}
where
\begin{align*}
I_1=\frac{8\gamma^2\log n}{\min_ip_i^8dn^3}\sum_{i,j=1}^d\sum_{k,l=1}^n|\E Y_{ik}Y_{jk}Y_{il}Y_{jl}|
\end{align*}
and
\begin{align*}
I_2=\frac{32\gamma^2\log n}{\min_ip_i^8dn^3}\sum_{i,j=1}^d\sum_{k,l=1}^n\E\hat M_{ik}Y_{jk}\hat M_{il}Y_{jl}\varepsilon_{ik}\varepsilon_{jk}\varepsilon_{il}\varepsilon_{jl}\ind_A.
\end{align*}
For the first term we obtain by \eqref{eq: tilde}, \eqref{eq: asymptotic}, uniform boundedness of the entries of $T_{d,n}$, and  \eqref{eq: tilde1}
\begin{align*}
I_1&=\frac{8\gamma^2\log n}{\min_ip_i^8dn^3}\Bigg(\sum_{\substack{i,j=1\\ i\neq j}}^d\sum_{\substack{k,l=1\\ k\neq l}}^n|\E Y_{ik}Y_{jk}Y_{il}Y_{jl}|+\sum_{i=1}^d\sum_{\substack{k,l=1\\ k\neq l}}^n|\E Y_{ik}^2Y_{il}^2|\\
&\hspace{2cm}+\sum_{\substack{i,j=1\\ i\neq j}}^d\sum_{k=1}^n|\E Y_{ik}^2Y_{jk}^2|+\sum_{i=1}^d\sum_{k=1}^n\E Y_{ik}^4\Bigg)\\
&\lesssim\frac{\log n}{nd^{4\alpha-1}}+\frac{\log n}{n}+\frac{d\log n}{n^2}+\frac{\log n}{n d^{1-2\alpha}}\\
&\lesssim \frac{\log n}{nd^{1-2\alpha}}\longrightarrow 0.
\end{align*}
Recall the definition \eqref{eq: M} of $\hat M_{d,n}$. Using again the bound 
\begin{equation}
|\hat W_{ii}|\le \frac{1}{(W_{ii})^{-1}-|(\hat W_{ii})^{-1}-(W_{ii})^{-1}|}\le \frac{2}{\min_i p_i^2}\ \ \ \text{on the event} \  A\label{eq: C on A}
\end{equation}
for $d$ sufficiently large, we get for the second term with the same type of arguments
\begin{align*}
I_2&=\frac{24\gamma^2\log n}{\min_ip_i^8dn^3}\sum_{i,j=1}^d\sum_{k_1,k_2,k_3,k_4=1}^n\hspace{-1mm}\E\frac{1}{n^2}\hat W_{ii}^2 Y_{jk_1}Y_{jk_2}Y_{ik_3} Y_{ik_4}\varepsilon_{ik_1}\varepsilon_{jk_1}\varepsilon_{ik_2}\varepsilon_{jk_2}\varepsilon_{ik_3}\varepsilon_{ik_4}\ind_A\\
&\le   \frac{96\gamma^2\log n}{\min_ip_i^{12}dn^5}\sum_{i,j=1}^d\sum_{k_1,k_2,k_3k_4=1}^n\left| \E Y_{jk_1} Y_{jk_2} Y_{ik_3}  Y_{ik_4} \right|\\
&\lesssim \frac{\log n}{dn^5}\Bigg[\sum_{i=1}^d\Bigg(\sum_{\substack{k_1,k_2,k_3,k_4=1\\ k_1\neq k_2\neq k_3\neq k_4}}^n+\sum_{\substack{k_1,k_2,k_3,k_4=1\\ \neg (k_1\neq k_2\neq k_3\neq k_4)}}^n\Bigg)\left| \E Y_{ik_1} Y_{ik_2} Y_{ik_3}Y_{ik_4} \right|\\
&\hspace{2cm}+ \sum_{\substack{i,j=1\\ i\neq j}}^d\Bigg(\sum_{\substack{k_1,k_2,k_3,k_4=1\\ k_1\neq k_2\neq k_3\neq k_4}}^n+\sum_{\substack{k_1,k_2,k_3,k_4=1\\ \neg(k_1\neq k_2\neq k_3\neq k_4)}}^n\Bigg)\left| \E Y_{jk_1} Y_{jk_2} Y_{ik_3}  Y_{ik_4} \right|\Bigg]\\
&\lesssim \frac{\log n}{dn^5}\left(d^{3-4\alpha}n^2+d^{2\alpha}n^{4}+d^{4-4\alpha}n^2+d^2n^{3}\right)\\
&\lesssim \frac{d^{2\alpha-1}\log n}{n}\longrightarrow 0.
\end{align*}
We need a sufficiently tight bound on the variance of $h_{d,n}\ind_{A_{d,n}}$ in order to conclude by the Borel-Cantelli lemma  that in addition $h_{d,n}\ind_{A_{d,n}}\to 0$ almost surely. Thereto, define 
 $$\hat G_{ij,d,n}=\frac{1}{n}\left(\hat W_{ij,d,n}-W_{ij,d,n}\right),\ \ i,j=1,...,d.$$ 
 Using \eqref{eq: p} in \eqref{equation: first term} and dropping those summands of \eqref{eq: drop} whose indices satisfy $\{i_1,j_1\}\cap\{i_2,j_2\}\neq\emptyset$, we get
\begin{align}
&\Var h\ind_A\notag \\
&=\frac{1}{d^2}\sum_{i_1,i_2,j_1,j_2=1}^d\E\Bigg\{\hat G_{i_1j_1}^2\Bigg(2\Bigg(\sum_{k\in\NN_{i_1j_1}}Y_{i_1k}Y_{j_1k}\Bigg)^2+8\Bigg(\sum_{k\in\NN_{i_1j_1}}\hat M_{i_1k}Y_{j_1k}\Bigg)^2\Bigg)\notag \\
&\hspace{3cm}\times \hat G_{i_2j_2}^2\Bigg(2\Bigg(\sum_{k\in\NN_{i_2j_2}} Y_{i_2k}Y_{j_2k}\Bigg)^2+8\Bigg(\sum_{k\in\NN_{i_2j_2}}\hat M_{i_2k}Y_{j_2k}\Bigg)^2\Bigg)\ind_A\Bigg\} \notag \\
&\hspace{0.5cm}-\frac{1}{d^2}\sum_{i_1,i_2,j_1,j_2=1}^d\E \Bigg\{\hat G_{i_1j_1}^2\Bigg(2\Bigg(\sum_{k\in\NN_{i_1j_1}} Y_{i_1k}Y_{j_1k}\Bigg)^2+8\Bigg(\sum_{k\in\NN_{i_1j_1}}\hat M_{i_1k}Y_{j_1k}\Bigg)^2\Bigg)\ind_A\Bigg\} \label{eq: drop}\\
&\hspace{3cm}\times\E \Bigg\{\hat G_{i_2j_2}^2\Bigg(2\Bigg(\sum_{k\in\NN_{i_2j_2}} Y_{i_2k}Y_{j_2k}\Bigg)^2+8\Bigg(\sum_{k\in\NN_{i_2j_2}}\hat M_{i_2k}Y_{j_2k}\Bigg)^2\Bigg)\ind_A\Bigg\} \notag \\
&\le \frac{2^{10}\gamma^4(\log n)^2}{\min p_i^{16} d^2n^6}\hspace{-2mm}\sum_{\substack{i_1,i_2,j_1,j_2=1\\ \{i_1,j_1\}\cap\{i_2,j_2\}\neq\emptyset}}^d\hspace{-2mm}\E\Bigg\{\Bigg(\Bigg(\sum_{k\in\NN_{i_1j_1}} Y_{i_1k}Y_{j_1k}\Bigg)^2+\Bigg(\sum_{k\in\NN_{i_1j_1}}\hat M_{i_1k}Y_{j_1k}\Bigg)^2\ind_A\Bigg) \notag \\
&\hspace{4.2cm}\times\Bigg(\Bigg(\sum_{k\in\NN_{i_2j_2}}^n Y_{i_2k}Y_{j_2k}\Bigg)^2+\Bigg(\sum_{k\in\NN_{i_2j_2}}\hat M_{i_2k}Y_{j_2k}\Bigg)^2\ind_A\Bigg)\Bigg\} \label{equation: first term}\\
&\hspace{0.5cm}+\frac{1}{d^2}\hspace{-2mm}\sum_{\substack{i_1,i_2,j_1,j_2=1\\ \{i_1,j_1\}\cap \{i_2,j_2\}=\emptyset}}^d\hspace{-2mm}\E\Bigg\{\hat G_{i_1j_1}^2\Bigg(\Bigg(\sum_{k\in\NN_{i_1j_1}} Y_{i_1k}Y_{j_1k}\Bigg)^2+\Bigg(\sum_{k\in\NN_{i_1j_1}}\hat M_{i_1k}Y_{j_1k}\Bigg)^2\Bigg)\notag \\
&\hspace{3.5cm}\times\hat G_{i_2j_2}^2\Bigg(\Bigg(\sum_{k\in\NN_{i_2j_2}} Y_{i_2k}Y_{j_2k}\Bigg)^2+\Bigg(\sum_{k\in\NN_{i_2j_2}}\hat M_{i_2k}Y_{j_2k}\Bigg)^2\Bigg)\ind_A\Bigg\}\label{equation: first second term}\\
&\hspace{0.5cm}-\frac{1}{d^2}\hspace{-2mm}\sum_{\substack{i_1,i_2,j_1,j_2=1\\ \{i_1,j_1\}\cap\{i_2,j_2\}=\emptyset }}^d\hspace{-2mm}\E\Bigg\{\hat G_{i_1j_1}^2\Bigg(\Bigg(\sum_{k\in\NN_{i_1j_1}} Y_{i_1k}Y_{j_1k}\Bigg)^2+\Bigg(\sum_{k\in\NN_{i_1j_1}}\hat M_{i_1k}Y_{j_1k}\Bigg)^2\Bigg)\ind_A \Bigg\}\notag \\
&\hspace{2.8cm}\times\E\Bigg\{\hat G_{i_2j_2}^2\Bigg(\Bigg(\sum_{k=\in\NN_{i_2j_2}} Y_{i_2k}Y_{j_2k}\Bigg)^2+\Bigg(\sum_{k\in\NN_{i_2j_2}}\hat M_{i_2k}Y_{j_2k}\Bigg)^2\Bigg)\ind_A\Bigg\} \label{equation: second second term} \\
&=I_1+I_2\notag,
\end{align}
where $I_1$ consists of the term (\ref{equation: first term}) and $I_2$ of (\ref{equation: first second term}) and (\ref{equation: second second term}).
The term $I_1$ yields
\begin{align*}
I_1
&\lesssim I_{1,1}+I_{1,2}+I_{1,3},
\end{align*}
with
\begin{align*}
I_{1,1}&=\frac{(\log n)^2}{d^2n^6}\sum_{\substack{i_1,i_2,j_1,j_2=1\\ \{i_1,j_1\}\cap\{i_2,j_2\}\neq\emptyset}}^d\sum_{k_1,k_2,k_3,k_4=1}^n\left|\E Y_{i_1k_1}Y_{j_1k_1}Y_{i_1k_2}Y_{j_1k_2}Y_{i_2k_3}Y_{j_2k_3}Y_{i_2k_4}Y_{j_2k_4}\right|,\\
I_{1,2}&=\frac{(\log n)^2}{d^2n^6}\sum_{\substack{i_1,i_2,j_1,j_2=1\\ \{i_1,j_1\}\cap\{i_2,j_2\}\neq\emptyset}}^d\sum_{k_1,k_2,k_3,k_4=1}^n\Big\arrowvert\E \Big(Y_{i_1k_1}Y_{j_1k_1}Y_{i_1k_2}Y_{j_1k_2}\hat M_{i_2k_3}Y_{j_2k_3}\hat M_{i_2k_4}Y_{j_2k_4}\\
&\hspace{5cm}\times\varepsilon_{i_1k_1}\varepsilon_{j_1k_1}\varepsilon_{i_1k_2}\varepsilon_{j_1k_2}\varepsilon_{i_2k_3}\varepsilon_{j_2k_3}\varepsilon_{i_2k_4}\varepsilon_{j_2k_4}\ind_A\Big)\Big\arrowvert,\\
I_{1,3}&=\frac{(\log n)^2}{d^2n^6}\sum_{\substack{i_1,i_2,j_1,j_2=1\\ \{i_1,j_1\}\cap\{i_2,j_2\}\neq\emptyset}}^d\sum_{k_1,k_2,k_3,k_4=1}^n\Big\arrowvert\E \Big(\hat M_{i_1k_1}Y_{j_1k_1}\hat M_{i_1k_2}Y_{j_1k_2}\hat M_{i_2k_3}Y_{j_2k_3}\hat M_{i_2k_4}Y_{j_2k_4}\\
&\hspace{5cm}\times\varepsilon_{i_1k_1}\varepsilon_{j_1k_1}\varepsilon_{i_1k_2}\varepsilon_{j_1k_2}\varepsilon_{i_2k_3}\varepsilon_{j_2k_3}\varepsilon_{i_2k_4}\varepsilon_{j_2k_4}\ind_A\Big)\Big\arrowvert.
\end{align*}
For $I_{1,1}$ we have 
\begin{align*}
I_{1,1}
&=\frac{(\log n)^2}{d^2n^6}\sum_{\substack{i_1,i_2,j_1,j_2=1\\ \{i_1,j_1\}\cap\{i_2,j_2\}\neq\emptyset \\ i_1\neq j_1\vee i_2\neq j_2}}^d\sum_{k_1,k_2,k_3,k_4=1}^n|\E Y_{i_1k_1}Y_{j_1k_1}Y_{i_1k_2}Y_{j_1k_2}Y_{i_2k_3}Y_{j_2k_3}Y_{i_2k_4}Y_{j_2k_4}|\\
&\hspace{0.5cm}+\frac{(\log n)^2}{d^2n^6}\sum_{i=1}^d\sum_{k_1,k_2,k_3,k_4=1}^n\E Y_{ik_1}^2Y_{ik_2}^2Y_{ik_3}^2Y_{ik_4}^2\\
&\lesssim \frac{n^{4\alpha}(\log n)^2}{n^4},
\end{align*}
where we used for $i_1,j_1,i_2,j_2$ with $\{i_1,j_1\}\cap\{i_2,j_2\}\neq\emptyset$ and $i_1\neq j_1$ or $i_2\neq j_2$ the bounds
\begin{align*}
|\E Y_{i_1k_1}Y_{j_1k_1}Y_{i_1k_2}Y_{j_1k_2}Y_{i_2k_3}Y_{j_2k_3}Y_{i_2k_4}Y_{j_2k_4}|\lesssim
\begin{cases}
 n^2d^{4\alpha-2}          &\text{ for }\#\{k_1,k_2,k_3,k_4\}=1\\
 nd^{2\alpha-1}              &\text{ for }\#\{k_1,k_2,k_3,k_4\}=2\\
1              &\text{ for }\#\{k_1,k_2,k_3,k_4\}=3\\
n^{-2}d^{2-4\alpha}                   &\text{ for }\#\{k_1,k_2,k_3,k_4\}=4
\end{cases}
\end{align*}
and for $i=i_1=j_1=i_2=j_2$ the estimates
\begin{align*}
\E Y_{ik_1}^2Y_{ik_2}^2Y_{ik_3}^2Y_{ik_4}^2\lesssim
\begin{cases}
n^3d^{6\alpha-3}           &\text{ for }\#\{k_1,k_2,k_3,k_4\}=1\\
n^2d^{4\alpha-2}             &\text{ for }\#\{k_1,k_2,k_3,k_4\}=2\\
nd^{2\alpha-1}           &\text{ for }\#\{k_1,k_2,k_3,k_4\}=3\\
1                &\text{ for }\#\{k_1,k_2,k_3,k_4\}=4.
\end{cases}
\end{align*}
These estimates are deduced by the following consideration. First, the expectation is factorized by independence into a product of moments of the $Y_{ik}$'s. Then applying \eqref{eq: tilde1} and \eqref{eq: tilde}, the $l$-th moment is bounded by
$$\left|\E Y_{ik}^l\right|\lesssim \left(n^{1/2}d^{\alpha-1/2}\right)^{l-2}, \ \ l\in\N.$$
Now we evaluate $I_{1,2}$. Using \eqref{eq: C on A} in \eqref{eq: I12}
\begin{align}
I_{1,2}&=\frac{(\log n)^2}{d^2n^6}\sum_{\substack{i_1,i_2,j_1,j_2=1\\ \{i_1,j_1\}\cap\{i_2,j_2\}\neq\emptyset}}^d\notag \\
&\hspace{0.5cm}\sum_{k_1,\dots,k_6=1}^n\big|\E Y_{i_1k_1}Y_{j_1k_1}Y_{i_1k_2}Y_{j_1k_2}Y_{j_2k_3}Y_{j_2k_4}Y_{i_2k_5}Y_{i_2k_6}\big|\notag \\
&\hspace{2.2cm}\times \E \Big(\frac{1}{n^2}\hat W_{i_2i_2}^2\varepsilon_{i_1k_1}\varepsilon_{j_1k_1}\varepsilon_{i_1k_2}\varepsilon_{j_1k_2}\varepsilon_{i_2k_3}\varepsilon_{j_2k_3}\varepsilon_{i_2k_4}\varepsilon_{j_2k_4}\varepsilon_{i_2k_5}\varepsilon_{i_2k_6}\ind_A\Big)\notag\\
&\lesssim \frac{(\log n)^2}{d^2n^8}\sum_{\substack{i_1,i_2,j_1,j_2=1\\ \{i_1,j_1\}\cap\{i_2,j_2\}\neq\emptyset}}^d\sum_{k_1,\dots,k_6=1}^n |\E Y_{i_1k_1}Y_{j_1k_1}Y_{i_1k_2}Y_{j_1k_2}Y_{i_2k_5}Y_{j_2k_3}Y_{i_2k_6}Y_{j_2k_4}| \label{eq: I12}\\
&\lesssim \frac{(\log n)^2d^{6\alpha}}{d^2n^4},\notag
\end{align}
where we used for the bound
\begin{align*}
|\E Y_{i_1k_1}Y_{j_1k_1}&Y_{i_1k_2}Y_{j_1k_2}Y_{i_2k_5}Y_{j_2k_3}Y_{i_2k_6}Y_{j_2k_4}|\\ 
&\lesssim \left(\frac{d}{n}\right)^{i-4}d^{2\alpha(4-i)} \ \ \text{for }i=\#\{k_1,k_2,k_3,k_4,k_5,k_6\}.
\end{align*}
Again by \eqref{eq: C on A}, we obtain with the same argument as for $I_{1,2}$
\begin{align*}
I_{1,3}&\lesssim \frac{(\log n)^2}{d^2n^{10}}\sum_{\substack{i_1,i_2,j_1,j_2=1\\ \{i_1,j_1\}\cap\{i_2,j_2\}\neq\emptyset}}^d \ \sum_{k_1,\dots,k_8=1}^n|\E Y_{i_1k_5}Y_{i_1k_6}Y_{j_1k_1}Y_{j_1k_2}Y_{i_2k_7}Y_{i_2k_8}Y_{j_2k_3}Y_{j_2k_4}|\\
&\lesssim \frac{(\log n)^2d^{6\alpha}}{d^2n^6}
\end{align*}
with
\begin{align*}
|\E Y_{i_1k_5}Y_{i_1k_6}&Y_{j_1k_1}Y_{j_1k_2}Y_{i_2k_7}Y_{i_2k_8}Y_{j_2k_3}Y_{j_2k_4}|\\
&\lesssim  \left(\frac{d}{n}\right)^{i-4}d^{2\alpha(4-i)} \ \ \text{for }i=\#\{k_1,k_2,k_3,k_4,k_5,k_6,k_7,k_8\}.
\end{align*}
As concerns $I_2$, define
\begin{align*}
U_{ij,d,n}=\hat G_{ij,d,n}^2\Bigg\{\Bigg(\sum_{k\in\NN_{ij,d,n}} Y_{ik,d,n}Y_{jk,d,n}\Bigg)^2+\Bigg(\sum_{k\in\NN_{ij,d,n}}\hat M_{ik,d,n}Y_{jk,d,n}\Bigg)^2\Bigg\},
\end{align*}
and note that $U_{ij,d,n}$ is bounded by a constant multiple of $n^6d^{4\alpha-2}$ because $\NN_{ij,d,n}$ contains at most $n$ elements, $\hat G_{ij,d,n}^2\lesssim 1$ since by Subsection \ref{subsec: modify} $\min_i p_{i,d,n}$ is uniformly bounded away from zero, $|Y_{ik,d,n}|\lesssim n^{1/2}d^{\alpha-1/2}$ by Subsection \ref{subsec: trunc T} and Subsection \ref{subsec: trunc X 1}, 
$$|\hat M_{ik,d,n}|=\left|\frac{1}{N_{ii,d,n}}\sum_{l\in\NN_{ii,d,n}}Y_{il,d,n}\right|\lesssim n^{1/2}d^{\alpha-1/2}.$$ 
Hence,
\begin{align*}
I_2&=\frac{1}{d^2}\sum_{\substack{i_1,i_2,j_1,j_2=1\\ \{i_1,j_1\}\cap\{i_2,j_2\}=\emptyset }}^d\E\left( U_{i_1j_1}U_{i_2j_2}\ind_A\right)-\E\left( U_{i_1j_1}\ind_A\right)\E\left( U_{i_2j_2}\ind_A\right)\\
&=\frac{1}{d^2}\sum_{\substack{i_1,i_2,j_1,j_2=1\\ \{i_1,j_1\}\cap\{i_2,j_2\}=\emptyset }}^d\Big\{-\E\left( U_{i_1j_1}U_{i_2j_2}\ind_{A^c}\right)+\E\left( U_{i_1j_1}\ind_{A^c}\right)\E\left( U_{i_2j_2}\right)\\
&\hspace{4cm}+\E\left( U_{i_1j_1}\right)\E\left( U_{i_2j_2}\ind_{A^c}\right)-\E \left(U_{i_1j_1}\ind_{A^c}\right)\E\left( U_{i_2j_2}\ind_{A^c}\right)\Big\}\\
&\le \frac{2}{d^2}\sum_{\substack{i_1,i_2,j_1,j_2=1\\ \{i_1,j_1\}\cap\{i_2,j_2\}=\emptyset }}^d\E \left(U_{i_1j_1}\right)\E \left( U_{i_2j_2}\ind_{A^c}\right)\\
&\lesssim d^{8\alpha-2}n^{12}\P(A^c)\\
&\lesssim n^{12+8\alpha-2\gamma^2}.
\end{align*}
Note that by choice of $\gamma$ in \eqref{eq: gamma} the exponent in the last line is strictly smaller than $-1$. Therefore by the lemma of Borel-Cantelli $h_{d,n}\ind_{A_{d,n}}\hspace{-1mm}\rightarrow 0$ almost surely ($d\to\infty$). In the following subsection we redefine the matrix $\hat T_{d,n}$ by $\tilde T_{d,n}$.

\subsection{Step VI: Removing $n^{-1}W\circ((Y\circ\varepsilon)(\hat M\circ\varepsilon)^\ast +(\hat M\circ\varepsilon)(Y\circ\varepsilon)^\ast )$}
By the same arguments as in Subsection $\ref{subsec: trunc X 1}$ we return to the original centered and standardized matrix $X_{d,n}$. Define
$$\tilde T_{d,n}=\frac{1}{n}W_{d,n}\circ \Big((Y_{d,n}\circ \varepsilon_{d,n})(Y_{d,n}\circ\varepsilon_{d,n})^\ast \Big).$$
We prove that $$d_{L}(\mu^{T_{d,n}},\mu^{\tilde T_{d,n}})\rightarrow 0$$ almost surely. For $\gamma>1$, define the event
$$\tilde A_{d,n}=\left\{\max_i |N_{ii,d,n}-np_{i,d,n}|<\gamma\sqrt{n\log{n}} \right\}.$$
Note that
$$\left\{\max_i\left|N_{ii,d,n}-np_{i,d,n}\right|<\gamma\sqrt{n\log n}\right\}=\left\{\max_i\left|\sum_{k=1}^n\left(\varepsilon_{ik,d,n}-p_{i,d,n}\right)\right|<\gamma\sqrt{n\log n}\right\}$$
for $d$ sufficently large. The union bound and Hoeffding's inequality yield
\begin{equation}
\P\left(\tilde A^c_{d,n}\right)\le 2dn^{-2\gamma^2}\label{eq: bound A}
\end{equation}
By the Borel-Cantelli Lemma all but finitely many of the events $(A_d)$ occur. 
Moreover, for $\frac{1}{2}<\eta<1$ define the event $$B_{d,n}=\left\{\sum_{i=1}^d\ind\left\{|\hat m_{i,d,n}|>\sqrt\frac{d^{2(1-\eta)}}{n}\right\}<d^{\eta}\right\}.$$
First observe that by the same type of argument as used in \eqref{eq: C on A} and by Markov's inequality
\begin{align*}
\max_i\P&\left(|\hat m_{i,d,n}|>\sqrt\frac{d^{2(1-\eta)}}{n},\tilde A_{d,n}\right)\\
&\le \max_i\P\left(\frac{2}{n\min p_{i,d,n}}\left|\sum_k\varepsilon_{ik,d,n}Y_{ik,d,n}\right|>\sqrt\frac{d^{2(1-\eta)}}{n}\right)\\
&\le \frac{4\E\left(\sum_k\varepsilon_{ik,d,n}Y_{ik,d,n}\right)^2}{n^2\min_i p_{i,d,n}^2\frac{d^{2(1-\eta)}}{n}}\\
&\lesssim d^{2\eta -2},
\end{align*}
where we have used
$$\frac{1}{N_{ii,d,n}}\le \frac{2}{n\min p_{i,d,n}}$$ 
for $d$ sufficiently large in the first inequality.
In particular,
\begin{align*}
\E \ind\left\{|\hat m_{i,d,n}|>\sqrt\frac{d^{2(1-\eta)}}{n}\right\}&=\E \ind\left\{|\hat m_{i,d,n}|>\sqrt\frac{d^{2(1-\eta)}}{n}\right\}\left(\ind_{\tilde A_{d,n}}+\ind_{\tilde A_{d,n}^c}\right)\\
&\le\kappa \left(d^{2\eta-2}+dn^{-2\gamma^2}\right)
\end{align*}
for some suitably chosen constant $\kappa>0$.
We conclude for $d$ sufficiently large by Hoeffding's inequality
\begin{align*}
\P\left(B_{d,n}^c\right)&\le\P\Bigg(\sum_{i=1}^d\ind\left\{\hat m_{i,d,n}>\sqrt\frac{d^{2(1-\eta)}}{n}\right\}-\E\ind\left\{\hat m_{i,d,n}>\sqrt\frac{d^{2(1-\eta)}}{n}\right\}\\
&\hspace{6cm}> d^{\eta}-\kappa \left(d^{2\eta-1}-d^{2}n^{-2\gamma^2}\right)\Bigg)\\
&\le \P\Bigg(\sum_{i=1}^d\ind\left\{\hat m_{i,d,n}>\sqrt\frac{d^{2(1-\eta)}}{n}\right\}-\E\ind\left\{\hat m_{i,d,n}>\sqrt\frac{d^{2(1-\eta)}}{n}\right\}>\frac{1}{2}d^{\eta}\Bigg)\\
&\le \exp\left(-\frac{d^{2\eta-1}}{2}\right).
\end{align*}
By the Borel-Cantelli Lemma all but finitely many of the events $(B_{d,n})$ occur. \\
Let $\gamma'>0$ be an appropriate constant such that for all $n$
$$2\sum_{k=1}^n\E|Y_{ik,d,n}|\le \gamma' n.$$
Then, define the event
$$D_{d,n}=\left\{ \sum_{i=1}^d\ind\left\{\sum_{k=1}^n |Y_{ik,d,n}|>\gamma'n\right\}\le \frac{d}{\log d}\right\}.$$
In the next step we shall prove that $\P(\limsup_{d} D_{d,n}^c)=0$ in order to remove the corresponding rows from the matrix $Y$. By Chebychev's inequality we have 
\begin{align*}
\max_i \P&\left(\sum_{k=1}^n|Y_{ik,d,n}|>\gamma'n\right)\\
&\le \max_i \P\left(\sum_{k=1}^n|Y_{ik,d,n}|-\E|Y_{ik,d,n}|>\gamma'n-\sum_{k=1}^n\E|Y_{ik,d,n}|\right)\\
&\le \max_i \P\left(\sum_{k=1}^n|Y_{ik,d,n}|-\E|Y_{ik,d,n}|>\frac{1}{2}\gamma'n\right)\\
&\le \frac{\kappa'}{n}
\end{align*}
for an appropriate constant $\kappa'>0$. Again, by the Hoeffding inequality for sufficiently large $d$,
\begin{align*}
\P(D_{d,n}^c)&\le \P \Bigg(\sum_{i=1}^d\ind\left\{\sum_{k=1}^n|Y_{ik,d,n}|>\gamma'n\right\}-\E \ind\left\{\sum_{k=1}^n|Y_{ik,d,n}|>\gamma'n\right\}> \frac{d}{\log d}-\frac{\kappa'd}{n}\Bigg)\\
&\le \P \left(\sum_{i=1}^d\ind\left\{\sum_{k=1}^n|Y_{ik,d,n}|>\gamma'n\right\}-\E \ind\left\{\sum_{k=1}^n|Y_{ik,d,n}|>\gamma'n\right\}> \frac{1}{2}\frac{d}{\log d}\right)\\
&\le \exp\left(-\frac{d}{2(\log d)^2}\right),
\end{align*}
and therefore $\P(\limsup_{d} D_{d,n}^c)=0$.
Now let
\begin{align*}
\check T_{d,n}&=\frac{1}{n}W_{d,n}\circ\bigg((Y_{d,n}\circ\varepsilon_{d,n})(Y_{d,n}\circ\varepsilon_{d,n})^\ast -(\tilde Y_{d,n}\circ \varepsilon_{d,n})(\tilde M_{d,n}\circ \varepsilon_{d,n})^\ast \\
&\hspace{4cm}-(\tilde M_{d,n}\circ \varepsilon_{d,n})(\tilde Y_{d,n}\circ \varepsilon_{d,n})^\ast \bigg),\end{align*}
where 
$$\tilde M_{ik,d,n}=\hat M_{ik,d,n}\ind\left\{|\hat M_{ik,d,n}|\le\sqrt\frac{d^{2(1-\eta)}}{n}\right\}$$ 
and
$$\tilde Y_{ik,d,n}=Y_{ik,d,n}\ind\left\{\sum_{l=1}^n|Y_{il,d,n}|\le\gamma'n\right\}.$$
By Theorem \ref{theorem: A43} and due to $\P(\limsup_d (D_{d,n}^c\cup B_{d,n}^c))=0$ we conclude by the same type of arguments as in Subsection \ref{subsec: trunc X 1}
\begin{align*}
d_L&\left(\mu^{\hat T_{d,n}},\mu^{\check T_{d,n}}\right)\\
&\hspace{0.2cm}\le \frac{1}{d}\rank\Bigg(\frac{1}{n}W_{d,n}\circ \Big((Y_{d,n}\circ \varepsilon_{d,n})(\hat M_{d,n}\circ \varepsilon_{d,n})^\ast +(\hat M_{d,n}\circ \varepsilon_{d,n})(Y_{d,n}\circ \varepsilon_{d,n})^\ast \\
&\hspace{3cm}-(\tilde Y_{d,n}\circ \varepsilon_{d,n})(\tilde M_{d,n}\circ \varepsilon_{d,n})^\ast -(\tilde M_{d,n}\circ \varepsilon_{d,n})(\tilde Y_{d,n}\circ \varepsilon_{d,n})^\ast \Big)\Bigg)\\
&\hspace{0.2cm}\overset{\text{a.s.}}{\longrightarrow} 0\ \ \text{as} \ \ d\to\infty.
\end{align*}
In order to save space the explicit dependence on $d$ and $n$ is suppressed in the displays until the end of the section. By Theorem \ref{theorem: inequality Levy distance trace},
\begin{align}
d_L^3&(\mu^{\check T},\mu^{\tilde T})\\
&\le \frac{1}{d}\tr\Bigg(\left(\frac{1}{n}W\circ \left((\tilde Y\circ \varepsilon)(\tilde M\circ \varepsilon)^\ast +(\tilde M\circ \varepsilon)(\tilde Y\circ \varepsilon)^\ast \right)\right)\notag\\
&\hspace{3cm}\times\left(\frac{1}{n}W\circ \left((\tilde Y\circ \varepsilon)(\tilde M\circ \varepsilon)^\ast+(\tilde M\circ \varepsilon)(\tilde Y\circ \varepsilon)^\ast \right)\right)\Bigg)\notag\\
&= \frac{2}{d}\tr\Bigg(\frac{1}{n^2}W^2\circ \Big((\tilde M\circ \varepsilon)(\tilde Y\circ \varepsilon)^\ast(\tilde Y\circ \varepsilon)(\tilde M\circ \varepsilon)^\ast +(\tilde Y\circ \varepsilon)(\tilde M\circ \varepsilon)^\ast(\tilde Y\circ \varepsilon)(\tilde M\circ \varepsilon)^\ast \Big)\Bigg)\notag\\
&\le \frac{4}{d}\tr\Bigg(\frac{1}{n^2}W^2\circ \left((\tilde M\circ \varepsilon)(\tilde Y\circ \varepsilon)^\ast (\tilde Y\circ \varepsilon)(\tilde M\circ \varepsilon)^\ast \right)\Bigg)\label{eq: elementary} \\
&\le \frac{4}{d}\sum_{i=1}^d \hat m_i^2\ind\left\{|\hat m_{i}|\le\sqrt\frac{d^{2(1-\eta)}}{n}\right\}\notag\\
&\hspace{3cm}\times\sum_{j=1}^d\frac{1}{n^2}W_{ij}^2\left(\sum_{k=1}^n\varepsilon_{ik}\varepsilon_{jk}Y_{jk}\right)^2\ind\left\{\sum_{l=1}^n|Y_{jl}|\le\gamma'n\right\}\notag \\
&\lesssim \frac{d^{2(\eta-1)}}{dn^3}\sum_{i=1}^d  \sum_{j=1}^d\left(\sum_{k=1}^n\varepsilon_{ik}\varepsilon_{jk}Y_{jk}\right)^2\ind\left\{\sum_{l=1}^n|Y_{jl}|\le\gamma'n\right\}\label{eq: J},
\end{align}
where we have used the elementary inequality
$$\tr(C^2)\le \tr(CC^\ast) \ \ \text{for any } C\in \R^{d\times d}$$
in \eqref{eq: elementary}.
It remains to prove that the last line \eqref{eq: J} converges to zero almost surely.  Let $\eta<\eta'<1$, and rewrite
\begin{align}
\max_{i}\P\Bigg(\sum_{j=1}^d&\left(\sum_{k=1}^n\varepsilon_{ik}\varepsilon_{jk}Y_{jk}\right)^2\ind\left\{\sum_{l=1}^n|Y_{jl}|\le\gamma'n\right\}\ge \frac{n^3}{d^{2(\eta'-1)}}\Bigg)\notag\\
&\hspace{-0.6cm}= \max_{i}\E\left\{\P\left(\sum_{j=1}^d\left(\sum_{k=1}^n\varepsilon_{ik}\varepsilon_{jk}Y_{jk}\right)^2\ind\left\{\sum_{l=1}^n|Y_{jl}|\le\gamma'n\right\}\ge \frac{n^3}{d^{2(\eta'-1)}}~\Bigg|~\varepsilon\right)\right\}\label{eq: conditional}
\end{align}
Define for $\eta'<\eta''<1$ the random variables
\begin{align*}
I_{ij,d,n}=\ind\left\{\left|\sum_{l=1}^n\varepsilon_{il,d,n}\varepsilon_{jl,d,n}Y_{jl,d,n}\right|\ge \sqrt{n}d^{(\eta''-1)}\right\},~1\le i,j \le d.
\end{align*}
Then by Markov's inequality for the conditional probability and an appropriate constant $\kappa''>0$,
\begin{align*}
\E (I_{ij}\big|\varepsilon)&=\P\left(\left|\sum_{l=1}^n\varepsilon_{il}\varepsilon_{jl}Y_{jl}\right|\ge \sqrt{nd^{2(\eta''-1)}}~\Bigg|~\varepsilon\right)\le \sum_{l=1}^n\frac{\varepsilon_{il}\varepsilon_{jl}\E Y_{jl}^2}{nd^{2(\eta''-1)}}\le \frac{\kappa''}{d^{2(\eta''-1)}}.
\end{align*}
The inner conditional probability in line \eqref{eq: conditional} can be further estimated by
\begin{align*}
&\P\left(\sum_{j=1}^d\left(\sum_{k=1}^n\varepsilon_{ik}\varepsilon_{jk}Y_{jk}\right)^2\ind\left\{\sum_{l=1}^n|Y_{jl}|\le\gamma'n\right\}\ge \frac{n^3}{d^{2(\eta'-1)}}~\Bigg|~\varepsilon\right)\\
&\hspace{1cm}\le \P\left((\gamma'n)^2\sum_{j=1}^d\ind\left\{\sqrt{n}d^{(\eta''-1)}\le\left|\sum_{l=1}^n\varepsilon_{il}\varepsilon_{jl}Y_{jl}\right|\le \gamma'n\right\}\ge \frac{n^3}{2d^{2(\eta'-1)}}~\Bigg|~\varepsilon\right)\\
&\hspace{1.7cm}+P\left(nd^{2(\eta''-1)}\sum_{j=1}^d\ind\left\{\left|\sum_{l=1}^n\varepsilon_{il}\varepsilon_{jl}Y_{jl}\right|\le \sqrt{n}d^{(\eta''-1)}\right\}\ge \frac{n^3}{2d^{2(\eta'-1)}}~\Bigg|~\varepsilon\right),
\end{align*}
where the last conditional probability disappears for $d$ sufficiently large. For the first probability on the right hand side, we obtain
\begin{align*}
 \P&\left((\gamma'n)^2\sum_{j=1}^d\ind\left\{\sqrt{n}d^{(\eta''-1)}\le\left|\sum_{l=1}^n\varepsilon_{il}\varepsilon_{jl}Y_{jl}\right|\le \gamma'n\right\}\ge \frac{n^3}{2d^{2(\eta'-1)}}~\Bigg|~\varepsilon\right)\\
 &\leq \P\left((\gamma'n)^2\sum_{j=1}^d\Big(I_{ij}-\E(I_{ij}|\varepsilon)\Big)\ge \frac{n^3}{2d^{2(\eta'-1)}}-\kappa''\frac{(\gamma'n)^2d}{d^{2(\eta''-1)}}\Bigg|~\varepsilon\right)\\
&\le \P\left(\sum_{j=1}^d\Big(I_{ij}-\E(I_{ij}|\varepsilon)\Big)\ge \frac{n}{4\gamma'^2d^{2(\eta'-1)}}\Bigg|~\varepsilon\right)
\end{align*}
for $d$ sufficiently large. Finally, by Hoeffding's inequality the last line is bounded by
$$\exp\left(-\frac{n^2}{8\gamma'^4d^{4\eta'-3}}\right).$$
Altogether, \eqref{eq: J} is bounded by $d^{2(\eta-\eta')}$ with probability
\begin{align*}
&1-\P\left(\frac{d^{2(\eta-1)}}{dn^3}\sum_{i=1}^d\sum_{j=1}^d\left(\sum_{k=1}^n\varepsilon_{ik}\varepsilon_{jk}Y_{jk}\right)^2\ind\left\{\sum_{l=1}^n|Y_{jl}|\le\gamma'\right\}\ge d^{2(\eta-\eta')}\right)\\
&\hspace{1cm}\ge1-d\max_{i}\P\left(\sum_{j=1}^d\left(\sum_{k=1}^n\varepsilon_{ik}\varepsilon_{jk}Y_{jk}\right)^2\ind\left\{\sum_{l=1}^n|Y_{jl}|\le\gamma'n\right\}\ge \frac{n^3}{d^{2(\eta'-1)}}\right)\\
&\hspace{1cm}\ge1-d\exp\left(-\frac{n^2}{8\gamma'^4d^{4\eta'-3}}\right).
\end{align*}
By the Lemma of Borel-Cantelli, $$d_L^3(\mu^{\check T_{d,n}},\mu^{\tilde T_{d,n}})\to 0$$ almost surely. Consequently,
\begin{align*}
d_L(\mu^{\hat T_{d,n}},\mu^{\tilde T_{d,n}})&\le d_L(\mu^{\hat T_{d,n}},\mu^{\check T_{d,n}})+d_L(\mu^{\check T_{d,n}},\mu^{\tilde T_{d,n}})\overset{\text{a.s.}}{\longrightarrow} 0\ \text{ as $d\to\infty$.}
\end{align*}
Subsequently, we denote $\tilde T_{d,n}$ by $\hat T_{d,n}$.

\subsection{Step VII: Diagonal manipulation} Rewrite the matrix $\hat T_{d,n}$ in the following way
\begin{align*}
\hat T_{d,n}&=\frac{1}{n}(w_{d,n}w_{d,n}^\ast)\circ \Big((Y_{d,n}\circ \varepsilon_{d,n})(Y_{d,n}\circ\varepsilon_{d,n})^\ast\Big)\\
&\hspace{0.5cm}-\diag\Bigg[\frac{1}{n}(w_{d,n}w_{d,n}^\ast)\circ \Big((Y_{d,n}\circ \varepsilon_{d,n})(Y_{d,n}\circ\varepsilon_{d,n})^\ast\Big)\\
&\hspace{3cm}-\frac{1}{n}W_{d,n}\circ \Big((Y_{d,n}\circ \varepsilon_{d,n})(Y_{d,n}\circ\varepsilon_{d,n})^\ast\Big)\Bigg].
\end{align*}
In this step we replace the diagonal matrix 
\begin{align*}
\hat S_{d,n}:&=\diag\Bigg[\frac{1}{n}(w_{d,n}w_{d,n}^\ast)\circ \Big((Y_{d,n}\circ \varepsilon_{d,n})(Y_{d,n}\circ\varepsilon_{d,n})^\ast \Big)\\
&\hspace{2cm}-\frac{1}{n}W_{d,n}\circ \Big((Y_{d,n}\circ \varepsilon_{d,n})(Y_{d,n}\circ\varepsilon_{d,n})^\ast \Big)\Bigg]
\end{align*}
by its diagonal deterministic counterpart $S_{d,n}$ with $$S_{ii,d,n}=\frac{1-p_{i,d,n}}{p_{i,d,n}}T_{ii,d,n},~i=1,...,d.$$ 
Thereto, we use similar arguments as in the last subsection. In contrast to the last subsection we cannot simply rely on Markov's inequality since $Y_{ik,d,n}$ is assumed to possess only two moments. In order to save space the explicit dependence on $d$ and $n$ is suppressed in the displays until the end of the section. Note that for any $u>0$,
\begin{align*}
\alpha_\text{max}&=\max_{i=1,...,d}\P\left(\left|\hat S_{ii}-S_{ii}\right|>u\right)\\
&=\max_{i=1,...,d}\P\left(\left|\frac{1-p_i}{np_i}\sum_{k=1}^n\left(Y_{ik}^2\frac{\varepsilon_{ik}}{p_i}-T_{ii}\right)\right|>u\right)\\
&\le \max_{i=1,...,d}\P\Bigg(\left|\frac{1-p_i}{p_i}T_{ii}-\frac{1-p_i}{np_i}\sum_{k=1}^nY_{ik}^2\frac{\varepsilon_{ik}}{p_i}\right|>u,~\left|\sum_{k=1}^n(\varepsilon_{ik}-p_i)\right|>\sqrt{n\log n}\Bigg)\\
&\hspace{0.5cm}+\max_{i=1,...,d}\P\Bigg(\left|\frac{1-p_i}{p_i}T_{ii}-\frac{1-p_i}{np_i}\sum_{k=1}^nY_{ik}^2\frac{\varepsilon_{ik}}{p_i}\right|>u,~\left|\sum_{k=1}^n(\varepsilon_{ik}-p_i)\right|\le \sqrt{n\log n}\Bigg).\\
\end{align*}
As concerns the first term in this last inequality, Hoeffding's inequality yields
\begin{align*}
&\max_{i=1,...,d}\P\left(\left|\frac{1-p_i}{p_i}T_{ii}-\frac{1-p_i}{np_i}\sum_{k=1}^nY_{ik}^2\frac{\varepsilon_{ik}}{p_i}\right|>u,~\left|\sum_{k=1}^n(\varepsilon_{ik}-p_i)\right|>\sqrt{n\log n}\right)\\
&\hspace{1cm}\le \max_{i=1,...,d} \P\left(\left|\sum_{k=1}^n(\varepsilon_{ik}-p_i)\right|>\sqrt{n\log n}\right)\\
&\hspace{1cm}\le 2n^{-2}.
\end{align*}
In order to bound the second term, note that
\begin{align}
&\max_{i=1,...,d}\P\left(\left|\frac{1-p_i}{p_i}T_{ii}-\frac{1-p_i}{np_i}\sum_{k=1}^nY_{ik}^2\frac{\varepsilon_{ik}}{p_i}\right|>u,~\left|\sum_{k=1}^n(\varepsilon_{ik}-p_i)\right|\le \sqrt{n\log n}\right)\notag \\
&=\max_{i=1,...,d}\sum_{l=\lceil np_i-\sqrt{n\log n}\rceil}^{\lfloor np_i+\sqrt{n\log n}\rfloor}\P\left(\left|\frac{1-p_i}{p_i}T_{ii}-\frac{1-p_i}{np_i}\sum_{k=1}^nY_{ik}^2\frac{\varepsilon_{ik}}{p_i}\right|>u,~\sum_{k=1}^n\varepsilon_{ik}=l\right)\notag \\
&=\max_{i=1,...,d}\sum_{l=\lceil np_i-\sqrt{n\log n}\rceil}^{\lfloor np_i+\sqrt{n\log n}\rfloor}\P\left(\left|\frac{1-p_i}{p_i}T_{ii}-\frac{1-p_i}{np_i}\sum_{k=1}^nY_{ik}^2\frac{\varepsilon_{ik}}{p_i}\right|>u~\Bigg|~\sum_{k=1}^n\varepsilon_{ik}=l\right)\notag\\
&\hspace{8cm}\times\P\left(\sum_{k=1}^n\varepsilon_{ik}=l\right)\notag \\
&=\max_{i=1,...,d}\sum_{l=\lceil np_i-\sqrt{n\log n}\rceil}^{\lfloor np_i+\sqrt{n\log n}\rfloor}\P\left(\left|\frac{1-p_i}{p_i}T_{ii}-\frac{1-p_i}{np_i}\sum_{k=1}^l\frac{Y_{ik}^2}{p_i}\right|>u\right)\P\left(\sum_{k=1}^n\varepsilon_{ik}=l\right) \label{eq: bedingt},
\end{align}
where the last identity holds true because $Y_{i1,d,n},\dots,Y_{in,d,n}$ are iid and jointly independent of $\varepsilon_{d,n}$. By the elementary inequality
\begin{align*}
\left|T_{ii}-\frac{1}{n}\sum_{k=1}^l\frac{Y_{ik}^2}{p_i}\right|&\le \left|T_{ii}-\frac{1}{n}\sum_{k=1}^{\lceil np_i-\sqrt{n\log n}\rceil}\frac{Y_{ik}^2}{p_i}\right| ~\vee~ \left|T_{ii}-\frac{1}{n}\sum_{k=1}^{\lfloor np_i+\sqrt{n\log n}\rfloor}\frac{Y_{ik}^2}{p_i}\right|,
\end{align*}
we conclude 
\begin{align}
\eqref{eq: bedingt}&\le \max_{i=1,...,d}\P\left(\left|\frac{1-p_i}{p_i}T_{ii}-\frac{1-p_i}{np_i}\sum_{k=1}^{\lceil np_i-\sqrt{n\log n}\rceil}\frac{Y_{ik}^2}{p_i}\right|>u\right)\notag \\
&\hspace{0.5cm}+\max_{i=1,...,d}\P\left(\left|\frac{1-p_i}{p_i}T_{ii}-\frac{1-p_i}{np_i}\sum_{k=1}^{\lfloor np_i+\sqrt{n\log n}\rfloor}\frac{Y_{ik}^2}{p_i}\right|>u\right)\notag \\
&\le 2\max_{i=1,...,d}\left[\P\left(\left|\frac{1-p_i}{p_i}T_{ii}-\frac{1-p_i}{np_i^2}\sum_{k=1}^{\lfloor np_i \rfloor}Y_{ik}^2\right|>\frac{u}{2}\right)\right.\notag\\
&\hspace{5cm}+\left.\P\left(\frac{1-p_i}{np_i^2}\sum_{k=1}^{\lceil \sqrt{n\log n}\rceil +1}Y_{ik}^2>\frac{u}{2}\right) \right]\notag \\
&\le 2\max_{i=1,...,d}\Bigg[\P\Bigg(\Bigg|\frac{1-p_i}{p_i}T_{ii}-\frac{1-p_i}{np_i^2}\sum_{k=1}^{\lfloor np_i \rfloor}T_{ii}X_{ik}^2\Bigg|>\frac{u}{2}\Bigg)\notag\\
&\hspace{5cm}+\frac{2T_{ii}(1-p_i)}{up_i^2}\left(\sqrt{\frac{\log n}{n}}+\frac{2}{n}\right)\Bigg].\notag
\end{align}
For $n$ sufficiently large, the last expression is bounded by
\begin{align}
2\max_{i=1,...,d}\Bigg[\P\Bigg(\Bigg|\frac{1}{\lfloor np_i \rfloor}\sum_{k=1}^{\lfloor np_i \rfloor}(X_{ik}^2-1)\Bigg|>\frac{up_i}{4(T_{ii}\vee 1)}\Bigg)+\frac{4T_{ii}(1-p_i)}{up_i^2}\sqrt{\frac{\log n}{n}}\Bigg].\label{eq: 3.22}
\end{align}
Note that by Subsection \ref{subsec: trunc T} and Subsection \ref{subsec: modify} 
$$\liminf_{d\to\infty} \min_{i=1,\dots,d}~ \frac{p_{i,d,n}}{T_{ii,d,n}\vee 1}>0 \ \ \text{ and }\ \ \liminf_{d\to\infty} \min_{i=1,\dots,d}\lfloor np_{i,d,n}\rfloor =\infty.$$
Hence, by the weak law of large numbers \eqref{eq: 3.22} converges to zero as $d\to \infty$ which implies $\alpha_{\max}\to 0$. 
Now, with $\alpha_i=\P\left(|\hat S_{ii}-S_{ii}|>u\right)$, $i=1,\dots,d$,
\begin{align*}
&\P\left(\sum_{i=1}^d\ind\left\{\left|\hat S_{ii}-S_{ii}\right|>u\right\}>2d\sqrt{\alpha_\text{max}\vee \sqrt{\frac{1}{d}}}\right)\\
&\hspace{2cm}\le \P\left(\sum_{i=1}^d\ind\left\{\left|\hat S_{ii}-S_{ii}\right|>u\right\}-\alpha_i>2d\sqrt{\alpha_\text{max}\vee \sqrt{\frac{1}{d}}}-d\alpha_\text{max}\right)\\
&\hspace{2cm}\le \P\left(\sum_{i=1}^d\ind\left\{\left|\hat S_{ii}-S_{ii}\right|>u\right\}-\alpha_i>d^{\frac{3}{4}}\right)\\
&\hspace{2cm}\le \exp\left(-2\sqrt{d}\right),
\end{align*}
where we used Hoeffding's inequality in the last line. Therefore,
\begin{align*}
\frac{1}{d}\sum_{i=1}^d\ind\left\{\left|\hat S_{ii,d,n}-S_{ii,d,n}\right|>u\right\}\overset{\text{a.s.}}{\longrightarrow}0
\end{align*}
as $d\to\infty$. Let $\tilde S_{d,n}$ be the diagonal matrix with entries $$\tilde S_{ii,d,n}=\hat S_{ii,d,n}\ind\left\{\left|\hat S_{ii,d,n}-S_{ii,d,n}\right|\le u\right\}.$$ We conclude by Theorem \ref{theorem: inequality Levy distance trace} and Theorem \ref{theorem: A43} that almost surely for sufficiently large $d$
\begin{align*}
&d_L\left(\mu^{\hat T_{d,n}},\mu^{\hat T_{d,n}-S_{d,n}+\hat S_{d,n}}\right)\\
&\hspace{1.5cm}\le d_L\left(\mu^{\hat T_{d,n}},\mu^{\hat T_{d,n}-\tilde S_{d,n}+\hat S_{d,n}}\right)+d_L\left(\mu^{\hat T_{d,n}-\tilde S_{d,n}+\hat S_{d,n}},\mu^{\hat T_{d,n}- S_{d,n}+\hat S_{d,n}}\right)\\
&\hspace{1.5cm}\le \frac{1}{d}\rank\left(\hat S_{d,n}-\tilde S_{d,n}\right)+\left(\frac{1}{d}\sum_{i=1}^d \left(S_{ii,d,n}-\tilde S_{ii,d,n}\right)^2\right)^{1/3}\\ 
&\hspace{1.5cm}\le \frac{1}{d}\sum_{i=1}^d\ind\left\{\left|\hat S_{ii,d,n}-S_{ii,d,n}\right|>u\right\}+u^{2/3}\\
&\hspace{1.5cm}\le 2u^{2/3}.
\end{align*}
Since the constant $u>0$ is chosen arbitrarily, we have
\begin{align*}
d_L\left(\mu^{\hat T_{d,n}},\mu^{\hat T_{d,n}-S_{d,n}+\hat S_{d,n}}\right)\overset{\text{a.s.}}{\longrightarrow}0
\end{align*}
for $d\to\infty$.
\subsection{Step VIII: Reverting the truncation}
Reverting finally the truncation steps I, III, IV yields the claim.

\medskip
\section{Proof of Proposition 6.1}\label{B}
\smallskip
Define $\hat X_{d,n}\in \R^{d\times n}$ by $\hat X_{ik,d,n}=X_{ik}\ind\{|X_{ik}|<\delta_{d,n} \sqrt{n}\}$. By Lemma 2.2 (truncation lemma)  of \cite{Bai1988} for $r=1/2$, given any preassigned decay rate to zero, there exists a sequence $(\delta_{d,n})$, $\delta_{d,n}\to 0$, with lower speed of convergence than that decay rate such that
$$\P\left(X_{d,n}\neq\hat X_{d,n} \ \text{infinitely often}\right)=0.$$
Let $(\delta_{d,n})$ be a sequence satisfying the truncation lemma with 
\begin{align}
\frac{1}{\sqrt{n}\delta_{d,n}^3}=o(1).\label{delta*}
\end{align} 
Therefore,
\begin{align*}
\limsup_{d\to \infty}\Bigg|&\left\Arrowvert\frac{1}{n}A_{d,n}\circ\left(\left(X_{d,n}\circ B_{d,n}\right)\left(X_{d,n}\circ B_{d,n}\right)^\ast\right)\right\Arrowvert_{S_\infty}\\
&-\left\Arrowvert\frac{1}{n}A_{d,n}\circ\left(\left(\hat X_{d,n}\circ B_{d,n}\right)\left(\hat X_{d,n}\circ B_{d,n}\right)^\ast\right)\right\Arrowvert_{S_\infty}\Bigg|= 0.
\end{align*}
Now let $\tilde X_{d,n}$ be the random matrix with entries $\tilde X_{ik,d,n}=\hat X_{ik,d,n}-\E \hat X_{ik,d,n}$. We prove
\begin{align*}
\limsup_{d\to \infty}\Bigg|&\left\Arrowvert\frac{1}{n}A_{d,n}\circ\left(\left(\tilde X_{d,n}\circ B_{d,n}\right)\left(\tilde X_{d,n}\circ B_{d,n}\right)^\ast\right)\right\Arrowvert_{S_\infty}\\
&-\left\Arrowvert\frac{1}{n}A_{d,n}\circ\left(\left(\hat X_{d,n}\circ B_{d,n}\right)\left(\hat X_{d,n}\circ B_{d,n}\right)^\ast\right)\right\Arrowvert_{S_\infty}\Bigg|= 0.
\end{align*}
As $\E X_{11}=0$, note first that
\begin{align}
\left|\E \hat X_{11,d,n}\right| &= \left\arrowvert \E X_{11}-\E X_{11}\ind\{\arrowvert X_{11}\arrowvert \geq\delta_n\sqrt{n}\}\right\arrowvert\nonumber\\\
&=\left\arrowvert \E X_{11}\ind\{\arrowvert X_{11}\arrowvert\geq\delta_n\sqrt{n}\}\right\arrowvert \nonumber\\
&\leq \E X_{11}^4 n^{-3/2}\delta_{d,n}^{-3}.\label{eq: bound EX}
\end{align}
Using the triangle inequality, the bound $\Arrowvert \cdot \Arrowvert_{S_\infty}\le \Arrowvert \cdot \Arrowvert_{S_2}$ as well as the inequality
$$\Arrowvert C \Arrowvert_{S_\infty}\le \max_{j=1,\dots,d}\sum_{i=1}^d |C_{ij}| \ \ \text{for symmetric }C\in\R^{d\times d}$$
in \eqref{eq: spectral1}, we conclude
\begin{align}
\Bigg|&\left\Arrowvert\frac{1}{n}A_{d,n}\circ\left(\left(\tilde X_{d,n}\circ B_{d,n}\right)\left(\tilde X_{d,n}\circ B_{d,n}\right)^\ast\right)\right\Arrowvert_{S_\infty}\notag\\
&\hspace{1.5cm}-\left\Arrowvert\frac{1}{n}A_{d,n}\circ\left(\left(\hat X_{d,n}\circ B_{d,n}\right)\left(\hat X_{d,n}\circ B_{d,n}\right)^\ast\right)\right\Arrowvert_{S_\infty}\Bigg|\notag\\
&\le \bigg\Arrowvert\frac{1}{n}A_{d,n}\circ\bigg(-\left(\hat{X}_{d,n}\circ B_{d,n}\right)\left(\E \hat{X}_{d,n}\circ B_{d,n}\right)^\ast-\left(B_{d,n}\circ\E \hat X_{d,n}\right)\notag\\
&\hspace{2cm}\times\left(\hat{X}_{d,n}\circ B_{d,n}\right)^\ast+\left(\E \hat{X}_{d,n}\circ B_{d,n}\right)\left(\E \hat{X}_{d,n}\circ B_{d,n} \right)^\ast\bigg)\bigg\Arrowvert_{S_\infty}\notag\\
&\le \frac{2}{n}\sqrt{\sum_{i,j=1}^dA_{ij,d,n}^2\left(\sum_{k=1}^n\hat X_{ik,d,n} B_{ik,d,n}B_{jk,d,n}\E \hat X_{jk}\right)^2}\label{eq: spectral1}\\
&\hspace{3.5cm}+d \max_{i,j}|A_{ij,d,n}| \left(\max_{ik}B_{ik,d,n}^2\right)\left(\E \hat X_{11,d,n}\right)^2\notag\\
&\le 2\sqrt\frac{d}{n}|\E \hat X_{11,d,n}|\max_{i,j} |A_{ij,d,n}|\left(\max_{ik}B_{ik,d,n}^2\right) \sqrt{d\max_{i=1,\dots,d}\sum_{k=1}^n{X}_{ik}^2}\label{eq: spectral2}\\
&\hspace{3.5cm}+d \max_{i,j}|A_{ij,d,n}| \left(\max_{ik}B_{ik,d,n}^2\right)\left(\E \hat X_{11,d,n}\right)^2\notag\\
&\longrightarrow 0 \ \ \text{a.s.},\notag
\end{align}
where the first summand in inequality \eqref{eq: spectral2} tends to $0$ by \eqref{eq: bound A B}, \eqref{delta*}, \eqref{eq: bound EX} and the Marcinkiewicz-Zygmund strong law of large numbers (cf. Lemma B.25 in \cite{Bai2010} with $\beta=1$ and $\alpha=3/4$). Since the entries of $\tilde X_{d,n}$ have all the same finite variance and
$\E \tilde X_{11,d,n}^2\rightarrow 1,$
we may assume for convergence statements about $$\left\Arrowvert\frac{1}{n}A_{d,n}\circ\left(\left(\tilde X_{d,n}\circ B_{d,n}\right)\left(\tilde X_{d,n}\circ B_{d,n}\right)^\ast\right)\right\Arrowvert_{S_\infty}$$ that the entries of $\tilde X_{d,n}$ to have unit variance. In order to apply the Lemma of Borel-Cantelli, we need to show that the probabilities
\begin{align*}
\P\left(\left\Arrowvert\frac{1}{n}A_{d,n}\circ\left(\left(\tilde X_{d,n}\circ B_{d,n}\right)\left(\tilde X_{d,n}\circ B_{d,n}\right)^\ast\right)\right\Arrowvert_{S_\infty}>z\alpha \right)
\end{align*}
are summable over $d\in \N$ for any $z>\left(1+\sqrt{y}\right)^2$. By Markov's inequality and because of $\Arrowvert S \Arrowvert_{\infty}^{2l}\le \tr\left(S^{2l}\right)$ for any symmetric matrix $S$ and $l\in\N$, it is sufficient to show that for any sequence $(l_{d,n})$  of even integers with $$l_{d,n}/\log n\to\infty \ \ \text{and} \ \ \delta_{d,n}^{1/6}l_{d,n}/\log n\to 0,$$ we get
\begin{align*}
m_{d,n,l_{d,n}}=\E  \tr\left[\ind_{E_{d,n}}\left(  \frac{1}{n}A_{d,n}\circ\left(\left(\tilde X_{d,n}\circ B_{d,n}\right)\left(\tilde X_{d,n}\circ B_{d,n}\right)^\ast\right)  \right)^{l_{d,n}}\right]\le (\alpha\eta)^{l_{d,n}}, 
\end{align*}
where $\left(1+\sqrt{y}\right)^2<\eta<z$ is an absolute constant and $E_{d,n}$ is the event
 $$E_{d,n}=\Big\{\max_{i,j} |A_{ij,d,n}|\Big(\max_{i,k}B_{ik,d,n}^2\Big)< \alpha\Big\}.$$
We have by independence of $\tilde X_{d,n}$ and $(A_{d,n},B_{d,n})$,
\begin{align*}
m_{d,n,l_{d,n}}&= n^{-l_{d,n}}\sum_{i_1,\dots,i_{l_{d,n}}=1}^d\sum_{k_1,\dots,k_{l_{d,n}}=1}^n\E\Big[ \ind_{E_{d,n}}A_{i_1i_2}A_{i_2i_3}\cdots A_{i_{l_{d,n}-1}i_{l_{d,n}}}A_{i_{l_{d,n}}i_1}\\
&\hspace{5cm}\times B_{i_1k_1}B_{i_2k_1}\cdots B_{i_{l_{d,n}}k_{l_{d,n}}}B_{i_{1}k_{l_{d,n}}}\Big]\\
&\hspace{4cm}\times \E\Big[ \tilde X_{i_1k_1}\tilde X_{i_2k_1}\cdots \tilde X_{i_{l_{d,n}}k_{l_{d,n}}}\tilde X_{i_{1}k_{l_{d,n}}} \Big]\\
&\le \alpha^{l_{d,n}}n^{-l_{d,n}}\sum_{i_1,\dots,i_{l_{d,n}}=1}^d\sum_{k_1,\dots,k_{l_{d,n}}=1}^n\left|\E\Big[ \tilde X_{i_1k_1}\tilde X_{i_2k_1}\cdots \tilde X_{i_{l_{d,n}}k_{l_{d,n}}}\tilde X_{i_{1}k_{l_{d,n}}} \Big]\right|\\
&\le \alpha^{l_{d,n}}\eta^{l_{d,n}},
\end{align*}
for $d$ sufficiently large in which case the inequality 
\begin{align*}
n^{-l_{d,n}}\sum_{i_1,\dots,i_{l_{d,n}}=1}^d\sum_{k_1,\dots,k_{l_{d,n}}=1}^n\left|\E\Big[ \tilde X_{i_1k_1}\tilde X_{i_2k_1}\cdots \tilde X_{i_{l_{d,n}}k_{l_{d,n}}}\tilde X_{i_{1}k_{l_{d,n}}} \Big]\right|\le \eta^{l_{d,n}}
\end{align*}
has been shown in the proof of Theorem 3.1 in \cite{Bai1988}.

\medskip
\section{Auxiliary results}\label{section: appendix}
\smallskip
\begin{lemma}\label{lemma4}[Lemma 4 in \cite{couillet2011}]
Let $A\in\C^{d\times d}$, $\tau\in\C$ and $r\in\R^d$ such that $A$ and $A+\tau rr^\ast$ are invertable. Then
\begin{equation}
r^\ast\left(A+\tau rr^\ast\right)^{-1}\ =\ \frac{1}{1+\tau r^\ast A^{-1}r}r^\ast A^{-1}.
\end{equation}
\end{lemma}

\begin{lemma}\label{lemma2.6}[Lemma 2.6 in \cite{Silverstein1995b}]
Let $z\in\C^+$, $A,B\in\C^{d\times d}$, $B$ Hermitian, $\tau\in\R$ and $q\in \C^d$. Then
\begin{equation}
\left\arrowvert \tr\left[\left(\left(B-zI_{d\times d}\right)^{-1}-\left(B+\tau qq^\ast-zI_{d\times d}\right)^{-1}\right)A\right]\right\arrowvert\ \leq\ \frac{\Arrowvert A\Arrowvert_{S_{\infty}}}{\Im z}. 
\end{equation}  
\end{lemma}

\begin{lemma}\label{lemma8}[Lemma 8 in \cite{couillet2011}]
Let $C=A+iB+ivI_{d\times d}$, with $A,B\in\R^{d\times d}$ symmetric and $B$ positive semidefinite, $v>0$. Then
\begin{equation}
\left\Arrowvert C^{-1}\right\Arrowvert_{S_{\infty}}\leq v^{-1}.
\end{equation}
\end{lemma}

%


\begin{lemma}\label{lemma: moments}
Let $Z=(Z_1,...,Z_d)\in\R^d$ be a centered random vector with components bounded in absolute value by some constant $c>0$. Then for any $p\ge 1$,
\begin{align}
\E \left|\Arrowvert Z \Arrowvert_2^2-\E\Arrowvert Z \Arrowvert_2^2 \right|^{p}&\le C^p p^{p/2} d^{p/2},\\
\E \Arrowvert Z \Arrowvert_2^{2p}&\le C^p p^{p/2}d^p,
\end{align}
where the constant $C>0$ depends on $c$ only. 
\end{lemma}
\begin{proof}
The lemma is an easy consequence of Lemma 5.9 of \cite{Vershynin2011} together with the Definition 5.7 of the subgaussian norm  of \cite{Vershynin2011}, since
\begin{align*}
\left\Arrowvert\frac{1}{d}\big(\Arrowvert Z \Arrowvert_2^2- \E \Arrowvert Z \Arrowvert_2^2\big) \right\Arrowvert_{\psi_2}^2 &\le  \frac{\Delta}{d^2}\sum_{i=1}^d \Arrowvert Z_i^2-\E Z_i^2 \Arrowvert_{\psi_2}^2\\
&\le \frac{8\Delta }{d}c^4,
\end{align*}
where $\Delta$ corresponds to the absolute constant of Lemma 5.9 of \cite{Vershynin2011}, and
\begin{align*}
\left\Arrowvert\frac{1}{d}\Arrowvert Z \Arrowvert_2^2\right\Arrowvert_{\psi_2}^2&=\left\Arrowvert \frac{1}{d}\E \Arrowvert Z\Arrowvert_2^2+\frac{1}{d}\left(\Arrowvert Z \Arrowvert_2^2- \E \Arrowvert Z \Arrowvert_2^2\right)  \right\Arrowvert_{\psi_2}^2\\
&\le 2\left(\left\Arrowvert  \frac{1}{d}\E \Arrowvert Z\Arrowvert_2^2\right\Arrowvert_{\psi_2}^2+\left\Arrowvert\frac{1}{d}\big(\Arrowvert Z \Arrowvert_2^2- \E \Arrowvert Z \Arrowvert_2^2\big)  \right\Arrowvert_{\psi_2}^2\right)\\
&\le \left(2+\frac{16\Delta }{d}\right)c^4.
\end{align*}
\end{proof}
\begin{lemma}\label{lemma: eigenvalue}
Let $d/n<c_1$ and $Z_1,...,Z_n\in\R^d$ be a sample of i.i.d. random vectors with centered and independent components of variance $1$ and bounded in absolute value by some constant $c_2>0$. Denote the largest eigenvalue of the matrix $n^{-1}\sum_{k}Z_kZ_k^\ast $ by $\lambda_1$. Then for any $p\ge 1$,
\begin{align*}
\E \lambda_1^p\le C,
\end{align*}
where $C$ depends on $c_1,c_2$ and $p$ only.
\end{lemma}
\begin{proof}
Since $$\frac{1}{n}\sum_{k=1}^nZ_kZ_k^\ast =\frac{1}{n}ZZ^\ast ,$$ where the $k$-th column of the matrix $Z\in\R^{d\times n}$ is given by $Z_k$, $\lambda_1=s_1^2$ with $s_1$ the largest singular value of $n^{-1/2}Z$. Dividing the right-hand side of inequality (5.22) of \cite{Vershynin2011} by $\sqrt{n}$ yields
\begin{align*}
s_{1}\le \sqrt{c_1}+\Delta_1+\frac{t}{\sqrt n}
\end{align*}
with probability at least $1-2\exp(-\Delta_2t^2)$ for some constant $\Delta_1,\Delta_2>0$ depending on $c_2$ only. Therefore,
\begin{align*}
\E \lambda_1^p&=\E s_1^{2p}\\
&=\int_0^\infty x^{2p} \P(s_1>x) \d x\\
&\le (\sqrt{c_1}+\Delta_1)^{2p}+2\int_{\sqrt{c_1}+\Delta_1}^\infty x^{2p}\exp\left(-\Delta_2n(x-(\sqrt{c_1}+\Delta_1))^2\right) \d x\\
&\le (\sqrt{c_1}+\Delta_1)^{2p}+2\int_{0}^\infty (x+\sqrt{c_1}+\Delta_1)^{2p}\exp\left(-\Delta_2nx^2\right) \d x\\
&\le C,
\end{align*}
where $C$ can be chosen independently of $n$.
\end{proof}

\begin{lemma}\label{lemma7}
Let $U_1,...,U_d$ iid random $\C$-valued random variables with $\E U_i=0$, $\E\arrowvert U_i\arrowvert^2=1$, $\arrowvert U_i\arrowvert\leq C$ for some constant $C>0$ and $A\in\C^{d\times d}$. Denote $U=(U_1,...,U_d)^\ast $. Then
$$
\E\left\arrowvert  U^\ast AU-\tr A\right\arrowvert^6\ \leq\ c\Arrowvert A\Arrowvert_{S_{\infty}}^6d^3C^{12}
$$
with a constant $c>0$ which does not depend on $d$, $A$ and the distribution of $U_i$.
\end{lemma}
\begin{proof}
The proof follows the lines of Lemma 3.1 in \cite{Silverstein1995b} by replacing the logarithmic bound on the entries of $U$ with $C$.
\end{proof}

\begin{lemma}\label{lemma: log d}
For $d\in \N$ and $n=n_d\in \N$ with $\limsup_{d}d/n\leq c_1<\infty$ let $X_{1,d},\dots,X_{n,d}$ be i.i.d. $d$-dimensional, centered random vectors with variance $1$ such that 
$$\limsup_{d\to\infty}\max_{i=1,\dots,d}\max_{k=1,\dots,n} |X_{i,k,d}|\le c_2$$ almost surely and $R_d\in \R^{d\times d}$ be a  positive definite diagonal matrix with 
$$\limsup_{d\to\infty}\max_{i=1,\dots,d} |R_{i,i,d}|\le c_3.$$ Then,
\begin{equation}\label{c prime}
\underset{d\rightarrow\infty}{\lim\sup}\,\lambda_{\max}\left(\frac{1}{n}\sum_{k=1}^nR_d^{1/2}X_{k,d}X_{k,d}^\ast R_d^{1/2}\right)\ < c \ \ a.s.
\end{equation}
for some constant $c>0$ depending on $c_1,c_2$ and $c_3$ only.
\end{lemma}
\begin{proof}
Since the random variables are uniformly bounded which implies uniform subgaussian tails, Theorem 5.39 of \cite{Vershynin2011} applies. The particular choice $t=\log d$ yields 
\begin{align*}
\lambda_{\max}\left(\frac{1}{n}\sum_{k=1}^nR_d^{1/2}X_{k,d}X_{k,d}^\ast R_d^{1/2}\right)\le \frac{d}{n}+C+\frac{(\log d)^2}{n}
\end{align*}
with probability at least $1-2\exp(-C'(\log d)^2)$ for two positive constants $C,C' $ which depend only on $c_1$ and $c_2$. Hence, the claim follows by the Lemma of Borel-Cantelli. 
\end{proof}
\begin{theorem}[Theorem A.43 \cite{Bai2010}]\label{theorem: A43}
Let $A$ and $B$ be two $d\times d$ Hermitian matrices. Then,
\begin{equation}
d_K\left(\mu^A,\mu^B\right)\le \frac{1}{d}\rank(A-B),
\end{equation}
where $\mu^A$ and $\mu^B$ denote the spectral distributions of $A$ and $B$, respectively. 
\end{theorem}

\begin{theorem}[Corollary A.41 from \cite{Bai2010}]\label{theorem: inequality Levy distance trace}
Let $A$ and $B$ be two $d\times d$ Hermitian matrices with spectral distribution $\mu^A$ and $\mu^B$. Then,
\begin{equation}
d_L^3\left(\mu^A,\mu^B\right)\le \frac{1}{d}\tr\big((A-B)(A-B)^\ast \big).
\end{equation}
\end{theorem}

\begin{theorem}[Theorem A. 38 \cite{Bai2010}]\label{theorem: A38}
Let $\lambda_1,\dots,\lambda_d$ and $\delta_1,\dots,\delta_d$ be two families of real numbers and their empirical distributions be denoted by $\mu$ and $\bar \mu$. Then, for any $\alpha>0$, we have 
\begin{equation}
d_L^{\alpha+1}(\mu,\bar \mu)\le \min_{\pi}\frac{1}{d}\sum_{k=1}^d |\lambda_k-\delta_{\pi(k)}|^\alpha,
\end{equation}
where the minimum is running over all permutations $\pi$ on $\{1,\dots, d\}$.
\end{theorem}
The next lemma and its proof are essentially taken from \cite{Krishnapur2012}, Lemma 34. Since the necessary dependence of (in his notation) $\delta$ on $y$ is neither mentioned in his statement nor its proof, we include a proof for completeness.
\begin{lemma}
\label{lemma: Levy distance}
Let $\mu$ and $\nu$ be two probability measures on the real line and $m_{\mu}$ and $m_{\nu}$ their Stieltjes transforms. Then for any $v>0$ we have
$$
d_L(\mu,\nu)\le 2\sqrt{\frac{v}{\pi}}+\frac{1}{2\pi}\int \left|\Im\left(m_{\mu}(u+iv)\right)-\Im\left(m_\nu(u+iv)\right)\right|\d u.
$$
\end{lemma}
\begin{proof}
Let $C_v$ denote the Cauchy distribution with scale parameter $v>0$. Recall that its Lebesgue density $f_v$ is given by
$$
f_v(x)=\frac{1}{\pi}\frac{v}{v^2+x^2}, \ \ x\in\R.
$$
By the triangle inequality,
\begin{equation}\label{Dreieck}
d_L(\mu,\nu)\leq d_L\left(\mu,\mu\star C_v\right) + d_L\left(\mu\star C_v,\nu\star C_v\right)+ d_L\left(\nu,\nu\star C_v\right).
\end{equation}
Now observe that for $\eta=\mu,\nu$ and any $z=u+iv\in\C^+$,
$$
-\frac{1}{\pi}\Im\left(m_\eta(u+iv)\right)=\int\frac{1}{\pi}\frac{v}{(u-\lambda)^2+v^2}\d \eta(\lambda)=f_{\eta\star C_v}\left(u\right), 
$$
where $f_{\eta\star C_v}$ is the Lebesgue density of the convolution $\eta\star C_v$. Therefore,
\begin{align} 
d_L\left(\mu\star C_v,\nu\star C_v\right)&\leq  d_K\left(\mu\star C_v,\nu\star C_v\right)\nonumber\\
 &\leq\frac{1}{2}\int\left\arrowvert f_{\mu\star C_v}(u)-f_{\nu\star C_v}(u)\right\arrowvert \d u\nonumber\\
 &= \frac{1}{2\pi}\int \left|\Im\left(m_{\mu}(u+iv)\right)-\Im\left(m_\nu(u+iv)\right)\right|\d u.  \label{teil1!}
 \end{align}
As concerns $d_L\left(\eta,\eta\star C_v\right)$, let $X\sim\eta$ and $Z\sim C_1 $ be two independent random variables on a common probability space, whence $X+vZ\sim \eta\star C_v$ for any $v>0$. Using the elementary tail inequalities
$$
\P(Z<-t)=\P(Z>t)\leq\int_{t}^{\infty}\frac{1}{\pi t^2}\d t= \frac{1}{\pi t},
$$ 
we obtain for any $\delta>0$ and $x\in\R$,
 \begin{align*}
\P\left(X\leq x-\delta\right)\leq \P\left(X+vZ\leq x \right) +\P\left(Z>\frac{\delta}{v}\right)\leq  \P\left(X+vZ\leq x\right) + \frac{1}{\pi}\frac{v}{\delta}.
\end{align*}
 That is,
 \begin{equation}\label{wochenende1}
  \P\left(X\leq x-\delta\right)-\delta\leq \P\left(X+vZ\leq x\right)
   \end{equation}
whenever $\delta\geq \sqrt{v/\pi}$, in which case we also have
 \begin{align}\label{wochenende2}
\P\left(X+vZ\leq x\right)\leq  \P\left(X\leq x+\delta\right) +\P\left(Z<-\frac{\delta}{v}\right)\leq  \P\left(X\leq x+\delta\right) + \delta.
 \end{align}
 \eqref{wochenende1} and \eqref{wochenende2} imply
 \begin{equation}\label{levy!}
d_L\left(\eta,\eta\star C_v\right)\leq \sqrt{\frac{v}{\pi}},\ \ \eta=\mu,\nu.
 \end{equation}
 Plugging \eqref{levy!} and \eqref{teil1!} into \eqref{Dreieck} yields the claim.
 \end{proof}
 
 \begin{lemma}\label{7.7}
Let $\mu$, $\nu$ be two probability measures on the real line and $m_{\mu}$, $m_{\nu}$ the corresponding Stieltjes transforms. Then for any $z\in\C^+$,
\begin{equation}\label{eq: levy 2}
\left\arrowvert m_{\mu}(z)-m_{\nu}(z)\right\arrowvert\ \leq\ 2\frac{d_{BL}(\mu,\nu)}{(\Im z)^2\wedge \Im z}.
\end{equation}
\end{lemma}
\begin{proof}
Note that
\begin{align*}
\left\arrowvert \frac{1}{\lambda-z}-\frac{1}{\lambda'-z}\right\arrowvert &= \frac{\arrowvert \lambda-\lambda'\arrowvert}{\arrowvert(\lambda-z)(\lambda'-z)\arrowvert}\leq\ \frac{\arrowvert \lambda-\lambda'\arrowvert}{(\Im z)^2},
\end{align*}
i.e. 
$$
\lambda\mapsto  \Re\left(\frac{(\Im z)^2\wedge \Im z}{\lambda-z}\right)\ \ \ \text{and}\ \ \ \lambda\mapsto  \Im\left(\frac{(\Im z)^2\wedge \Im z}{\lambda-z}\right)
$$
are bounded by $1$ in absolute value and $1$-Lipschitz. This proves \eqref{eq: levy 2}.
\end{proof}

 \begin{lemma}\label{7.8}
 Let $(\mu_n)_{n\in\N}$ and $(\nu_n)_{n\in N}$ be two sequences of probability measures on the Borel $\sigma$-algebra on $\R$. Assume that $(\mu_n)_{n\in\N}$ is tight. Then
\begin{equation}\label{we3}
d_L(\mu_n,\nu_n)\rightarrow 0
\ \ \Leftrightarrow\ \ d_{BL}(\mu_n,\nu_n)\rightarrow 0.
\end{equation}
Moreover, tightness of $(\mu_n)_{n\in\N}$ and \eqref{we3} imply weak convergence $\mu_n-\nu_n\Rightarrow 0$ on the space of finite signed measures on $\R$.
 \end{lemma}
 
 \begin{proof}
 As concerns the equivalence relation, we need only to verify that 
 \begin{equation}\label{Folgerung}
 d_L(\mu_n,\nu_n)\rightarrow 0\ \Rightarrow\  d_{BL}(\mu_n,\nu_n)\rightarrow 0,
 \end{equation}
 because 
 $d_L^2\leq d_{BL}$ (see, e.g. \cite{Huber1974}).  Assume that  $d_L(\mu_n,\nu_n)\rightarrow 0$. Tightness of $(\mu_n)_n$ implies that any subsequence $(\mu_{n_k})_k$ possesses a subsubsequence $(\mu_{n_{k_l}})_l$ which converges weakly to a limiting probability measure $\mu$, say. Consequently, as both, $d_{BL}$ and $d_L$  metrize weak convergence on the space of probability measures on $\R$,
 \begin{equation}
 d_L(\mu_{n_{k_l}},\mu)\rightarrow 0\ \ \Leftrightarrow\ \ d_{BL}(\mu_{n_{k_l}},\mu)\rightarrow 0.
  \end{equation} 
 By the triangle inequality,
 $$
 d_L(\nu_{n_{k_l}},\mu)\leq d_L(\mu_{n_{k_l}},\mu) + d_L(\mu_{n_{k_l}},\nu_{n_{k_l}})\rightarrow 0,
 $$
which in turn is equivalent to $ d_{BL}(\nu_{n_{k_l}},\mu) \rightarrow 0$. Again by the triangle inequality, $ d_{BL}(\mu_{n_{kl}},\nu_{n_{k_l}}) \rightarrow 0$. 
This proves \eqref{Folgerung} and therefore the equivalence relation \eqref{we3}. \\
As concerns the second statement, it is sufficient to show that any subsequence $(n_k)_k$ possesses a subsubsequence $(n_{k_l})_l$ with $\mu_{n_{k_l}}-\nu_{n_{k_l}}\Rightarrow 0$. But this follows immediately from the above arguments, because for any subsequence $(n_k)_k$, there exist a  subsubsequence $(n_{k_l})_l$ and a measure $\mu$ such that both, $\mu_{n_{k_l}}\Rightarrow\mu$ and  $\nu_{n_{k_l}}\Rightarrow\mu$, hence   $\mu_{n_{k_l}}-\nu_{n_{k_l}}\Rightarrow 0$. \end{proof}
\smallskip

\subsection*{Acknowledgements}We are grateful to Jack Silverstein for kindly answering numerous questions and for pointing us to the reference \cite{couillet2011}.\bigskip

\bibliographystyle{imsart-nameyear}
\bibliography{referenceA}

\end{document}